\documentclass[11pt, reqno]{amsart}
\usepackage{inputenc}
\usepackage{amsmath}
\usepackage{amssymb}
\usepackage{color}
\usepackage{mathtools}
\usepackage{bbm}
\usepackage{calrsfs}
\DeclareMathAlphabet{\pazocal}{OMS}{zplm}{m}{n}
\usepackage{epsfig}
\usepackage[mathscr]{eucal}
\usepackage{enumerate}
\usepackage[shortlabels]{enumitem}
\usepackage{verbatim}

\usepackage{xcolor}
\usepackage{hyperref}
\usepackage[nameinlink]{cleveref}
\hypersetup{
	colorlinks,
	linkcolor={blue!50!black},
	citecolor={blue!50!black},
	urlcolor={blue!50!black}
}

\usepackage{blindtext}

\newtheorem{theorem}{Theorem}
\newtheorem{definition}{Definition}
\newtheorem{lemma}{Lemma}
\newtheorem{proposition}{Proposition}
\newtheorem{corollary}{Corollary}
\newtheorem{remark}{Remark}
\newtheorem{assumption}{Assumption}
\newtheorem{question}{Question}

\numberwithin{equation}{section}
\numberwithin{theorem}{section}
\numberwithin{lemma}{section}
\numberwithin{corollary}{section}
\numberwithin{remark}{section} \numberwithin{proposition}{section}
\numberwithin{definition}{section}
\numberwithin{assumption}{section}

\usepackage[margin=1in]{geometry}

\newcommand{\dd}{\mathrm{d}}
\newcommand{\supp}{\operatorname{supp}}

\newcommand{\loc}{\operatorname{loc}}

\newcommand{\Div}{\operatorname{div}}

\newcommand{\bv}{\operatorname{BV}}

\newcommand{\slog}{\operatorname{slog}}
\newcommand{\sexp}{\operatorname{sexp}}

\newcommand{\pseudo}[1]{{\left\vert\kern-0.25ex\left\vert\kern-0.25ex\left\vert #1 
		\right\vert\kern-0.25ex\right\vert\kern-0.25ex\right\vert}}

\usepackage[foot]{amsaddr}

\newcommand{\R}{\mathbb R}

\begin{document}

\title[DiPerna--Lions flows]{Quantitative Osgood regularity for DiPerna--Lions flows}

\author[H. Borrin\and J.F. Nariyoshi]{Henrique Borrin$^\dagger$ \and João Fernando Nariyoshi$^\ddagger$}
\address{$^\dagger$Faculdade de Filosofia, Ci\^{e}ncias e Letras de Ribeir\~{a}o Preto, USP-Universidade de São Paulo  Address: Avenida Bandeirantes, 3900. Ribeir\~{a}o Preto, SP, Brazil. Zip Code 14040-901. ORCID: 0000-0003-4670-4444.}
\email{henriqueborrin@usp.br (Corresponding author)}
\address{$^\ddagger$Instituto de Matemática e Estatística, USP-Universidade de São Paulo // Address: Rua do Matão, 1010. São Paulo, SP, Brazil. Zip code 05508090. ORCID: 0000-0001-6881-3305.}
\email{jfc@ime.usp.br}

\thispagestyle{empty}

\begin{abstract}
We study the spatial regularity of regular Lagrangian flows associated with vector fields in the DiPerna--Lions class \(\boldsymbol b\in L^1((0,T);W^{1,1}_{\loc}(\mathbb R^d)), \) under the standard growth and compressibility assumptions. For vector fields in \(L^1_tW^{1,p}_{\loc,x}\), with \(p>1\), the flow \(\boldsymbol X(t,\cdot)\) is known to satisfy a quantitative local Lipschitz estimate, which implies that it is Lipschitz continuous in the Lusin sense. We prove that, at the endpoint \(p=1\), this estimate admits an Osgood-type counterpart. More precisely, we construct an increasing function \(G\), with \(G(0+)=-\infty\), determined by the integrability of \(D\boldsymbol b\), such that \[ G\bigl(|\boldsymbol X(t,x)-\boldsymbol X(t,y)|\bigr) \leq G\bigl(|\boldsymbol X(s,x)-\boldsymbol X(s,y)|\bigr) + \int_s^t \bigl(k(\tau,x)+k(\tau,y)\bigr)\,\dd\tau, \] where \(k\) is locally integrable. As a consequence, the flow \(\boldsymbol X(t,\cdot)\) is uniformly continuous outside a set of arbitrarily small measure, with an explicit modulus of continuity determined by the integrability properties of \(D\boldsymbol b\). The resulting moduli include H\"older and log-Lipschitz regimes, as well as substantially weaker Osgood moduli. Our approach is based on a new family of weighted maximal operators associated with slowly varying functions in the sense of Karamata. We also provide examples showing that the resulting estimates are sharp in several respects and that the classical Lipschitz-type estimate may fail at the endpoint \(p=1\). Finally, we apply the flow estimates to transport equations, obtaining weighted logarithmic Sobolev regularity for transported scalars and corresponding lower bounds on functional and geometric mixing scales in the \(W^{1,1}\) setting. 

\bigskip

\bigskip

\end{abstract}

\maketitle

\section{Introduction}
\subsection{Motivation}

Let us consider the ordinary differential equation
\begin{equation}\label{flow}
	\begin{cases}
		\displaystyle\frac{\dd}{\dd t}\boldsymbol{X}(t,0,x)=\boldsymbol{b}(t,\boldsymbol{X}(t,0,x))\quad &\text{for } (t,x)\in (0,T)\times\mathbb{R}^d;\\
		\boldsymbol{X}(0,0,x)=x \quad&\text{for } x\in\mathbb{R}^d,
	\end{cases}
\end{equation}
where the vector field $\boldsymbol{b} : (0,T)\times \mathbb R^d \to \mathbb R^d$ is assumed to satisfy the classical hypotheses of DiPerna--Lions \cite{dipernalions}, namely:
\begin{itemize}
    \item Sobolev regularity:
    \begin{equation}
            \boldsymbol{b} \in L^1((0,T); W_{\loc}^{1,1}(\mathbb R^d)); \label{regularity}
    \end{equation}
    \item Growth condition:
    \begin{equation}\label{growth}
	\begin{split}
		\frac{\boldsymbol{b}(t,x)}{1+|x|}=\boldsymbol{b}^1(t,x)+\boldsymbol{b}^2(t,x),
	\end{split}
\end{equation}
where $\boldsymbol{b}^1\in L^1((0,T)\times\mathbb{R}^d)$ and $\boldsymbol{b}^2\in L^1((0,T);L^\infty(\mathbb{R}^d))$; and
\item Compressibility condition:
\begin{equation}
    (\Div \boldsymbol b)_- \in L^1((0,T); L^\infty(\mathbb R^d)). \label{compressibility}
\end{equation}
\end{itemize}
Under these assumptions, DiPerna and Lions proved in their seminal work \cite{dipernalions} that \eqref{flow} admits a unique regular Lagrangian flow; see also \cite{ambrosio,bouchutcrippa,partiallyregular,miotsingularset,nguyen} and the references therein. Thus, although neither the vector field nor its flow need be classically regular, the ordinary differential equation is well posed for almost every initial condition.

A natural question, and the main subject of the present paper, is whether the map
\(
x\mapsto \boldsymbol X(t,0,x)
\)
retains any quantitative continuity with respect to the initial datum. Besides strengthening the Lagrangian formulation of the DiPerna--Lions theory, such an estimate yields regularity information for solutions of the associated transport equation
\begin{equation}
    \partial_t u+\boldsymbol b\cdot Du=0.
    \label{transport}
\end{equation}

The problem of quantitative regularity for regular Lagrangian flows was first addressed by Crippa and De Lellis in their groundbreaking paper \cite{crippadelellis}. To simplify the discussion, we temporarily assume, as in \cite{crippadelellis}, that
\begin{equation}
    \boldsymbol b\in L^\infty((0,T)\times\mathbb R^d).
    \label{linfty}
\end{equation}
Building on a previous result of Ambrosio--Lecumberry--Maniglia \cite{Ambrosio2005}, Crippa and De Lellis proved that, under the stronger assumption
\begin{equation}
    \boldsymbol b\in L^1\bigl((0,T);W^{1,p}_{\loc}(\mathbb R^d)\bigr),
    \qquad p>1,
    \label{assumptionp1}
\end{equation}
the flow is locally Lipschitz continuous in the Lusin sense. More precisely, given any ball \(B_R\subset\mathbb R^d\) and any \(\epsilon>0\), there exists a set \(\Omega_\epsilon\subset B_R\) such that
\(\mathcal L^d(B_R\setminus\Omega_\epsilon)<\epsilon\) and
\(\boldsymbol X(t,0,\cdot)|_{\Omega_\epsilon}\) is Lipschitz continuous. We emphasize that, in general, one must remove a small set of ``bad'' initial data in order to obtain such a regularity statement, since the full map
\(\boldsymbol X(t,0,\cdot)|_{B_R}\) may be highly irregular; see
\cite{jabin2,mazzucato1,mazzucato2}.

A considerably simpler proof of the Crippa--De Lellis result was later given by Bru\`e and Nguyen \cite{bruenguyen}. Their argument is based on the stronger pointwise estimate
\begin{equation}
    |\boldsymbol X(t,0,x)-\boldsymbol X(t,0,y)|
    \leq
    \exp\left(
        \int_s^t\bigl(k(\tau,x)+k(\tau,y)\bigr)\,\dd\tau
    \right)
    |\boldsymbol X(s,0,x)-\boldsymbol X(s,0,y)|,
    \label{estimate}
\end{equation}
which holds for almost every \(x,y\in\mathbb R^d\), every
\(0\leq s\leq t\leq T\), and a suitable function
\(k\in L^1((0,T);L^p_{\loc}(\mathbb R^d))\). The Lusin--Lipschitz regularity then follows immediately by applying Chebyshev's inequality to
\(x\mapsto\int_0^T k(\tau,x)\,\dd\tau\) and taking \(s=0\).

At the endpoint \(p=1\), the proof of \eqref{estimate} breaks down. Indeed, the argument relies on the strong \(L^p\)-boundedness of the Hardy--Littlewood maximal operator, whereas at \(p=1\) only a weak-type estimate is available. Although the existence and uniqueness of the regular Lagrangian flow remain valid in this regime through a modified argument due to Jabin \cite{jabin}, it is not clear how to obtain an estimate of the form \eqref{estimate} for a general vector field in the DiPerna--Lions class.

Our main result provides a weaker, but remarkably robust, substitute for \eqref{estimate} under the endpoint assumption \eqref{regularity}. We prove that there exist an increasing continuous function
\(G:(0,\infty)\to\mathbb R\), with \(G(0+)=-\infty\), and a locally integrable function \(k\) such that
\begin{equation}
    G\bigl(|\boldsymbol X(t,0,x)-\boldsymbol X(t,0,y)|\bigr)
    \leq
    G\bigl(|\boldsymbol X(s,0,x)-\boldsymbol X(s,0,y)|\bigr)
    +
    \int_s^t\bigl(k(\tau,x)+k(\tau,y)\bigr)\,\dd\tau.
    \label{estimatenew}
\end{equation}
Both \(G\) and \(k\) are constructed from the integrability properties of \(D\boldsymbol b\). When \(G(r)=\log r\), estimate \eqref{estimatenew} reduces to the classical Lipschitz-type bound \eqref{estimate}.

As an immediate consequence, for every ball \(B_R\) and every
\(\epsilon>0\), there exists a set
\(\Omega_{\epsilon,R}\subset B_R\), with
\(\mathcal L^d(B_R\setminus\Omega_{\epsilon,R})<\epsilon\), such that
\(\boldsymbol X(t,0,\cdot)\) is uniformly continuous on
\(\Omega_{\epsilon,R}\), uniformly with respect to \(t\). Moreover, the modulus of continuity is explicit and has the form
\(\omega^\epsilon(r)=G^{-1}(G(r)+C_\epsilon)\). Depending on the integrability of \(D\boldsymbol b\), the resulting moduli range from H\"older and log-Lipschitz behavior to substantially weaker regimes.

The proof is based on a new family of weighted maximal operators associated with slowly varying functions in the sense of Karamata. We also provide several examples that clarify the sharpness and the range of validity of the resulting estimates. Finally, we apply the flow estimate to transported scalars and derive lower bounds on geometric and functional mixing scales.

From a conceptual perspective, our results suggest that the endpoint \(p=1\) is governed by an Osgood-type mechanism, in contrast with the Cauchy--Lipschitz structure underlying the range \(p>1\).

\subsection{Statement of the main results}

Without delving too deeply into the technical details, our main theorem may be stated as follows. We shall use the standard notion of regular Lagrangian flow in the renormalized sense, which will be recalled precisely in \Cref{rlf}. For simplicity, we formulate it only for the initial time $s=0$, although a straightforward adaptation yields the corresponding statement for arbitrary $s\in[0,T)$. 

\begin{theorem}\label{main}
Let $\boldsymbol b(t,x)$ be a vector field satisfying \eqref{regularity}, \eqref{growth}, and \eqref{compressibility}, and let $\boldsymbol X(t,s,\cdot)$ denote the associated regular Lagrangian flow of \eqref{flow}. Let also $R>0$, and set $B_R=\{x\in\mathbb R^d:\ |x|\leq R\}$.

Then, for every $\epsilon>0$, there exist
\begin{enumerate}
    \item a measurable set $U_{\epsilon,R}\subset B_R$ such that $\mathcal{L}^d(B_R\setminus U_{\epsilon,R})<\epsilon/2$,
    \item a number \(\lambda=\lambda(R,\epsilon)>0\),
    \item a continuous increasing function $G=G_\lambda:(0,\infty)\to\mathbb R$ such that $G(0+)=-\infty$, and
    \item a function $k\in L^1((0,T)\times U_{\epsilon, R})$
\end{enumerate}
such that
\begin{equation}\label{inequalitycrude}
    G\bigl(|\boldsymbol{X}(t,0,x)-\boldsymbol{X}(t,0,y)|\bigr)
    \leq
    G\bigl(|\boldsymbol{X}(s,0,x)-\boldsymbol{X}(s,0,y)|\bigr)
    +
    \int_s^t\bigl(k(\tau,x)+k(\tau,y)\bigr)\,\dd \tau
\end{equation}
for all $x,y\in U_{\epsilon,R}$ and all $0\leq s\leq t\leq T$. The functions $G(r)$ and $k(\tau,x)$ depend only on the modulus of integrability of $D\boldsymbol b(t,x)$ on the cylinder $(0,T)\times B_{R+2\lambda}$.

If, in addition, $\boldsymbol b\in L^1((0,T); L^\infty( \mathbb R^d))$, then $\lambda$ may be taken to be $\int_0^T\Vert \boldsymbol b(t,\cdot)\Vert_{L^\infty(\mathbb R^d)}\,\dd t$, and $U_{\epsilon,R}$ may be taken to be $B_R$.

If \eqref{regularity} is strengthened to $\boldsymbol b \in L^1((0,T); W^{1,1}(\mathbb R^d))$, then $\lambda$ may also be taken equal to $\infty$, and $U_{\epsilon,R}$ may be taken to be $B_R$.
\end{theorem}

\begin{corollary}[Lusin continuity]\label{lusin}
Under the hypotheses of \Cref{main}, for every $\epsilon>0$ there exist a set $\Omega_{\epsilon,R}\subset B_R$ such that $\mathcal L^d(B_R\setminus \Omega_{\epsilon,R})<\epsilon$ and an increasing continuous function $\omega^\epsilon:\mathbb{R}_+\rightarrow \mathbb{R}_+$ such that $\omega^\epsilon(0^+)=0$ and, for all $t\in[0,T]$,
\begin{equation}\label{lipschitz}
	|\boldsymbol{X}(t,0,x)-\boldsymbol{X}(t,0,y)|\leq \omega^{\epsilon}(|x-y|)\quad \text{for all } x,y\in \Omega_{\epsilon,R}.
\end{equation}
In other words, $\boldsymbol{X}(t,0,\cdot)$ is uniformly continuous on $\Omega_{\epsilon,R}$.
\end{corollary}

The significance of \Cref{main} lies in the fact that both \(G\) and \(k\) are determined by the modulus of integrability of \(D\boldsymbol b\). Once this dependence is made explicit, one obtains a quantitative modulus of continuity of the form
\[
\omega^\epsilon(r)=G^{-1}\bigl(G(r)+C/\epsilon\bigr).
\]
This gives rise to a variety of Lusin continuity regimes, ranging from familiar H\"older and log-Lipschitz behavior to substantially weaker Osgood moduli. The following corollary records several representative examples.

\begin{corollary}\label{examples}
Under the hypotheses of \Cref{main}, and with the notation of \Cref{lusin}, let $\epsilon>0$. Then, for constants $C_\epsilon\to\infty$ and $\alpha_\epsilon\to0$ as $\epsilon\to0$, depending on $\boldsymbol b(t,x)$, the following assertions hold.
\begin{enumerate}
    \item If $|D\boldsymbol{b}| \log_+ (|D \boldsymbol{b}|)^b \in L^1_{\loc}$ for some $b \in (0,1)$, then $\boldsymbol{X}(t,0,\cdot)$ is Hölder continuous on $\Omega_{\epsilon,R}$ for any exponent $0<\alpha<1$. More precisely, if $D\boldsymbol{b} \in L^1((0,T); L \log^{b} L_{\loc}(\mathbb R^d))$, then $\omega_1^\epsilon(z) \leq C_{\alpha,\epsilon} |z|^\alpha$.

    \item If $|D\boldsymbol{b}| \log_+ (|D \boldsymbol{b}|)/(1+\log_+^{(n)} |D\boldsymbol{b}|)\in L^1_{\loc}$, where $n\geq2$ and $\log^{(n)}$ denotes the $n$-fold iterated logarithm, then $\boldsymbol{X}(t,0,\cdot)$ is $\log^{(n-1)}$-Lipschitz on $\Omega_{\epsilon,R}$. More precisely, if $D\boldsymbol{b} \in L^1((0,T); L \frac{\log}{\log^{(n)}} L_{\loc}(\mathbb R^d))$, then $\omega_2^\epsilon(z) \leq C_\epsilon |z| \log^{(n-2)}(|\log z|)^{1/\alpha_\epsilon}$.

    \item If $|D\boldsymbol{b}| \log_+ \log_+ (|D \boldsymbol{b}|)\in L^1_{\loc}$, then $\boldsymbol{X}(t,0,\cdot)$ is Hölder continuous on $\Omega_{\epsilon,R}$, but the Hölder exponent degenerates as $\epsilon\to0$. More precisely, if  $D \boldsymbol{b} \in L^1((0,T); L \log\log_{\loc} L(\mathbb R^d))$, then $\omega_3^\epsilon(z) \leq C_\epsilon |z|^{\alpha_\epsilon}$ as $z\to0$.

    \item If $|D\boldsymbol{b}| \log_+^{(n)} (|D \boldsymbol{b}|)\in L^1_{\loc}$ for some $n\geq 2$, or more precisely $D \boldsymbol{b} \in L^1((0,T); L \log^{(n)}_{\loc} L(\mathbb R^d))$, then
    \[
    \omega_4^\epsilon(z) \leq \bigl(\exp^{(n-1)}(\alpha_\epsilon \log^{(n-1)}(z^{-d}))\bigr)^{-1/d} \text{ as $z\to0$},
    \]
    where $\exp^{(n-1)}$ denotes the inverse of $\log^{(n-1)}$.
\end{enumerate}
\end{corollary}

The estimates in \Cref{main} and \Cref{lusin} provide upper bounds for the separation of trajectories and for the corresponding moduli of continuity. It is therefore natural to ask whether the scale determined by \(G\) can actually be attained, or whether it is merely an artifact of the argument. The following example shows that the predicted continuity class is, in general, genuine.

\begin{proposition} \label{exampleprop}
    Let $T>0$ be arbitrarily large. There exists an autonomous bounded vector field \(\boldsymbol{b}\) satisfying \eqref{regularity}, \eqref{growth}, and \eqref{compressibility} such that, for every \(R>0\), for $\epsilon>0$ sufficiently small, and for the function \(G\) given by \Cref{main} one has
    \[
    G\bigl(|\boldsymbol{X}(T,0,x)-\boldsymbol{X}(T,0,y)|\bigr)
    =
    G(|x-y|)+T
    \]
    for some \(x,y\in \Omega_{\epsilon,R}\). Consequently, the modulus of continuity of \(\boldsymbol{X}(t,0,\cdot)\) on \(\Omega_{\epsilon,R}\) is in the class predicted by \Cref{lusin}.
\end{proposition}


Two further counterexamples highlight complementary sharpness features of our result. In \Cref{counter}, we show that the stronger Lipschitz-type estimate \eqref{estimate} may fail at the endpoint \(p=1\), even for a bounded autonomous vector field with integrable derivative. In \Cref{counter2}, we prove that the Osgood estimate \eqref{estimatenew} does not extend, in the same form, from the Sobolev class \(W^{1,1}_{\loc}\) to the \(BV_{\loc}\) setting. Thus, the examples in Section~\ref{section3} establish three distinct facts: the continuity class predicted by our theory can be attained, the classical \(p>1\) estimate may fail at the endpoint, and the Sobolev and \(BV\) regimes exhibit genuinely different behavior.

We also derive applications to the regularity of solutions of the transport equation and to quantitative mixing. Since their formulation requires additional notation and background, these results are presented separately in Sections~\ref{SecReg} and~\ref{SecMix}.

\subsection{A glimpse of the proof of \texorpdfstring{\Cref{main}}{}}

We now briefly explain the main ideas in the proof of \Cref{main}. For simplicity, we again assume \eqref{linfty}.

We begin by recalling the derivation of \eqref{estimate}, which closely parallels the classical Cauchy--Lipschitz theory. In the elementary case
\(\boldsymbol b\in L^1((0,T);W^{1,\infty}_{\loc}(\mathbb R^d))\), local Lipschitz continuity of the flow follows immediately from
\begin{align}
    \frac{\dd}{\dd t}
    |\boldsymbol X(t,0,x)-\boldsymbol X(t,0,y)|
    &\leq
    |\boldsymbol b(t,\boldsymbol X(t,0,x))
      -\boldsymbol b(t,\boldsymbol X(t,0,y))|
    \nonumber\\
    &\leq
    \|\nabla\boldsymbol b(t,\cdot)\|_
    {L^\infty(B_R+2T\|\boldsymbol b\|_\infty)}
    |\boldsymbol X(t,0,x)-\boldsymbol X(t,0,y)|,
    \label{lipschitzineq}
\end{align}
where \(x,y\in B_R\). Gronwall's inequality then gives
\[
|\boldsymbol X(t,0,x)-\boldsymbol X(t,0,y)|
\leq
\exp\left(
    \int_s^t
    \|\nabla\boldsymbol b(\tau,\cdot)\|_
    {L^\infty(B_R+2T\|\boldsymbol b\|_\infty)}
    \,\dd\tau
\right)
|\boldsymbol X(s,0,x)-\boldsymbol X(s,0,y)|.
\]

The argument of Bru\`e and Nguyen \cite{bruenguyen} for proving \eqref{estimate} under \eqref{assumptionp1} consists in replacing the classical mean-value inequality by the Lusin mean-value estimate
\[
|\boldsymbol b(t,x)-\boldsymbol b(t,y)|
\leq
C_d\bigl(
    \operatorname M_\lambda(D\boldsymbol b)(t,x)
    +
    \operatorname M_\lambda(D\boldsymbol b)(t,y)
\bigr)|x-y|,
\]
valid whenever \(|x-y|<\lambda\). Here,
\[
(\operatorname M_\lambda f)(x)
:=
\sup_{0<r<\lambda}
\frac{1}{\mathcal L^d(B_r)}
\int_{B_r}|f(x+y)|\,\dd y
\]
denotes the truncated Hardy--Littlewood maximal operator. Applying this estimate along two trajectories yields
\begin{align*}
    \frac{\dd}{\dd t}
    |\boldsymbol X(t,0,x)-\boldsymbol X(t,0,y)|
    &\leq
    C_d\bigl(
        \operatorname M_{R+2T\|\boldsymbol b\|_\infty}
        (D\boldsymbol b)(t,\boldsymbol X(t,0,x))
        \\
    &\qquad\quad+
        \operatorname M_{R+2T\|\boldsymbol b\|_\infty}
        (D\boldsymbol b)(t,\boldsymbol X(t,0,y))
    \bigr)
    |\boldsymbol X(t,0,x)-\boldsymbol X(t,0,y)|.
\end{align*}
Thus, \eqref{estimate} follows from Gronwall's inequality with
\[
k(\tau,x)
=
C_d\operatorname M_{R+2T\|\boldsymbol b\|_\infty}
(D\boldsymbol b)(\tau,\boldsymbol X(\tau,0,x)).
\]
The strong \(L^p\)-boundedness of the maximal operator implies that
\(k\in L^1((0,T);L^p_{\loc}(\mathbb R^d))\).

At the endpoint \(p=1\), this reasoning encounters an integrability obstruction. Indeed, the classical Wiener--Stein characterization gives, locally,
\[
\operatorname M_\lambda u\in L^1_{\loc}
\quad\Longleftrightarrow\quad
u\in L\log L_{\loc}.
\]
Consequently, the estimate \eqref{estimate} with
\(k\in L^1((0,T);L^1_{\loc}(\mathbb R^d))\) remains available when
\(D\boldsymbol b\in L^1((0,T);L\log L_{\loc}(\mathbb R^d))\), but the classical maximal operator does not cover the full DiPerna--Lions class. In particular, it leaves open the case in which \(D\boldsymbol b\) is locally integrable but does not belong to \(L\log L_{\loc}\).

Our principal contribution is to replace the truncated Hardy--Littlewood maximal operator by the weighted maximal operator
\[
(\operatorname M_\lambda^g f)(x)
:=
\sup_{0<r<\lambda}
\frac{1}{g(r)\mathcal L^d(B_r)}
\int_{B_r}|f(x+y)|\,\dd y,
\]
where \(g:\mathbb R_+\to\mathbb R_+\) is chosen so as to improve the integrability of the maximal function. For weights such as
\(g(r)=\log(1+1/r)\), one is led to a weighted Lusin mean-value inequality of the form
\begin{equation}
    |f(x)-f(y)|
    \leq
    C_{d,g}
    \bigl(
        (\operatorname M_\lambda^g Df)(x)
        +
        (\operatorname M_\lambda^g Df)(y)
    \bigr)
    g(|x-y|)|x-y|.
    \label{meanvalueineq}
\end{equation}
Applying \eqref{meanvalueineq} to the vector field along two trajectories gives the Osgood-type differential inequality
\begin{align*}
    \frac{\dd}{\dd t}
    |\boldsymbol X(t,0,x)-\boldsymbol X(t,0,y)|
    &\leq
    C_d\bigl(
        \operatorname M_{R+2T\|\boldsymbol b\|_\infty}^g
        (D\boldsymbol b)(t,\boldsymbol X(t,0,x))
        \\
    &\qquad\quad+
        \operatorname M_{R+2T\|\boldsymbol b\|_\infty}^g
        (D\boldsymbol b)(t,\boldsymbol X(t,0,y))
    \bigr)
    \\
    &\qquad\qquad\times
    g\bigl(|\boldsymbol X(t,0,x)-\boldsymbol X(t,0,y)|\bigr)
    |\boldsymbol X(t,0,x)-\boldsymbol X(t,0,y)|.
\end{align*}
Separating variables leads to \eqref{estimatenew}, with
\[
G(s):=\int_a^s\frac{\dd u}{u g(u)}
\]
and
\[
k(\tau,x)
:=
C_{d,g}\operatorname M_{R+2T\|\boldsymbol b\|_\infty}^g
(D\boldsymbol b)(\tau,\boldsymbol X(\tau,0,x)).
\]
This produces a meaningful estimate provided that \(g\) satisfies the Osgood condition
\begin{equation}
    \int_0^a\frac{\dd r}{r g(r)}=\infty
    \label{osgoodcond}
\end{equation}
and that \(k\in L^1((0,T);L^1_{\loc}(\mathbb R^d))\). For example, when
\(g(r)=\log(1+1/r)\), condition \eqref{osgoodcond} is immediate, while the required integrability of \(k\) follows from
\(D\boldsymbol b\in L^1((0,T);L\log\log L_{\loc}(\mathbb R^d))\).

In spite of its informality, the preceding discussion identifies three ingredients required by the argument:
\begin{itemize}
    \item a weighted Lusin mean-value inequality of the form \eqref{meanvalueineq};
    \item a Wiener-type estimate ensuring that
    \(\Phi(|u|)\in L^1_{\loc}\) implies
    \(\operatorname M_\lambda^g u\in L^1_{\loc}\), for a suitable superlinear function \(\Phi\);
    \item the Osgood condition \eqref{osgoodcond}.
\end{itemize}

As we shall see, the weighted mean-value inequality naturally leads to weights \(g\) that are slowly varying at the origin in the sense of Karamata. This is an important class of functions, arising in several applications of analysis and probability; see \cite{BGT}.
The Wiener-type estimate, in turn, requires the monotonicity of $g(r)$ near $r=0$ and is associated with the function
\[
\Phi(s)
=
s\int^s\frac{\dd r}{r g(r^{-1/d})}.
\]
The Osgood condition is then equivalent to the superlinearity of \(\Phi\), i.e., that the information $\Phi(u) \in L^1_{\loc}$ is nontrivial.

One may also reverse this perspective and construct \(g\) from \(\Phi\). By the de la Vall\'ee Poussin theorem (see, for instance, \cite[Theorem 6.19]{superlinear}), for every \(u\in L^1(\mathbb R^d)\) there exists a continuous increasing function
\begin{equation}
    \Phi(0)=0
    \qquad\text{(``normalization'')},
    \label{normalization}
\end{equation}
\begin{equation}
    \lim_{t\to\infty}\frac{\Phi(t)}{t}=\infty
    \qquad\text{(``superlinearity'')},
    \label{superlinearity}
\end{equation}
and
\begin{equation}
    \int_{\mathbb R^d}\Phi(|u|)\,\dd x<\infty
    \qquad\text{(``\(\Phi\) integrates \(u\)'')}.
    \label{integrability}
\end{equation}
Thus, given
\(\boldsymbol b\in L^1((0,T);W^{1,1}_{\loc}(\mathbb R^d))\), the problem reduces to constructing, from a superlinear function \(\Phi\) adapted to the integrability of \(D\boldsymbol b\), a weight \(g\) for which the three ingredients above hold. 

This leaves one fundamental question: for an arbitrary vector field in the DiPerna--Lions class, can one always choose a sufficiently slowly growing superlinear function \(\Phi\) so that this construction can be carried out? The answer is affirmative and is provided by \Cref{slowdown1}. This is what ultimately allows us to prove \Cref{main} under the sole Sobolev assumption \eqref{regularity}.

\begin{remark}\label{remarkLiLuo}
    \textnormal{An abstract Osgood--Lusin framework for regular Lagrangian flows was previously developed by Li and Luo \cite{osgoodflow}. More precisely, they assume that there exist an Osgood modulus \(\rho:[0,\infty)\to[0,\infty)\) and a nonnegative function
    \(f\in L^1((0,T);L^1_{\loc}(\mathbb R^d))\) such that
    \[
    |\boldsymbol b(t,x)-\boldsymbol b(t,y)|
    \leq
    \bigl(f(t,x)+f(t,y)\bigr)\rho(|x-y|)
    \]
    for almost every \(x,y\in\mathbb R^d\) and almost every \(t\in(0,T)\). Here, the Osgood condition means that
    \(\int_0^1\rho(s)^{-1}\,\dd s=\infty\). Under this structural hypothesis, together with suitable boundedness and compressibility assumptions, they establish existence, uniqueness, and Lusin-type regularity of the associated flow. This approach was further expanded very recently by De Philippis and Franchi \cite{dpf}; see also \cite{ck} for another very recent development in the Osgood theory.}

    \textnormal{Our result may be viewed as an intrinsic Sobolev realization of this framework. Indeed, we do not assume an Osgood--Lusin estimate on \(\boldsymbol b\) \textit{a priori}. Starting instead from the DiPerna--Lions regularity \eqref{regularity}, we construct a slowly varying weight \(g\), adapted to the integrability of \(D\boldsymbol b\), such that
    \[
    |\boldsymbol b(t,x)-\boldsymbol b(t,y)|
    \leq
    C_{d,g}|x-y|g(|x-y|)
    \bigl(
        \operatorname M_\lambda^g(D\boldsymbol b)(t,x)
        +
        \operatorname M_\lambda^g(D\boldsymbol b)(t,y)
    \bigr);
    \]
    see \Cref{meanvalue}. Thus, in the notation of Li and Luo, one may take
    \(\rho(r)=r g(r)\) and
    \(f(t,x)=C_{d,g}\operatorname M_\lambda^g(D\boldsymbol b)(t,x)\). The Wiener-type estimate developed in \Cref{wienersubsection} guarantees the required local integrability of \(f\), while the construction in \Cref{constructiong} relates the resulting Osgood modulus explicitly to the integrability properties of \(D\boldsymbol b\). In this way, the Osgood--Lusin condition becomes a consequence of Sobolev regularity rather than an additional hypothesis on the vector field.}
\end{remark}

\subsection{Structure of the paper.}
In Section~2, we introduce the \(g\)-maximal operator and establish its fundamental properties. Section~3 is devoted to the proof of \Cref{main}: we first derive a quantitative result under additional assumptions and then remove these restrictions through the construction in \Cref{slowdown1}. In Section~4, we compute the moduli of continuity appearing in \Cref{examples} and provide examples and counterexamples describing the sharpness and the range of validity of \Cref{lusin} and \eqref{inequalitycrude}. Section~\ref{SecReg} concerns the regularity of solutions to the transport equation \eqref{transport}, while Section~\ref{SecMix} develops applications to quantitative mixing. Finally, in Section~\ref{SecFinal}, we discuss several open problems related to our results.

\subsection{Some notation}\label{section:notation}

We shall often formulate our results in the language of Orlicz spaces, which are generalizations of the classical \(L^p\)-spaces; see, e.g., \cite{Rao}. We therefore briefly fix the notation that will be used throughout the paper.

For our purposes, it will be sufficient to consider finite superlinear Young functions, namely continuous convex functions
\(
\Phi_0:[0,\infty)\to[0,\infty)
\)
satisfying \eqref{normalization} and \eqref{superlinearity}. If \((X,\mu)\) is a \(\sigma\)-finite measure space, we define the Orlicz space \(L^{\Phi_0}(X)\) by
\[
L^{\Phi_0}(X)
=
\left\{
f:X\to\mathbb R :
f \text{ measurable and there exists } \lambda>0
\text{ such that }
\int_X \Phi_0\bigl(\lambda^{-1}|f|\bigr)\,\dd\mu<\infty
\right\}.
\]
On \(L^{\Phi_0}(X)\) we consider the Luxemburg norm
\[
\|f\|_{L^{\Phi_0}(X)}
=
\inf\left\{
\lambda>0 \,;\,
\int_X \Phi_0\bigl(\lambda^{-1}|f(x)|\bigr)\,\dd\mu(x)\leq 1
\right\},
\]
under which \(L^{\Phi_0}(X)\) becomes a Banach space.


The Bochner spaces \(L^1((0,T);L^{\Phi_0}(X))\) are defined in the usual way; see, for example, \cite{evans}. If $X \subset \mathbb R^d$, then the local space $L^{\Phi_0}_{\loc}(X)$ is also defined as usual.

\subsection{Acknowledgments.}
This study was financed in part by the São Paulo Research Foundation (FAPESP) (Process Number 2024/21041-1, Process Number	
2025/12847-5, and Process Number 2023/13426-8) and by the Pró-Reitoria de Pesquisa e Inovação of the University of São Paulo through the ``Programa de Apoio a Novos Docentes'' (Grant 22.1.09345.01.2). 
We would like to express our gratitude to our dear friend Christian Táfula for numerous insightful discussions and for acquainting us with the theory of functions of regular variation and to professor Luigi Ambrosio for bibliography recommendations and careful reading of the article.

\section{The \texorpdfstring{$g$}{}-maximal operator}\label{section:maximal}

In this section we introduce the \(g\)-maximal operator mentioned above and establish the three basic ingredients needed for our approach: a Lusin-type mean value inequality, a Wiener-type integrability bound, and the Osgood condition. We begin with the definition.

\begin{definition}
    Let \(g : \mathbb R_+ \to \mathbb R_+\) be a positive continuous function, and let \(\lambda\in\mathbb R_+ \cup \{\infty\}\). We define the \(g\)-maximal operator \(\operatorname{M}_\lambda^g\) by
    \begin{equation}
        (\operatorname{M}_\lambda^g f)(x)\coloneqq \sup_{r\in(0,\lambda)}\frac{1}{g(r)\mathcal{L}^d(B_r)}\int_{B_r(x)}|f(y)|\,\dd y,
    \end{equation}
    where \(f \in L^1_{\loc}(\mathbb R^d)\).
\end{definition}

As with the classical Hardy--Littlewood maximal operator, \(\operatorname{M}_\lambda^g f\) is measurable; indeed, it may be written as the supremum over rational radii of measurable averaging operators.

\subsection{The Lusin inequality}\label{lusinsubsection}

We first investigate how the weight appearing in the definition of \(\operatorname{M}_\lambda^g\) leads to weighted mean value inequalities.

\begin{lemma}\label{meanvalue}
    Let \(\lambda>0\). Assume that \(g:\mathbb{R}_+\to \mathbb{R}_+\) satisfies the following condition: there exists \(C_\lambda>0\) such that
\begin{equation}\label{hypothesisg}
    \int_0^1\tau^{d-1}g(z\tau)\,\dd \tau\leq C_\lambda g(z)
\end{equation}
    for every \(z\in (0,\lambda)\).

    Then, for every \(u\in W^{1,1}_{\loc}(\mathbb{R}^d)\) and for almost every \(x,y \in \mathbb R^d\) with \(|x-y|<\lambda\),
    \[
    |u(x)-u(y)|\leq C_\lambda C_d\bigl((\operatorname{M}_{\lambda}^gDu)(x)+(\operatorname{M}_{\lambda}^gDu)(y)\bigr)|x-y|g(|x-y|),
    \]
    where \(C_d>0\) depends only on the dimension.
\end{lemma}

\begin{proof}
    This is a natural extension of the classical Lusin mean value inequality. Let us follow, for instance, the argument of \cite[Lemma 2.7]{partiallyregular}. By the fundamental theorem of calculus, one has
    \[
    |u(x)-u(y)|\leq \frac{C_d}{|x-y|^{d-1}}\int_0^1\left[\int_{\Omega_{\tau,x}}\frac{|D u (w)|}{1-\tau}\,\dd w+\int_{\Omega_{\tau,y}}\frac{|D u (w)|}{1-\tau}\,\dd w\right]\,\dd \tau,
    \]
    where \(\Omega_{\tau,z}=B_{(1-\tau)|x-y|}(z)\). On the other hand, by the definition of the \(g\)-maximal operator,
    \[
    \int_0^1\int_{\Omega_{\tau,x}}\frac{|D u (w)|}{1-\tau}\,\dd w\,\dd \tau
    \leq
    C_d|x-y|^d \left(\int_0^1 (1-\tau)^{d-1}g(|x-y|(1-\tau))\,\dd \tau \right) (\operatorname{M}_{\lambda}^gDu)(x),
    \]
    and the same estimate holds for the second term. The conclusion now follows from \eqref{hypothesisg}.
\end{proof}

We are thus led to ask when \eqref{hypothesisg} is satisfied. Although the condition is trivial when \(g\) is increasing, that regime is of little interest for our purposes, since we are mainly concerned with weights that blow up at the origin. The natural class to consider is therefore that of regularly varying functions in the sense of Karamata; see \cite{BGT}.

\begin{definition} \label{defRV}
\textnormal{Let \(\ell : \mathbb R_+\to \mathbb R_+\) be a measurable function.}
\begin{itemize}
    \item \textnormal{We say that \(\ell\) is regularly varying at \(\infty\) with index \(\rho \in \mathbb R\) (symbolically, \(\ell \in \mathrm{RV}_{\rho}(\infty)\)) if, for every \(\lambda>0\),}
    \[
    \operatornamewithlimits{ess\, lim}_{t \to \infty} \frac{\ell(\lambda t)}{\ell(t)} = \lambda^\rho.
    \]

    \item \textnormal{We say that \(\ell\) is regularly varying at \(0\) with index \(\rho \in \mathbb R\) (symbolically, \(\ell \in \mathrm{RV}_{\rho}(0)\)) if, for every \(\lambda>0\),}
    \[
    \operatornamewithlimits{ess\, lim}_{t \to 0} \frac{\ell(\lambda t)}{\ell(t)} = \lambda^\rho.
    \]

    \item \textnormal{We say that \(\ell\) is slowly varying (either at \(0\) or at \(\infty\)) if \(\rho = 0\).}
\end{itemize}
\end{definition}

\begin{remark}\label{examplesslowly}
    \textnormal{Examples of slowly varying functions at \(0\) include constants, \(\log(2+\frac{1}{r})\), \(\log(2+\frac{1}{r})^\alpha\), \(\log \log(10 + \frac{1}{r})\), and \(\log \log \log (500 + \frac{1}{r})\). Likewise, examples of regularly varying functions at \(\infty\) with index \(1\) include \(t\), \(t\log_+ t\), \(t\log_+^\alpha t\), and \(t \log_+\log_+ t\).}
\end{remark}

The previous remark may seem somewhat artificial, but Karamata's representation theorem shows that, in a precise sense, slowly varying functions are essentially of logarithmic behavior. More precisely, a function \(\ell \in \operatorname{RV}_0(\infty)\) if and only if there exists \(A>0\) such that
\begin{equation}
    \ell(x) = \exp\left\{\eta(x) + \int_A^x \frac{\epsilon (u)}{u} \, \dd u\right\} \label{karamata}
    \qquad \text{for all } x > A,
\end{equation}
where \(\eta, \epsilon \in L^\infty(0,\infty)\), with
\[
\operatorname*{ess\,lim}_{x \to \infty} \eta(x) = L \in \mathbb R
\qquad\text{and}\qquad
\operatorname*{ess\,lim}_{x \to \infty} \epsilon(x) = 0.
\]

We shall also use the following notion.

\begin{definition}\label{normalized}
    \textnormal{We say that \(\ell \in \operatorname{RV}_0(\infty)\) is normalized if \(\eta(x)\) may be chosen constant in \eqref{karamata}.}
\end{definition}

\begin{remark}\label{normalized0}
    \textnormal{Although the notion of a normalized slowly varying function may at first seem somewhat elusive, in practice it is often easy to verify. Indeed, a function \(\ell:\mathbb R_+\to\mathbb R_+\) is slowly varying and normalized if and only if there exists \(A>0\) such that \(\ell\in W^{1,1}_{\loc}(A,\infty)\) and
    \[
    \operatornamewithlimits{ess\,lim}_{x\to\infty}\frac{x\ell'(x)}{\ell(x)}=0.
    \]
    In that case, one may simply define
    \(
    \epsilon(x)=x\ell'(x) \ell(x)^{-1}
    \) for $x>A$.}

\end{remark}

We also recall that \(\ell \in \operatorname{RV}_\rho(\infty)\) if and only if
\(
\ell(x)=x^\rho L(x)
\)
for some slowly varying function \(L \in \operatorname{RV}_0(\infty)\); see \cite[Theorem 1.4.1]{BGT}.

\begin{remark}\label{equivalence}
    \textnormal{Completely analogous statements hold for regularly varying functions at \(0\), upon making the change of variables \(x=1/r\).}
\end{remark}

We conclude this subsection with the result linking the \(g\)-maximal operator to slowly varying functions.

\begin{lemma} \label{slowvaryingg}
    Assume that \(g \in \operatorname{RV}_0(0)\) is continuous.

    Then, for every \(\lambda \in \mathbb R_+\), there exists \(C_{\lambda,g}>0\) such that \eqref{hypothesisg} holds for all \(z \in (0,\lambda)\).

    If, furthermore, \(g(r)\) is nondecreasing for sufficiently large \(r\), then there exists \(C_0>0\) such that
    \begin{equation}
        \int_0^1\tau^{d-1}g(z\tau)\,\dd \tau\leq C_{\infty,g} g(z)
        \qquad \text{for all } z>0.
        \label{hypothesisg2}
    \end{equation}
    As a consequence, in this case \eqref{hypothesisg} holds for $\lambda = \infty$ and \(z \in \mathbb R_+\).
\end{lemma}

\begin{proof}
    Clearly, \eqref{hypothesisg} is equivalent to
    \begin{equation}
        \sup_{0<z<\lambda}\int_0^1 \tau^{d-1} \frac{g(z\tau)}{g(z)} \,\dd\tau < \infty.
        \label{goal}
    \end{equation}
    By Potter's bounds, for any \(0<\epsilon<1\), there exists \(s_0>0\) such that
    \[
    \frac{g(s\tau)}{g(s)} \leq (1+\epsilon)\tau^{-\epsilon}
    \qquad \text{for all } \tau \in (0,1) \text{ and } s \in (0,s_0).
    \]
    Hence, for \(0<s<s_0\),
    \[
    \int_0^1 \tau^{d-1} \frac{g(s\tau)}{g(s)} \,\dd\tau
    \leq
    (1+\epsilon)\int_0^1 \tau^{d-1-\epsilon}\,\dd \tau
    =
    \frac{1+\epsilon}{d-\epsilon}.
    \]
    For \(s\in[s_0,\lambda)\), the desired bound follows from the continuity and positivity of \(g\). This proves \eqref{hypothesisg}.

    Assume now, in addition, that \(g(r)\) is nondecreasing for sufficiently large \(r\), say for \(r\in(R_1,\infty)\). Since \eqref{hypothesisg} already holds for \(\lambda=2R_1\), it remains to consider \(s>2R_1\). For such \(s\), we split
    \begin{align*}
        \int_0^1 \tau^{d-1} \frac{g(s\tau)}{g(s)}\,\dd\tau
        &=
        \int_0^{R_1/s} \tau^{d-1} \frac{g(s\tau)}{g(s)} \,\dd \tau
        +
        \int_{R_1/s}^1 \tau^{d-1} \frac{g(s\tau)}{g(s)} \,\dd \tau = \operatorname{(I)} + \operatorname{(II)}.
    \end{align*}
    By monotonicity, \(\operatorname{(II)} \leq 1/d\). On the other hand,
    \[
    \operatorname{(I)}
    \leq
    \frac{1}{g(R_1) s^d}\int_0^{R_1} u^{d-1} g(u)\,\dd \tau.
    \]
    The latter is uniformly bounded by the argument already used above, combining Potter's bounds near the origin with continuity on compact intervals. This proves \eqref{hypothesisg2}.
\end{proof}

\subsection{The Wiener bound and the Osgood condition}\label{wienersubsection}

We now turn to the question of whether the weight \(g(r)\) forces \(\operatorname{M}_\lambda^g\) to enjoy stronger integrability properties than the classical Hardy--Littlewood maximal operator. When \(\lambda=\infty\), we adopt the convention that \(B_\lambda=\mathbb R^d\).

\begin{lemma}[Wiener bound]\label{constructiong}
    Let \(\lambda \in \mathbb R_+ \cup \{\infty\}\), and assume that \(g : \mathbb R_+ \to \mathbb R_+\) is continuous and nonincreasing on \((0,R_0)\).

    Assume further that
    \[
    \gamma=\gamma_\lambda\coloneqq \inf_{z\in (0,\lambda)} g(z) > 0
    \]
    and define
    \begin{equation}
           \widetilde\Phi(z) = z \left(\int_{R_0^{-d}}^z \frac{\dd s}{s\, g(s^{-1/d})}\right)_+. \label{widetildePhi}
    \end{equation}

    Then, for every \(u\in L^1_{\loc}(\mathbb{R}^d)\) and every bounded measurable set \(U \subset \mathbb R^d\), one has
    \begin{align}
        \int_{U}\left|\operatorname{M}_{\lambda}^g u(x)\right|\,\dd x &\leq C_{d, \gamma, R_0, \mathcal L^d(U)} \left(1+ \Vert u \Vert_{L^1(U+B_\lambda)}\right) +  C_{d}  \int_{U + B_\lambda} \Phi\!\left(|u(x)|\right)\,\dd x,
        \label{wiener}
    \end{align}
    where
    \[
    U + B_\lambda = \{x+y \in \mathbb R^d : x\in U,\ y \in B_\lambda\},
    \]
    and the right-hand side of \eqref{wiener} is to be understood to be $0$ if $\Vert u \Vert_{L^1(U+B_\lambda)} = 0$.
\end{lemma}

\begin{proof}
    Fix \(u \in L^1_{\loc}(\mathbb R^d)\), \(\lambda \in \mathbb R_+ \cup \{\infty\}\), and a bounded measurable set \(U \subset \mathbb R^d\). We may clearly assume that \(\Vert u \Vert_{L^1(U+B_\lambda)}\neq 0\).

    \textit{Step 1: A weak bound.}
    We claim that, for
    \begin{equation}
      t>\frac{\|u\|_{L^1(U + B_\lambda)}}{\gamma \mathcal L^d(B_1) R_0^d}=:A,
      \label{condt}
   \end{equation}
   one has
   \begin{align}
       \mathcal{L}^d&\left(\left\{x\in U: \operatorname{M}_{\lambda}^g u(x)>2t\right\}\right) \nonumber \\
       &\leq \frac{C_{d}}{t\,g\left(\left[(t\gamma \mathcal L^d(B_1))^{-1}\|u\|_{L^1(U + B_\lambda)}\right]^{1/d}\right)}
       \int_{\left\{x\in U+B_\lambda: |u(x)|>\gamma t\right\}}|u(x)|\,\dd x.
       \label{weakbound}
   \end{align}

   Indeed, if \(x \in U\) satisfies \(\operatorname{M}_{\lambda}^g u(x)>t\), then there exists \(r\in(0,\lambda)\) such that
    \begin{equation}
        \mathcal L^d(B_1) r^d< \frac{1}{ t g(r)}\int_{B_r(x)}|u(y)|\,\dd y.
        \label{equacao1}
    \end{equation}
    In particular,
    \[
    r^d< \frac{1}{t\gamma\mathcal L^d(B_1)}\|u\|_{L^1(U + B_\lambda)}.
    \]
    If \(t\) satisfies \eqref{condt}, then necessarily \(r\in(0,R_0)\), so that \(g\) lies in its nonincreasing regime. Returning to \eqref{equacao1}, we therefore obtain
    \begin{equation}
        \mathcal L^d(B_r(x))
        =
        \mathcal L^d(B_1) r^d
        <
        \frac{1}{ t g\left(\left[(t\gamma \mathcal L^d(B_1))^{-1}\|u\|_{L^1(U + B_\lambda)}\right]^{1/d}\right)}
        \int_{B_r(x)}|u(y)|\,\dd y.
        \label{equacao2}
    \end{equation}

    Consequently, under \eqref{condt}, the set \(\{x\in U:\operatorname{M}_{\lambda}^g u(x)>t\}\) may be covered by balls whose volumes satisfy \eqref{equacao2}. The Besicovitch covering lemma (see \cite[Theorem 2.17]{ambrosio2000functions}) therefore yields the existence of a universal constant \(C_d>0\) such that
    \[
    \mathcal{L}^d\left(\left\{x\in U: \operatorname{M}_{\lambda}^g u(x)>t\right\}\right)\leq \frac{C_{d}}{t\,g\left(\left[(t\gamma \mathcal L^d(B_1))^{-1}\|u\|_{L^1(U+B_{\lambda})}\right]^{1/d}\right)}\int_{ U+ B_{\lambda}}|u(x)|\,\dd x.
    \]

    To sharpen this estimate, decompose \(u=u_1+u_2\), where
    \[
    u_1(x)=u(x)\mathbbm{1}_{\{x\in U+B_\lambda:\ |u(x)|>\gamma t\}},
    \qquad
    u_2(x)=u(x)\mathbbm{1}_{\{x\in U+B_\lambda:\ |u(x)|\leq \gamma t\}}.
    \]
    Since \(|u_2|\leq \gamma t\) and \(g(r)\geq \gamma\) for all \(r\in(0,\lambda)\), one has \(\operatorname{M}_\lambda^g u_2\leq t\). Hence
    \[
    \{\operatorname{M}_\lambda^g u>2t\}\subset \{\operatorname{M}_\lambda^g u_1>t\},
    \]
    and \eqref{weakbound} follows by applying the previous estimate to \(u_1\).

    \textit{Step 2: The Wiener argument.}
    By Cavalieri's principle and \eqref{weakbound},
    \begin{align}
        \int_{U}&\left|\operatorname{M}_\lambda^g u(x)\right|\,\dd x \nonumber\\
        &=\int_0^\infty\mathcal{L}^d\left(\left\{x\in U:\left|\operatorname{M}_\lambda^g u(x)\right|>t\right\}\right)\,\dd t \nonumber\\
        &\leq \frac{1}{2}\mathcal L^d(U)A+2\int_{A}^\infty\mathcal{L}^d\left(\left\{x\in U: \operatorname{M}_{\lambda}^g u(x)> 2t\right\}\right)\,\dd t \nonumber\\
        &\leq \frac{1}{2}\mathcal L^d(U)\frac{\Vert u\Vert_{L^1(U+B_\lambda)}}{\gamma \mathcal L^d(B_1)R_0^d} \nonumber\\
        &\qquad + C_{d}\int_{A}^\infty \frac{1}{t\,g\left(\left[(t\gamma \mathcal L^d(B_1))^{-1}\|u\|_{L^1(U+B_\lambda)}\right]^{1/d}\right)}
        \int_{\{x\in U + B_\lambda: |u(x)|> \gamma t\}}|u(x)|\,\dd x\,\dd t.
        \label{integralMg}
    \end{align}
    We now apply Fubini--Tonelli to the second term in \eqref{integralMg}. After the change of variables
    \[
    s=\frac{ t \gamma \mathcal L^d(B_1)}{\|u\|_{L^1(U+B_\lambda)}},
    \]
    one obtains
    \begin{align*}
        &C_{d}\int_{A}^\infty \frac{1}{t\,g\left(\left[(t\gamma \mathcal L^d(B_1))^{-1}\|u\|_{L^1(U+B_\lambda)}\right]^{1/d}\right)}
        \int_{\{x\in U + B_\lambda: |u(x)|> \gamma t\}}|u(x)|\,\dd x\,\dd t \\
        &\qquad=
        C_{d}\frac{\Vert u \Vert_{L^1(U+B_\lambda)}}{\mathcal L^d(B_1)}
        \int_{\{x \in U+B_\lambda: |u(x)| > \Vert u \Vert_{L^1(U+B_\lambda)}/(\mathcal L^d(B_1)R_0^d) \}}
        \widetilde \Phi\!\left(\frac{\mathcal L^d(B_1) |u(x)|}{\Vert u \Vert_{L^1(U+B_\lambda)}}\right)\,\dd x.
    \end{align*}
    Hence \eqref{integralMg} becomes
    \begin{align}
        \int_{U}&\left|\operatorname{M}_{\lambda}^g u(x)\right|\,\dd x\nonumber \\
        &\leq \frac{C_d}{\gamma R_0^d} \mathcal L^d(U)\Vert u\Vert_{L^1(U+B_\lambda)}  +  C_{d} \frac{\Vert u \Vert_{L^1(U+B_\lambda)}}{\mathcal L^d(B_1)} \int_{U + B_\lambda}\widetilde \Phi\!\left(\frac{\mathcal L^d(B_1) |u(x)|}{\Vert u \Vert_{L^1(U+B_\lambda)}}\right)\,\dd x.
        \label{wiener0}
    \end{align}

    \textit{Step 3: Conclusion.}
  To deduce \eqref{wiener} from \eqref{wiener0}, we proceed as follows. Define
    \[
    \Psi(z) = \left( \int_{R_0^{-d}}^z \frac{\dd s}{ s \, g(s^{-1/d}) } \right)_+,
    \]
    so that \(\widetilde \Phi(z) = z \,\Psi(z)\). Then \eqref{wiener0} may be rewritten as
    \begin{align}
        \int_{U}&\left|\operatorname{M}_{\lambda}^g u(x)\right|\,\dd x\nonumber \\
        &\leq \frac{C_d}{\gamma R_0^d} \mathcal L^d(U)\Vert u\Vert_{L^1(U+B_\lambda)}  +  C_{d}  \int_{U + B_\lambda}|u(x)| \Psi\!\left(\frac{\mathcal L^d(B_1) |u(x)|}{\Vert u \Vert_{L^1(U+B_\lambda)}}\right)\,\dd x.
        \label{wiener0'}
    \end{align}
    Moreover, for every \(\alpha>0\),
    \[
    \Psi(\alpha z) \leq  \Psi(z) + \frac{1}{\gamma}\log_+\alpha.
    \]
    It follows that
    \begin{align}
        \int_{U}&\left|\operatorname{M}_{\lambda}^g u(x)\right|\,\dd x\nonumber \\
        &\leq \frac{C_d}{\gamma} \Vert u\Vert_{L^1(U+B_\lambda)}\bigg(\frac{\mathcal L^d(U)}{R_0^d} + \log_+\bigg( \frac{\mathcal L^d(B_1)}{\Vert u \Vert_{L^1(U+B_\lambda)}} \bigg) \bigg)  +  C_{d}  \int_{U + B_\lambda}|u(x)|  \Psi(|u(x)|)\,\dd x.
        \nonumber
    \end{align}
    Because $\log_+(1/z) \leq C_{\epsilon} z^{-\epsilon}$, for $z > 0$ and $0<\epsilon<1$, we thus derive \eqref{wiener} via Young's inequality. The proof is complete.
\end{proof}

\begin{remark}
\textnormal{A few remarks are in order.}
\begin{itemize}
    \item \textnormal{(Relation with the Osgood condition). The estimate \eqref{wiener} is naturally related to the Osgood condition \eqref{osgoodcond}. Indeed, the change of variables \(r=s^{-1/d}\) shows that \eqref{osgoodcond} is equivalent to
    \[
    \lim_{z\to\infty}\frac{\widetilde\Phi(z)}{z}=\infty,
    \]
    that is, to the superlinearity of \(\widetilde\Phi\).}

    \item \textnormal{(A Stein-type counterpart). We refer to \Cref{constructiong} as a Wiener bound because Wiener proved (see \cite{wiener}) that, if \(u \in L\log L_{\loc}\), then \(\operatorname M_\lambda u \in L^1_{\loc}\). Stein later established the converse implication: if \( \operatorname M_\lambda u \in L^1_{\loc}\), then \(u \in L\log L_{\loc}\); see \cite{Stein1969}. It would be interesting to know whether an analogous Stein-type counterpart holds for \Cref{constructiong}.}


\end{itemize}
\end{remark}

In the next sections, we shall need \Cref{constructiong} in the following balanced form.

\begin{corollary}\label{constructiong2} 
Assume the hypotheses of \Cref{constructiong}, and suppose in addition that there exists a finite superlinear Young function \(\Phi_0:[0,\infty)\to[0,\infty)\) such that \begin{equation} \widetilde \Phi(z)\leq C_0\bigl(1+\Phi_0(z)\bigr) \qquad \text{for all } z\geq 0 \label{domination1} \end{equation} for some constant \(C_0>0\). Then, for every \(u\in L^{\Phi_0}(U+B_\lambda)\), one has \begin{equation} \|\operatorname{M}^g_\lambda u\|_{L^1(U)} \leq C_{d,\Phi_0,U}\bigl( \|u\|_{L^1(U+B_\lambda)} + \|u\|_{L^{\Phi_0}(U+B_\lambda)} \bigr). \end{equation} \end{corollary} 
\begin{proof} Let \( \mu\coloneqq\|u\|_{L^{\Phi_0}(U+B_\lambda)}, \) which we may assume is positive. Applying \eqref{wiener} to \(\mu^{-1}u\), and using the homogeneity of \(\operatorname{M}^g_\lambda\), we obtain 
\[ \int_U |\operatorname{M}^g_\lambda u(x)|\,\dd x \leq C_{d,\gamma,R_0,\mathcal L^d(U)}(\mu+ \|u\|_{L^1(U+B_\lambda)}) + C_d\,\mu \int_{U+B_\lambda}\widetilde \Phi\!\left(\mu^{-1}|u(x)|\right)\,\dd x. \] By \eqref{domination1}, \[ \widetilde \Phi\!\left(\mu^{-1}|u(x)|\right) \leq C_0\Bigl(1+\Phi_0\!\left(\mu^{-1}|u(x)|\right)\Bigr). \] Since \(\mu\) is the Luxemburg norm of \(u\) in \(L^{\Phi_0}(U+B_\lambda)\), we have \( \int_{U+B_\lambda}\Phi_0\!\left(\mu^{-1}|u(x)|\right)\,\dd x\leq 1. \) Therefore, \[ \int_U |\operatorname{M}^g_\lambda u(x)|\,\dd x \leq C_{d,\gamma,R_0,\mathcal L^d(U)}\|u\|_{L^1(U+B_\lambda)} + C_{d,\Phi_0,U}\,\mu, \] which is precisely the desired estimate. \end{proof}

\section{Proof of \texorpdfstring{\Cref{main}}{}}\label{section2}

Before proving our main result, we introduce some basic notation and definitions. We denote by $\pazocal{B}(E,F)$ the space of bounded functions from $E$ to $F$, by $L^0(\mathbb R^d)$ the space of measurable functions endowed with the topology of convergence in measure, and by $\log L(\mathbb R^d)$ the space of functions $u$ such that
\[
\pseudo u _{\log L(\mathbb R^d)} \coloneqq  \int_{\mathbb R^d}\log(1+|u(x)|)\,\dd x<\infty.
\]
Each of these spaces admits a natural local counterpart, and we shall use the corresponding notation without further comment.

\begin{definition}[Regular Lagrangian flow]\label{rlf}
\textnormal{We say that $\boldsymbol X : \{(s,t,x) \in [0,T]\times [0,T] \times \mathbb R^d; s\leq t\} \to \mathbb R^d$ is a regular Lagrangian flow in the renormalized sense associated with \eqref{flow} if the following conditions hold:
\begin{enumerate}
    \item For every $s\in[0,T)$, one has
    \[
    \boldsymbol X(\,\cdot\,,s,\,\cdot\,)\in C([s,T];L^0_{\loc}(\mathbb R^d))\cap \pazocal{B}([s,T];\log L_{\loc}(\mathbb R^d)),
    \]
    and
    \[
    \boldsymbol X(s,s,x)=x
    \]
    for almost every $x\in\mathbb R^d$.
    \item For every $s \in [0,T)$ and every $\beta\in C^1(\mathbb R^d;\mathbb R)$ satisfying
    \[
    |\beta(z)|\leq C\bigl(1+\log(1+|z|)\bigr)
	\quad\text{and}\quad
	|\nabla\beta(z)|\leq C(1+|z|)^{-1}
    \qquad \forall z\in\mathbb R^d,
	\]
	for some constant $C>0$, one has
	\[
    \partial_t\bigl(\beta(\boldsymbol X(t,s,x))\bigr)
	=
	\nabla\beta(\boldsymbol X(t,s,x))\cdot \boldsymbol b(t,\boldsymbol X(t,s,x))
	\]
	in the weak sense on $(s,T)\times\mathbb R^d$.
    \item There exists a constant $L>0$, called the compressibility constant, such that for every $0\le s\le t\le T$,
	\[
    \boldsymbol X(t,s,\cdot)_\#\mathcal L^d\leq L\mathcal L^d,
	\]
    that is,
	\[
	\int_{\mathbb R^d}\varphi(\boldsymbol X(t,s,x))\,\dd x
	\leq
	L\int_{\mathbb R^d}\varphi(x)\,\dd x
	\]
	for every measurable nonnegative function $\varphi$.
    \item The (forward) semigroup property holds: for every $0\le s\le \tau\le t\le T$,
    \[
    \boldsymbol X(t,\tau,\boldsymbol X(\tau,s,x))=\boldsymbol X(t,s,x)
    \]
    for almost every $x\in\mathbb R^d$.
\end{enumerate}}
\end{definition}

Endowed with the notion of the $g$-maximal function, we are in a position to prove \Cref{main}. Let \(\epsilon>0\) be fixed, and assume throughout this section that the vector field \(\boldsymbol{b}\) satisfies \eqref{regularity}, \eqref{growth}, and \eqref{compressibility}.

A fundamental ingredient is the following sublevel estimate of Crippa--De Lellis \cite[Proposition 3.2]{crippadelellis}, which asserts, roughly speaking, that every integral curve of \eqref{flow}, apart from a set of small measure, remains bounded. This allows us to remove the auxiliary assumption \eqref{linfty}, which was imposed in the introduction for expository purposes. The point is that, after excluding a set of initial data of arbitrarily small measure, one may work with a truncated operator \(\operatorname{M}_\lambda^g\), which depends only on the local integrability of \(D\boldsymbol{b}\). For a detailed proof of the result below, we refer to \cite[Lemma 5.5]{bouchutcrippa}.

\begin{lemma}\label{sublevel} Let \(R>0\), and define
    \[
    V_{\lambda,R}=\bigl\{x\in B_R:\ |\boldsymbol{X}(t,0,x)|<\lambda \text{ for all } t\in[0,T]\bigr\}.
    \]
    Then, for every \(\epsilon>0\), there exists \(\lambda=\lambda(R,\epsilon)>0\) such that
    \begin{equation}
        \mathcal{L}^d\bigl(B_R\setminus V_{\lambda,R}\bigr)<\epsilon/2.
        \label{ineqsublevel}
    \end{equation}
\end{lemma}

Henceforth, we fix \(\lambda>0\) as in \eqref{ineqsublevel} and set
\[
U_{\epsilon,R}=V_{\lambda,R}.
\]

Let us also note that, if \(\boldsymbol{b}\in L^1((0,T);L^\infty(\mathbb R^d))\), then one may simply take
\[
\lambda=\int_0^T\|\boldsymbol{b}(t,\cdot)\|_{L^\infty(\mathbb R^d)}\,\dd t.
\]
On the other hand, if \(\boldsymbol{b}\in L^1((0,T);W^{1,1}(\mathbb R^d))\), no truncation of \(\operatorname{M}_\lambda^g\) is needed, and one may therefore take \(\lambda=\infty\). In both cases, the set \(U_{\epsilon,R}\) may be chosen to be \(B_R\).

Accordingly, in what follows, we shall mainly restrict attention to the genuinely local case
\[
\boldsymbol{b}\in L^1((0,T);W^{1,1}_{\loc}(\mathbb R^d)),
\]
since the two special cases appearing in the statement of \Cref{main} are covered by the preceding remarks.

\subsection{The quantitative lemma}

We begin by deriving a quantitative version of \Cref{main}. In order to keep the resulting estimates explicit, we shall impose a number of additional assumptions. In the next subsection, however, we will show that these hypotheses may always be arranged, although the corresponding functions must in general be constructed in an ad hoc fashion.

\begin{assumption}[Regularity of \(\Phi\)]\label{regularityforPhi} 
We assume that 
\[ D\boldsymbol b \in L^1\bigl((0,T);L^{\Phi_0}(U_{\epsilon,R}+B_\lambda)\bigr), \] 
where \(\Phi_0\) is a finite superlinear Young function. We further assume that there exists a function \(\Phi:[0,\infty)\to[0,\infty)\) satisfying \eqref{normalization}, \eqref{superlinearity}, the estimate
\begin{equation}
   \Phi(s)\leq C_0\bigl(1+\Phi_0(s)\bigr) \qquad \text{for all } s\geq 0,   \label{domination}
\end{equation}
for some constant \(C_0>0\), and the following properties: 

\begin{enumerate} 

\item \(\Psi(s)\coloneqq s^{-1}\Phi(s)\in C^1([0,\infty))\). 
\item The function \(s\mapsto s\Psi'(s)\) is positive on \((0,\infty)\), nonincreasing for large \(s\), and normalized slowly varying at \(\infty\). 
\item If \(\lambda=\infty\), we additionally assume that \(s\mapsto s\Psi'(s)\) is nondecreasing for small \(s>0\). \end{enumerate} \end{assumption} 

Although the normalization assumption in item~(2) will not be needed in the present section, it will play an important role in the later parts of the paper.

Under \Cref{regularityforPhi}, we define the weight \(g\) by inverting the relation in \eqref{widetildePhi}. Proceeding in this way, we obtain a quantitative version of \Cref{main}.

\begin{lemma}\label{quantitative}
Assume the hypotheses of \Cref{main}, and suppose in addition that there exist functions \(\Phi_0\) and $\Phi$ satisfying \Cref{regularityforPhi}.

Then, defining
\begin{equation}
    g\left(t^{-\frac{1}{d}}\right)=\frac{1}{t\Psi'(t)},
    \label{definitiong}
\end{equation}
\begin{equation}
   k(\tau,x)=C_{d,g}(\operatorname{M}_{2\lambda}^g D\boldsymbol{b})(\tau,\boldsymbol{X}(\tau,0,x)),
   \label{definitionk}
\end{equation}
and, for some \(a>0\),
\begin{equation}
  G(z)\coloneqq \int_{a}^z \frac{1}{u g(u)}\,\dd u
  =-\frac{1}{d}z^d\Phi\bigg(\frac{1}{z^d}\bigg)+C_{d,a},
  \label{definitionG}
\end{equation}
the conclusions of \Cref{main} hold. In particular, $g(r)$ is bounded away from $0$ and a normalized slowly varying function at $0$. Furthermore,
$$\int_0^T\int_{U_{\epsilon, R}} k(t, x)\,\dd x \,\dd \tau \leq C_{d, \Phi_0, U_{\epsilon, R}} \bigg(\int_0^T \Vert D\boldsymbol{b} (t,\cdot)\Vert_{L^1(U_{\epsilon,R} + B_\lambda)}\,\dd\tau + \int_0^T \Vert D\boldsymbol{b}(t,\cdot) \Vert_{L^{\Phi_0}(U_{\epsilon,R} + B_\lambda)}\,\dd\tau \bigg).$$
\end{lemma}

\begin{proof}
By \Cref{regularityforPhi}, the function \(g\) defined in \eqref{definitiong} satisfies the hypotheses of \Cref{meanvalue,slowvaryingg,constructiong}. Moreover, the Osgood condition \eqref{osgoodcond} follows immediately from the next computation: for every \(\epsilon>0\),
\[
\int_0^\epsilon\frac{1}{u g(u)}\,\dd u
=
\int_0^\epsilon\frac{\Psi'(u^{-d})}{u^{d+1}}\,\dd u
=
\frac{1}{d}\int_{\epsilon^{-d}}^\infty \Psi'(v)\,\dd v
=
\frac{1}{d}\int_{\epsilon^{-d}}^\infty \frac{\dd}{\dd v}\left(\frac{\Phi(v)}{v}\right)\,\dd v
=
\infty,
\]
since \(\Phi\) is superlinear by assumption.

We may therefore proceed exactly as in the heuristic discussion from the introduction. Let \(x,y\in U_{\epsilon,R}\). By \Cref{sublevel}, we have
\[
|\boldsymbol{X}(t,0,x)|\leq \lambda
\qquad\text{and}\qquad
|\boldsymbol{X}(t,0,y)|\leq \lambda
\]
for all \(0\leq t\leq T\). As a consequence, \Cref{meanvalue} and \Cref{slowvaryingg} yield
\[
\begin{split}
\frac{\dd}{\dd t}\log|\boldsymbol{X}(t,0,x)-\boldsymbol{X}(t,0,y)|
&\leq C_d\, g(|\boldsymbol{X}(t,0,x)-\boldsymbol{X}(t,0,y)|) \\
&\quad\times \bigl(\operatorname{M}_{2\lambda}^g D\boldsymbol{b}(t,\boldsymbol{X}(t,0,x))
+\operatorname{M}_{2\lambda}^g D\boldsymbol{b}(t,\boldsymbol{X}(t,0,y))\bigr).
\end{split}
\]
The reason for differentiating \(\log |\boldsymbol{X}(t,0,x)-\boldsymbol{X}(t,0,y)|\), rather than \(|\boldsymbol{X}(t,0,x)-\boldsymbol{X}(t,0,y)|\) itself, is the renormalization property in \Cref{rlf}. This, however, does not alter the argument in any essential way.

Invoking the comparison principle for ordinary differential equations (see \cite{hartman}), we obtain
\[
\int_{\log |x-y|}^{\log |\Delta \boldsymbol{X}(t,0,x,y)|}\frac{1}{g(e^u)}\,\dd u
\leq
C_d \int_s^t\bigl(\operatorname{M}_{2\lambda}^g D\boldsymbol{b}(\tau,\boldsymbol{X}(\tau,0,x))
+\operatorname{M}_{2\lambda}^g D\boldsymbol{b}(\tau,\boldsymbol{X}(\tau,0,y))\bigr)\,\dd \tau,
\]
where
\[
|\Delta \boldsymbol{X}(t,0,x,y)|\coloneqq|\boldsymbol{X}(t,0,x)-\boldsymbol{X}(t,0,y)|.
\]
Equivalently,
\[
\int_{|x-y|}^{|\Delta \boldsymbol{X}(t,0,x,y)|}\frac{1}{u g(u)}\,\dd u
\leq
C_d \int_s^t\left(\operatorname{M}_{2\lambda}^g D\boldsymbol{b}(\tau,\boldsymbol{X}(\tau,0,x))
+\operatorname{M}_{2\lambda}^g D\boldsymbol{b}(\tau,\boldsymbol{X}(\tau,0,y))\right)\,\dd \tau.
\]
Therefore, defining \(G\) by \eqref{definitionG} and \(k\) by \eqref{definitionk}, we obtain the desired estimate in \Cref{main}, noticing that $k \in L^1((0,T); L^1(U_{\epsilon, R}))$ by \eqref{constructiong2}.  This concludes the proof.
\end{proof}

Before proceeding, however, it is worth briefly discussing the meaning of \Cref{regularityforPhi}. We do admit that \Cref{regularityforPhi} may appear somewhat technical. Nevertheless, as the next subsection---and in particular \Cref{slowdown1}---will show, it is automatically satisfied whenever \(\boldsymbol b\) fulfills the general Sobolev assumption \eqref{regularity}. 

In most applications in which we will be interested, we shall mostly consider finite superlinear Young functions \(\Phi_0\) of the form \[ \Phi_0(s)=s\Psi_0(s), \] where, for all sufficiently large \(s\), the function \(s\mapsto s\Psi_0'(s)\) is positive, nonincreasing, and normalized slowly varying at \(\infty\). Under these assumptions, one can construct the \(g\)-maximal operator when \(\lambda\) is sufficiently small; see \Cref{quantitative} below.  However, when \(\lambda\) is large or $\infty$---which is precisely the situation in \Cref{SecReg}---it becomes necessary to modify \(\Phi_0\) near the origin so as to ensure that \(s\mapsto s\Psi'(s)\) is nondecreasing for small \(s>0\). This modification is purely technical and somewhat artificial, since the interesting behavior of \(g(r)\) arises near \(r=0\), but it is nonetheless needed for the theory developed here. We also point out that, by modifying $\Phi_0$, it may no longer be convex.

Let us now show how such a modification can be performed under the above hypotheses.

\begin{proposition} \label{onPhiandPhi0}
    Let $\Phi_0 : [0,\infty) \to [0,\infty)$ be a finite superlinear Young function. Assume that for $s$ sufficiently large, $s^{-1} \Phi_0(s) =: \Psi_0$ is $C^1$, and \(s\mapsto s\Psi_0'(s)\) is positive, nonincreasing, and normalized slowly varying at \(\infty\).

    Then, there exists some increasing nonnegative $\Psi \in C^1([0,\infty))$ such that $\Psi(s) = \Psi_0(s)$ for all sufficiently large $s$, and conditions (1)--(3) of \Cref{regularityforPhi} are valid. In particular, $\Phi(z) \coloneqq  z\,\Psi(z)$ satisfies the domination estimate \eqref{domination}

    In this case, for $g(z)$ and $G(z)$ given by \eqref{definitionG} and $z$ sufficiently small,
    $$\begin{dcases}
        g(z^{-1/d}) = \frac{1}{z\Psi_0'(z)} \text{ and} \\
        G(z) = -\frac{1}{d}\Psi_0(z^{-d}) + \text{(constant)}.
    \end{dcases}$$
\end{proposition}
\begin{proof}
   We claim that \(\Psi_0\) itself is normalized slowly varying at \(\infty\). Indeed, since \(\Psi_0(s)\to\infty\) and \(s\Psi_0'(s)\) is eventually nonincreasing, one immediately has that \[ \lim_{s\to\infty}\frac{s\Psi_0'(s)}{\Psi_0(s)}=0; \] 
    see \Cref{normalized0}.
    
    Hence one may choose \(s_0>0\) so large that $\Psi'(s_0) > 0$ and
\[ \frac{s_0\Psi_0'(s_0)}{\Psi_0(s_0)}<2. \] 
Let us then define \[ \Psi(s)= \begin{dcases} \Psi_0(s_0)-\frac12\Psi_0'(s_0)s_0+\frac{\Psi_0'(s_0)}{2s_0}s^2 & \text{for } 0\leq s\leq s_0,\\[1ex] \Psi_0(s) & \text{for } s>s_0. \end{dcases} \]
It is then immediate that \(\Phi(s)\coloneqq s\Psi(s)\) satisfies \Cref{regularityforPhi}.
\end{proof}

\begin{remark}
    \textnormal{Conversely, we mention that  if $\Psi : [0,\infty) \to [0,\infty)$ is eventually $C^1$ with \(s \mapsto s\Psi'(s)\) eventually positive and normalized slowly varying at \(\infty\), then $\Phi(s) =  s\,\Psi(s)$ is eventually convex. As a consequence, provided that $\Psi(s) \to \infty$ as $s \to \infty$, one may modify linearly $\Phi(s)$ for sufficiently small $s$ and obtain a finite superlinear Young function $\Phi_0 : [0,\infty) \to [0,\infty)$. (For instance, for some suitable $s_0>0$, one can take $\Phi_0(s) = 1_{[0,s_0]}(s)\,s\,\Psi(s_0) + 1_{(s_0,\infty)}(s)\, s\,\Psi(s)$.)}

    \textnormal{The proof of the eventual convexity of $\Phi(s)$ also revolves around the normalized slowly variation of $s \Psi'(s)$, for it is equivalent to
    \[
    \operatornamewithlimits{ess \, lim}_{s\to\infty}\frac{s\Psi''(s)}{\Psi'(s)} = -1;    \]
    see \Cref{normalized0}. Accordingly,
        \[
    \frac{\dd^2}{\dd s^2}\bigl(s\Psi(s)\bigr)
    =
    2\Psi'(s)+s\Psi''(s)
    =
    \Psi'(s)\biggl(2+\frac{s\Psi''(s)}{\Psi'(s)}\biggr)
    \]
    is positive for almost all sufficiently large $s>0$.} 
\end{remark}

\subsection{The qualitative lemma} \label{qualitativesub}

We now show that \Cref{regularityforPhi} is not a restrictive hypothesis from the point of view of vector fields satisfying \eqref{regularity}, and that this suffices to complete the proof of \Cref{main}. 

The key point is that, as already noted, the de la Vallée Poussin theorem yields a finite superlinear increasing Young function \(\Phi_0:[0,\infty)\to[0,\infty)\) satisfying \eqref{normalization}, \eqref{superlinearity}, and \begin{equation} \int_0^T \int_{U_{\epsilon,R}+B_\lambda} \Phi_0(|D\boldsymbol{b}(t,x)|)\,\dd x\,\dd t < \infty. \label{delavallee} \end{equation} In other words, \eqref{delavallee} asserts precisely that \( D\boldsymbol b \in L^{\Phi_0}\bigl((0,T)\times (U_{\epsilon,R}+B_\lambda)\bigr). \) From this one readily deduces the Bochner-space integrability \begin{equation} D\boldsymbol b \in L^1\bigl((0,T);L^{\Phi_0}(U_{\epsilon,R}+B_\lambda)\bigr). \label{integratesu} \end{equation} Indeed, for almost every \(t\in(0,T)\), an elementary convexity argument yields that the Luxemburg norm satisfies the elementary estimate \[ \|D\boldsymbol b(t,\cdot)\|_{L^{\Phi_0}(U_{\epsilon,R}+B_\lambda)} \leq \operatorname{max} \bigg\{1,\int_{U_{\epsilon,R}+B_\lambda}\Phi_0(|D\boldsymbol b(t,x)|)\,\dd x \bigg\}. \] Thence, \eqref{integratesu} follows immediately from \eqref{delavallee}. 

We have thus proved \Cref{main} under the additional assumption \Cref{regularityforPhi}. At first sight, this may appear rather special, since an arbitrary Young function \(\Phi_0\) for which \( D\boldsymbol b\in L^1\bigl((0,T);L^{\Phi_0}(U_{\epsilon,R}+B_\lambda)\bigr) \) need not satisfy \Cref{regularityforPhi} even in the asymptotic sense, as discussed in \Cref{onPhiandPhi0}. For instance, if \( D\boldsymbol b\in L^1((0,T);L_{\loc}^p(\mathbb R^d))\) for some \(p>1\), then the natural choice is \(\Phi(s)=s^p\), which clearly fails to satisfy \Cref{regularityforPhi}. This apparent difficulty is, however, misleading. One should keep in mind that one is always free to replace \(\Phi\) by a slower superlinear function. Thus, in the previous example, one may replace \(\Phi(s)=s^p\) by \(\Phi(s)=s\log_+ s\) for $s$ large, which does satisfy \Cref{regularityforPhi} (see \Cref{onPhiandPhi0}). One may in fact go even further and work with substantially slower functions, such as \( \Phi(s)=s\log\log(e+s). \) 

Accordingly, \Cref{main} will follow once we show that, given any finite superlinear Young function \(\Phi_0\) satisfying \eqref{integratesu}, one can construct a much slower function which still satisfies \eqref{normalization}, \eqref{superlinearity}, and \Cref{regularityforPhi}. It is convenient to reformulate the problem in terms of \[ \Psi_0(s)=\frac{\Phi_0(s)}{s}. \] Since \(\Phi_0\) is convex, the function \(\Psi_0\) is nondecreasing. Therefore, if one can find a function \(\Xi:\mathbb R_+\to\mathbb R_+\) satisfying the regular variation and eventual monotonicity requirements appearing in \Cref{regularityforPhi}, and such that \[ \Xi(s)\leq \Psi_0(s) \qquad \text{for all sufficiently large } s, \] then \Cref{main} follows. As it turns out, such a ``slowing down'' procedure is always possible. For the reader's convenience, we now restate the properties that will be required.

\begin{lemma}\label{slowdown1}
    Let \(\Psi_0:\mathbb{R}_+\to\mathbb{R}_+\) be continuous, nondecreasing, and unbounded. Then there exists an increasing unbounded function \(\Xi  : [0,\infty) \to [0,\infty)\) satisfying the domination estimate
    $$\text{\(\Xi(s)\leq \Psi_0(s)\) for all sufficiently large \(s\)}$$
    and the following properties:
    \begin{enumerate}[(i)]
        \item The function \(\Xi \in C^1([0,\infty))\) is increasing.
        \item The function \(s\mapsto s\Xi'(s)\) is positive everywhere and eventually nonincreasing, with \(s\Xi'(s)\to 0\) as \(s\to \infty\). 
        \item The function \(s \mapsto s \Xi'(s)\) is a normalized slowly varying function at \(\infty\).
        \item The function \(s \mapsto s \Xi'(s)\) is nondecreasing for all sufficiently small \(s\).
    \end{enumerate}
\end{lemma}

\begin{proof}
    Before turning to the proof, let us first explain the main ideas, which become quite natural once they are viewed from the right perspective.

    We shall work extensively in logarithmic scale:
    \[
    x\coloneqq \log s,
    \]
    or equivalently \(s=e^x\). The reason is simple: if
    \[
    \Lambda(x)\coloneqq \Xi(e^x),
    \]
    then
    \[
    s\Xi'(s)\big|_{s=e^x}=e^x\Xi'(e^x)=\Lambda'(x).
    \]
    Thus the analysis of \(s\Xi'(s)\) becomes considerably simpler, since condition (ii) in \Cref{slowdown1} reduces to the statement that \(\Lambda'(x)\) is positive, eventually nonincreasing, and satisfies
    \[
    \lim_{x\to\infty}\Lambda'(x)=0.
    \]

    Moreover, by Karamata's representation theorem, condition (iii) in \Cref{slowdown1} amounts to requiring that \(s\Xi'(s)=\Lambda'(\log s)\) have the form
    \[
    \Lambda'(\log s)=\exp\bigg\{ \eta_0 + \int_T^s \frac{\epsilon(u)}{u}\,\dd u \bigg\}
    \qquad \text{for } s\geq T,
    \]
    where \(T>0\), \(\eta_0 \in \mathbb R\), and \(\epsilon\in L^\infty((T,\infty))\) with \(\epsilon(s)\to0\) as \(s\to\infty\). Returning to logarithmic scale, this becomes
    \[
    \Lambda'(x)=\exp\bigg\{\eta_0+\int_{\log T}^{x}\epsilon(e^\tau)\,\dd \tau\bigg\}.
    \]

    Consequently, it is natural to try to construct, for some large \(X>0\),
    \[
    \begin{dcases}
        \Lambda(x)=\Lambda(X)+\int_X^x\Lambda'(\tau)\,\dd \tau,\\
        \Lambda'(x)=Ce^{\phi(x)},
    \end{dcases}
    \]
    where \(C>0\) is a constant and \(\phi:[X,\infty)\to\mathbb R\) is Lipschitz with
    \[
    \operatornamewithlimits{ess\,lim}_{\tau\to\infty}\phi'(\tau)=0.
    \]
    The function \(\phi\) must satisfy a delicate balance: if \(\widetilde\Psi(x)\coloneqq \Psi_0(e^x)\), then
    \begin{itemize}
        \item \(\phi\) must be sufficiently negative so that \(\Lambda(x)\leq \widetilde\Psi(x)\) for \(x\geq X\) and \(\phi(x)\to-\infty\), in order that \(\Lambda'(x)\to0\); but
        \item \(\phi\) cannot be too negative, since we also need \(\Lambda(x)\to\infty\).
    \end{itemize}
    Reconciling these two requirements is the most delicate point of the argument. We shall achieve it by constructing \(\phi\) block by block in logarithmic scale. More precisely, for some sufficiently large \(X\), we decompose
    \begin{equation}
       [X,\infty)=\bigcup_{j=1}^\infty [X_j,X_{j+1}),
       \label{3:decomp}
    \end{equation}
    with \(X=X_1\), and on each inductively defined block \(B_j\coloneqq [X_j,X_{j+1})\) we define an affine function \(\phi\) so that all the required properties hold. Once \(\Xi(s)=\Lambda(\log s)\) has been defined for \(s\geq e^{X_1}\), we then extend it in a \(C^1\) fashion to the whole half-line \([0,\infty)\).

    \textit{Step 1: the inductive definition of the blocks and of \(\Lambda\).}
    Let us now carry out the program outlined above. Keep in mind that
    \begin{equation}
        \widetilde \Psi(x) \nearrow \infty \text{ as } x \nearrow \infty
        \label{3:limwtA}
    \end{equation}
    by assumption, since \(\widetilde \Psi(x)=\Psi_0(e^x)\).

    To begin, choose \(X_1\in\mathbb R\) so large that \(\widetilde \Psi(X_1)\geq 2\), and set \(\Lambda'_1=2\). Next choose \(X_2>X_1\) so that, writing \(\Delta_1=X_2-X_1\),
    \begin{equation}
    \begin{dcases}
        \widetilde \Psi(X_2)\geq \widetilde \Psi(X_1) + 1, \text{ and}\\
        \Lambda'_1\Delta_1\geq 1.
    \end{dcases}
    \label{3:cond1}
    \end{equation}
    On the first block \([X_1,X_2]\), define
    \begin{equation}
        \Lambda'(x)\coloneqq \Lambda'_1 \exp \{-2 \Lambda'_1 (x-X_1)\}
        \qquad \text{for } x\in [X_1,X_2].
        \label{3:K'1}
    \end{equation}

    In general, suppose that for some \(j\geq 2\), the points \(X_1,\dots,X_j\) have already been chosen and that \(\Lambda'(x)\) has been defined as a positive continuous function on \([X_1,X_j]\). Let \(\Lambda'_j\coloneqq \Lambda'(X_j)\), and choose \(X_{j+1}>X_j\) so that, for \(\Delta_j\coloneqq X_{j+1}-X_j\),
    \begin{equation}
         \begin{dcases}
        \widetilde \Psi(X_{j+1}) \geq \widetilde\Psi(X_j)+1, \text{ and}\\
        \Lambda'_j \Delta_j \geq 1.
    \end{dcases}
    \label{3:condB_j}
    \end{equation}
    Then define
    \begin{equation}
    \Lambda'(x)=\Lambda'_j \exp\{-2\Lambda'_j(x-X_j)\}
    \qquad \text{for } x\in [X_j,X_{j+1}].
    \label{3:K'j}
    \end{equation}

    These choices are possible by virtue of \eqref{3:limwtA}. Moreover, it is clear that $\Lambda_j'$ is decreasing, so $\Delta_j \geq \frac{1}{2}$ and \(X_j\nearrow \infty\). Thus, \([X_1,\infty)\) is indeed exhausted by the blocks, and \eqref{3:decomp} holds. Notice also that \(\Lambda'(x)\) is continuous on \([X_1,\infty)\).

    Finally, define \(\Lambda(X_1)=1\), and set
    \begin{equation}
        \Lambda(x)=1+\int_{X_1}^x \Lambda'(\tau)\,\dd \tau
        \qquad \text{for } x\geq X_1.
        \label{3:defK}
    \end{equation}
    We now verify that \(\Lambda\) enjoys all the required properties.

    \textit{Step 2: monotonicity and limit of \(\Lambda'\).}
    It is clear that \(\Lambda'(x)\) is positive and decreasing for \(x\geq X_1\). Furthermore, by induction, \eqref{3:K'1} and \eqref{3:K'j} yield
    \[
    \Lambda'_j=\Lambda'(X_j)
    =\Lambda'_1 \exp \bigg\{-2\sum_{n=1}^{j-1} \Lambda'_n\Delta_n\bigg\}.
    \]
    Since \(\Lambda'_n\Delta_n\geq 1\) by \eqref{3:cond1} and \eqref{3:condB_j},
    \begin{equation}
        \Lambda'_j=\Lambda'(X_j)\leq 2 \exp \bigg\{-2\sum_{n=1}^{j-1}1\bigg\}\to 0
        \qquad \text{as } j\to\infty.
        \label{3:boundK'}
    \end{equation}
    Accordingly,
    \[
    \Lambda'(x)\searrow 0
    \qquad \text{as } x\nearrow \infty.
    \]

    \textit{Step 3: divergence of \(\Lambda\).}
    Since \(\Lambda'(x)>0\), the function \(\Lambda(x)\) is increasing. Moreover, \eqref{3:K'1} and \eqref{3:K'j} show that, for every \(j\geq 1\),
    \begin{align}
    \Lambda(X_{j+1})-\Lambda(X_j)
    &= \int_{X_j}^{X_{j+1}} \Lambda'(\tau)\,\dd \tau \nonumber\\
    &= \int_{0}^{\Delta_j} \Lambda'_j e^{-2\Lambda'_j\tau}\,\dd \tau \nonumber\\
    &= \frac{1}{2}(1 - e^{-2 \Lambda'_j  \Delta_j}).
    \label{3:Kincrease}
    \end{align}
    In particular, since \(\Lambda(X_1)=1\) by \eqref{3:defK} and \(\Lambda'_j\Delta_j\geq 1\),
    \[
    \Lambda(X_j)
    = 1 + \frac{1}{2} \sum_{n=1}^{j-1} \bigl(1 - e^{-2\Lambda'_n \Delta_n}\bigr)
    \geq 1 + \frac{1}{2} \sum_{n=1}^{j-1}\bigl(1-e^{-2}\bigr)
    \geq cj
    \qquad \text{as } j\to\infty,
    \]
    for some constant \(c>0\). Therefore,
    \[
    \Lambda(x)\nearrow\infty
    \qquad \text{as } x\nearrow\infty.
    \]

    \textit{Step 4: the dominance \(\Lambda\leq \widetilde \Psi\).}
    In order to show that \(\Lambda(x)\leq \widetilde \Psi(x)\) for \(x\geq X_1\), we claim that
    \begin{equation}
        \Lambda(X_{j+1}) \leq \widetilde \Psi(X_j)
        \qquad \text{for every } j\geq 1.
        \label{3:ineqKA}
    \end{equation}

    The verification of \eqref{3:ineqKA} uses the crude bound, obtained from \eqref{3:Kincrease},
    \begin{equation}
        \Lambda(X_{j+1}) - \Lambda(X_j) =  \frac{1}{2}(1 - e^{-2 \Lambda'_j  \Delta_j}) \leq \frac{1}{2}.
        \label{3:Kincrease2}
    \end{equation}
    Thus \(\Lambda(X_2)\leq \Lambda(X_1)+\frac{1}{2}=\frac{3}{2}\), while \(\widetilde \Psi(X_1)\geq 2\) by construction. In general, assuming the claim holds up to index \(j\), we obtain
    \[
    \Lambda(X_{j+1})
    \leq \Lambda(X_j)+\frac{1}{2}
    \leq \widetilde \Psi(X_{j-1})+\frac{1}{2}
    \leq \widetilde \Psi(X_j)-\frac{1}{2},
    \]
    by \eqref{3:condB_j}. This proves \eqref{3:ineqKA}.

    Consequently, for any \(x\in [X_j,X_{j+1}]\),
    \[
    \Lambda(x)\leq \Lambda(X_{j+1})\leq \widetilde \Psi(X_j)\leq \widetilde \Psi(x),
    \]
    and therefore
    \begin{equation}
        \Lambda(x)\leq \widetilde \Psi(x)
        \qquad \text{for all } x\geq X_1.
        \label{3:dominance}
    \end{equation}

    \textit{Step 5: definition of \(\Xi\) and regular variation of \(\Xi'\).}
    Set \(T_1=e^{X_1}\), and define \(\Xi(s)\coloneqq \Lambda(\log s)\) for \(s\geq T_1\). Recall that
    \[
    s\Xi'(s)=\Lambda'(\log s).
    \]
    For \(0\leq s\leq T_1\), define instead
    \[
    \Xi(s)=\bigg(\frac{s}{T_1}\bigg)^2.
    \]
    Observe that
    \[
    \begin{dcases}
        \displaystyle{\lim_{s\rightarrow T_1^-}}\Xi'(s) = \frac{2}{T_1},\\
        \displaystyle{\lim_{s\rightarrow T_1^+}}\Xi'(s) = \frac{1}{T_1}\Lambda'(\log T_1) = \frac{1}{T_1}\Lambda'_1 = \frac{2}{T_1}.
    \end{dcases}
    \]
    Since also
    \[
    \lim_{s\rightarrow T_1^-}\Xi(s)=1
    \qquad\text{and}\qquad
    \lim_{s\rightarrow T_1^+}\Xi(s)=\Lambda(X_1)=1,
    \]
    it follows that \(\Xi\in C^1([0,\infty))\). Moreover, for \(0<s<T_1\), one has \(s\Xi'(s)=2\Xi(s)\), so that this quantity is increasing near the origin as well.

    Hence, in view of the preceding steps, properties (i), (ii), and (iv) in the statement are already verified. It remains to prove (iii), namely that \(s\Xi'(s) \in \operatorname {RV}_{0}(\infty)\), for, by construction, $s\Xi'(s)$ will be normalized. Returning to logarithmic scale \(s\Xi'(s)=\Lambda'(\log s)\), the proof of the slow variation of $s \Xi'(s)$  amounts to establishing that
    \begin{equation}
       \lim_{x \to \infty} \frac{\Lambda' (x+c)}{\Lambda' (x)} = 1
       \label{3:finalgoal}
    \end{equation}
    for every \(c \in \mathbb R\). Fix such a \(c\).

    Set \(\phi(x)\coloneqq \log \Lambda'(x)\). Then \eqref{3:K'1} and \eqref{3:K'j} imply that
    \[
    \phi(x)=\log \Lambda'_j - 2\Lambda'_j (x- X_j)
    \qquad \text{for } x \in [X_j, X_{j+1}].
    \]
    Since \(\phi\) is continuous, it follows that, in the sense of distributions on \(\mathcal D'([X_1,\infty))\), \(\phi'\) is the step function given by
    \[
    \phi'(x) = -2\Lambda'_j
    \qquad \text{for } x \in [X_j, X_{j+1}].
    \]
    As a consequence of \eqref{3:boundK'},
    \[
    \operatornamewithlimits{ess\, lim}_{x \to \infty} \phi'(x) = 0.
    \]

    We may therefore write
    \[
    \log\frac{\Lambda'(x+c)}{\Lambda'(x)} = \phi(x+c)-\phi(x)=\int_x^{x+c}\phi'(s)\,\dd s,
    \]
    whence
    \[
    \bigg| \log\frac{\Lambda'(x+c)}{\Lambda'(x)}\bigg|
    \leq
    |c|\, \operatornamewithlimits{ess\,sup}_{s \geq x - |c|} |\phi'(s)|
    =o(1)
    \qquad \text{as } x\to\infty.
    \]
    This proves \eqref{3:finalgoal}, and hence establishes (iii).

    The proof is complete.
\end{proof}

Combining \Cref{quantitative} and \Cref{slowdown1}, we obtain \Cref{main}.

\begin{proof}[Proof of \Cref{lusin}]
    We heavily utilize the results of \Cref{main}: for every $\epsilon>0$ there exists a measurable set $U_{\epsilon,R}\subset B_R$ such that $\mathcal{L}^d(B_R\setminus U_{\epsilon,R})<\epsilon/2$, a function $k\in L^1((0,T)\times U_{\epsilon,R})$, and for every $x,y \in U_{\epsilon,R}$ and all $0\leq s\leq t \leq T$, it holds the inequality
    \[|\boldsymbol{X}(t,0,x)-\boldsymbol{X}(t,0,y)|
    \leq
    G^{-1}\left(G\bigl(|\boldsymbol{X}(s,0,x)-\boldsymbol{X}(s,0,y)|\bigr)
    +
    \int_s^t\bigl(k(\tau,x)+k(\tau,y)\bigr)\,\dd \tau\right).\]
    In the latter, we have used that $G$ has an inverse function for it is increasing.
    
By Chebyshev's inequality, the set
\[
W_{\epsilon,R}
:=
\left\{
x\in U_{\epsilon,R}:
\int_0^T k(\tau,x)\,\dd\tau
\leq
\frac{2}{\epsilon}
\|k\|_{L^1((0,T)\times U_{\epsilon,R})}
\right\}
\]
satisfies
\[
\mathcal L^d(U_{\epsilon,R}\setminus W_{\epsilon,R})
\leq
\frac{\epsilon}{2}.
\]
Therefore, setting
\(\Omega_{\epsilon,R}:=W_{\epsilon,R}\), we have
\[
\mathcal L^d(B_R\setminus\Omega_{\epsilon,R})
\leq
\mathcal L^d(B_R\setminus U_{\epsilon,R})
+
\mathcal L^d(U_{\epsilon,R}\setminus W_{\epsilon,R})
<
\epsilon.
\]
Defining
\begin{equation}
    \omega^\epsilon(r)
    :=
    G^{-1}\left(
        G(r)
        +
        \frac{4}{\epsilon}
        \|k\|_{L^1((0,T)\times U_{\epsilon,R})}
    \right),
    \label{definitionomega}
\end{equation}
we obtain the desired conclusion.
\end{proof}

\section{Examples and counterexamples}\label{section3}

\subsection{Examples of moduli of continuity in \texorpdfstring{\Cref{lusin}}{}}

We now prove \Cref{examples}.

\begin{proof}[Proof of \Cref{examples}]
The argument is entirely computational. In each assertion of \Cref{examples}, a superlinear function \(\Phi_{j}\) as in \Cref{regularityforPhi} is implicitly present. The task is to determine the corresponding weight \(g_j(r)\) via \eqref{definitiong}, the associated function \(G_j\) via \eqref{definitionG}, and finally the modulus
\[
\omega_j^\epsilon(z)=G_j^{-1}\bigl(G_j(z)+C_{\epsilon,\lambda,d}\bigr),
\]
where $C_{\epsilon, \lambda, d} \to \infty$ as $\epsilon \to 0_+.$

In general, the exact form of \(\omega_j^\epsilon\) is difficult to compute explicitly. We therefore derive asymptotic estimates as \(z\to0_+\). Notice that \(\omega_j^\epsilon(z)>z\) and \(\omega_j^\epsilon(z)\to0\) as \(z\to0_+\). Hence
\begin{equation}
    \int_{\log z}^{\log \omega_j^\epsilon(z)} \frac{\dd v}{g_j(e^v)}
    =
    \int_z^{\omega_j^\epsilon(z)} \frac{\dd u}{u g_j(u)}
    =
    G_j(\omega_j^\epsilon(z)) - G_j(z)
    =
    C_{\epsilon,\lambda,d}.
    \label{formulaomega0}
\end{equation}
By the mean value theorem, there exists \(\xi=\xi(z)\in(\log z,\log \omega_j^\epsilon(z))\) such that
\[
\log\bigg(\frac{\omega_j^\epsilon(z)}{z}\bigg)
=
C_{\epsilon,\lambda,d}\, g_j(e^\xi).
\]
Since \(g_j(u)\) is decreasing as \(u\to0_+\), we infer that
\[
\log\bigg(\frac{\omega_j^\epsilon(z)}{z}\bigg)
\leq
C_{\epsilon,\lambda,d}\, g_j(z)
\qquad \text{as } z\to0_+,
\]
or equivalently,
\begin{equation}
    \omega_j^\epsilon(z)\leq z\,\exp \bigl(C_{\epsilon,\lambda,d}\,g_j(z)\bigr)
    \qquad \text{as } z\to0_+.
    \label{formulaomega}
\end{equation}

Before analyzing the various cases appearing in statement \Cref{examples}, let us record two general observations. First, each function
\[
\Psi_{0,j}(z)=z^{-1}\Phi_{0,j}(z)
\]
in the cases below is such that $z\Psi_{0,j}'(z)$ is a positive, eventually decreasing, and normalized slowly varying function at \(\infty\). Consequently, according to \Cref{onPhiandPhi0}, if we modify $\Psi_{0,j}(z)$ near $z=0$, we derive a $\Phi_j(z)$ for which the conditions of \Cref{regularityforPhi} are all satisfied.  Second, in view of \eqref{definitiong} and \eqref{definitionG}, the interesting regime of \(\omega_j^\epsilon(z)\) depends only on the behavior of \(\Phi_j(z)\) for large \(z\). Therefore, in what follows, we shall therefore only be concerned with the behavior of the functions \(\Phi_{0,j}(z) = \Phi_j(z)\) for large \(z\).

\textit{Step 1: the \(L\log^b L\) case.}
Here
\[
\Phi_1(z)=z\log^b z
\qquad \text{for all sufficiently large } z.
\]
Hence, for \(t\) sufficiently large,
\[
g_1\left(t^{-1/d}\right)\sim \frac{1}{b}\log^{1-b}(t),
\]
so that
\[
g_1(z)\sim C_{b,d}\,\log^{1-b}(1/z)
\qquad \text{as } z\to0_+.
\]
Therefore,
\[
\omega_1^\epsilon(z)\leq z\,\exp\bigl(2C_{b,d,\epsilon}\log^{1-b}(1/z)\bigr)
\qquad \text{as } z\to0_+.
\]
Since \(\log^{1-b}(1/z)=o(\log(1/z))\) as \(z\to0_+\), it follows that for every \(0<\alpha<1\),
\[
\omega_1^\epsilon(z)\leq z^\alpha
\qquad \text{for all sufficiently small } z>0.
\]

\textit{Step 2: the \(L(\log/\log^{(n)})L\) case.}
Let us now consider
\[
\Phi_2(z)=\frac{z\log z}{\log^{(n)}z}
\qquad \text{for all sufficiently large } z.
\]

We begin with the case \(n=2\). In this case,
\[
g_2\left(t^{-1/d}\right)\sim \frac{(\log\log t)^2}{\log\log t-1}
\qquad \text{as } t\to\infty,
\]
and hence
\[
g_2(z)\sim \log\log(1/z)
\qquad \text{as } z\to0_+.
\]
Therefore,
\[
\omega_2^\epsilon(z)\leq z\,\log(1/z)^{2C_{\epsilon,\lambda,d}}
\qquad \text{as } z\to0_+.
\]

For \(n\geq3\), the computation is analogous. One finds
\[
g_2(t^{-1/d})\sim \log^{(n)} t
\qquad \text{as } t\to\infty,
\]
and hence
\[
g_2(z)\sim \log^{(n)}(1/z)
\qquad \text{as } z\to0_+.
\]
Consequently,
\[
\omega_2^\epsilon(z)\leq z\,(\log^{(n-1)}(1/z))^{2C_{\epsilon,\lambda,d}}
\qquad \text{as } z\to0_+.
\]

\textit{Step 3: the \(L\log\log L\) case.}
Here
\[
\Phi_3(z)=z\log\log z
\qquad \text{for all sufficiently large } z.
\]
Thus, for \(t\) sufficiently large,
\[
g_3\left(t^{-1/d}\right)\sim \log t,
\]
which means that
\[
g_3(z)\sim \log(1/z)
\qquad \text{as } z\to0_+.
\]
In this case, rather than using \eqref{formulaomega}, it is more convenient to use \eqref{formulaomega0}, which gives
\[
\log\log(1/z)-\log\log(1/\omega_3^\epsilon(z))
\sim C_{\epsilon,\lambda,d}
\qquad \text{as } z\to0_+.
\]
After rearranging, we obtain
\[
 \omega^\epsilon_3(z)\leq z^{\alpha_\epsilon}
\qquad \text{as } z\to0_+,
\]
where
\[
\alpha_\epsilon\coloneqq \exp(-2C_{\epsilon,\lambda,d}).
\]
Since \(C_{\epsilon,\lambda,d}\to\infty\) as \(\epsilon\to0\), the conclusion follows.

\textit{Step 4: the \(L\log^{(n)}L\) case.}
Finally, consider
\[
\Phi_4(z)=z\log^{(n)}z
\qquad \text{for all sufficiently large } z.
\]
The computation is again analogous and yields
\[
\log^{(n)}(1/z)-\log^{(n)}(1/\omega^\epsilon_4(z))
\leq 2C_{\epsilon,\lambda}
\qquad \text{as } z\to0_+.
\]
Using the identity \(\log^{(n)}z=\log(\log^{(n-1)}z)\), we obtain
\[
\log\left(\frac{\log^{(n-1)}(1/z)}{\log^{(n-1)}(1/\omega^\epsilon_4(z))}\right)
\leq 2C_{\epsilon,\lambda}
\qquad \text{as } z\to0_+.
\]
This gives the desired estimate, with
\(
\alpha_\epsilon=\exp(-2C_{\epsilon,\lambda}).
\)
\end{proof}

\begin{remark}\label{Phi5}
\textnormal{Some remarks are in order.}
\begin{itemize}
    \item \textnormal{(A lower bound for \(\omega^\epsilon(z)\)). Notice that \eqref{formulaomega0} also yields
\[
\log\bigg(\frac{\omega_j^\epsilon(z)}{z}\bigg)
\geq C_{\epsilon,\lambda,d}\, g_j(\omega_j^\epsilon(z))
\qquad \text{as } z\to0_+,
\]
which may be used to derive lower bounds for \(\omega_j^\epsilon\), if desired.}

\textnormal{For instance, in the \(L\log^b L\) case, this shows that
\[
\lim_{z\to0_+}\frac{\omega_1^\epsilon(z)}{z}=\infty,
\]
so that the resulting modulus of continuity does not belong to the Lipschitz regime.}

\textnormal{Furthermore, in the \(L(\log/\log^{(n)})L\) case, one obtains
\begin{equation}
    \frac{\omega_2^\epsilon(z)}
    {\log^{(n-1)}(1/\omega_2^\epsilon(z))^{C_{\epsilon,\lambda,d}/2}}
    \geq z
    \qquad \text{as } z\to0_+.
    \label{lower}
\end{equation}
Combined with the upper bound already obtained for \(\omega_2^\epsilon\), this implies that
\begin{equation}
    \omega_2^\epsilon(z)\geq c\, z\,\log^{(n-1)}(1/z)^{C_{\epsilon,\lambda,d}/2}
    \qquad \text{as } z\to0_+,
    \label{lower0}
\end{equation}
for some constant \(c>0\).}

\textnormal{Similar lower bounds can be obtained in $L \log \log L$ and $L \log^{(n)} L$ cases.}

\item \textnormal{(A modulus of continuity that does not satisfy the Dini condition).
Although all the explicit moduli in \Cref{examples} satisfy the Dini condition
\[
\int_0^1\frac{\omega(s)}{s}\,\dd s<\infty,
\]
this property is not guaranteed by our theory in general.}

\textnormal{To see this, let \(\slog\) denote the Kneser analytic superlogarithm, namely the inverse of the tetration \({}^{x}e = \sexp(x)\); see \cite{Kneser1950}. That being so, it satisfies the functional identity
\[\slog(e^z)=\slog(z)+1,\] which implies that $\slog$ is a normalized slowly varying function at \(\infty\). Indeed, given \(x>1\), choose \(m=m(x)\) so that
\(u=\log^{(m)}x\in[1,e)\). Iterating the functional identity gives
\(\slog(x)=m+\slog(u)\), and hence \(\slog(x)\to\infty\) as
\(x\to\infty\). Differentiating the same identity yields
\[
\slog'(x)
=
\frac{\slog'(u)}
{x\log x\,\log^{(2)}x\cdots\log^{(m-1)}x}.
\]
Since \(u\in[1,e)\), it follows that \(x\slog'(x)\) is bounded and therefore
\[
\frac{x\slog'(x)}{\slog(x)}\to0
\qquad\text{as }x\to\infty.
\]
This shows that \(\slog\) is normalized slowly varying at infinity. The same functional relation, together with a direct differentiation, shows that the function introduced below satisfies the regularity requirements of \Cref{regularityforPhi}.}

\textnormal{Consider, for sufficiently large \(s\), the Young function
\[
\Phi_{5,0}(s)
=
s\,\slog_+\bigl(\log_+^{(2)}s\bigr),
\]
modified near the origin if necessary, and let
\(\Psi_5(s)=s^{-1}\Phi_{5,0}(s)\). Denote by
\(C_{\epsilon,\lambda,d}\) the constant appearing in the definition of
\(\omega_5^\epsilon\). From \eqref{definitionG} and
\eqref{definitionomega}, one has
\[
\Psi_5\bigl((\omega_5^\epsilon(z))^{-d}\bigr)
=
\Psi_5(z^{-d})-dC_{\epsilon,\lambda,d}.
\]
Since \(C_{\epsilon,\lambda,d}\to\infty\) as \(\epsilon\to0_+\), we may choose \(\epsilon>0\) so small that
\(dC_{\epsilon,\lambda,d}\geq2\). By monotonicity, \(\omega_5^\epsilon\) is then bounded from below by the modulus corresponding to a shift equal to \(2\).}

\textnormal{To compute \(\omega_5^\epsilon\), let us recall that
\(\sexp(a-n)=\log^{(n)}(\sexp(a))\). As a result,
\[
\omega_5^\epsilon(z)
\geq
\left[
\exp^{(2)}
\sexp\bigl(
\slog(\log^{(2)}(z^{-d}))-2
\bigr)
\right]^{-1/d}
=
\bigl[\log d+\log^{(2)}(z^{-1})\bigr]^{-1/d}.
\]
Consequently, for every \(\delta>0\),
\[
\int_0^\delta\frac{\omega_5^\epsilon(z)}{z}\,\dd z
\geq
\int_{|\log\delta|}^\infty
\frac{\dd u}{(\log d+\log u)^{1/d}}
=
\infty.
\]
Thus, the modulus of continuity furnished by \Cref{lusin} need not satisfy the Dini condition.}

\textnormal{More generally, let \(\Phi\) be another admissible function, set
\(\Psi(s)=s^{-1}\Phi(s)\), and assume that
\[
\Psi'(s)=o\bigl(\Psi_5'(s)\bigr)
\qquad\text{as }s\to\infty.
\]
Then \Cref{monotone} below shows that the modulus generated by \(\Phi\) is asymptotically larger than \(\omega_5^\epsilon\), and therefore also fails to satisfy the Dini condition. Finally, the sharpness example in \Cref{exampleprop} shows that, in general, $\boldsymbol{X}(t,0,\cdot)$ is not Dini continuous in the ``good regions'' $\Omega_{\epsilon, R}$.}

\end{itemize}
\end{remark}

\begin{proposition}[Comparison of the induced moduli of continuity]\label{monotone} 
Let \(\Phi_1,\Phi_2:[0,\infty)\to[0,\infty)\) be \(C^1\) functions satisfying \eqref{normalization} and \eqref{superlinearity}. Assume that \begin{equation} \left(\frac{\Phi_2(s)}{s}\right)' = o\!\left(\left(\frac{\Phi_1(s)}{s}\right)'\right) \qquad \text{as } s\to\infty. \label{condPhi} \end{equation}

For \(j=1,2\), set \( \Psi_j(s)\coloneqq s^{-1}\Phi_j(s), \) and assume that the function \(s\mapsto s\Psi_j'(s)\) is positive on \((0,\infty)\) and bounded on $[1,\infty)$. Let \(g_j\), \(G_j\), and \(\omega_j^{\epsilon_j}\) be defined, respectively, by \eqref{definitiong}, \eqref{definitionG}, and \eqref{definitionomega}, where \(\epsilon_j>0\). 

Then, for every \(\epsilon_1,\epsilon_2>0\), \[ \frac{\omega_1^{\epsilon_1}(r)}{\omega_2^{\epsilon_2}(r)} \to 0 \qquad \text{as } r\to 0_+. \] 
\end{proposition}

In other words, the moduli of continuity generated by the theory from \(\Phi_1\) are ``much better'' than those generated from \(\Phi_2\).

\begin{proof}
    Since the modulus of continuity provided by \Cref{lusin} is of the form \[\omega_j^{\epsilon_j}(r) = G_j^{-1}(G_j(r) + \delta_j)\] for some $\delta_j>0$, the desired conclusion will be derived once we show that, for any \(\delta_1,\delta_2>0\),
    \[
    \lim_{r \to 0_+}\frac{G_1^{-1}(\delta_1 + G_1(r))}{G_2^{-1}(\delta_2 + G_2(r))}=0.
    \]
    To this end, for \(j=1,2\), define
    \[
    \ell_j^{\delta_j}(r)\coloneqq G_j^{-1}(\delta_j+G_j(r)).
    \]
    Then \(\ell_j^{\delta_j}(r)>r\), and
    \[
    \int_r^{\ell_j^{\delta_j}(r)} \frac{\dd u}{u g_j(u)}
    =
    G_j(\ell_j^{\delta_j}(r))-G_j(r)
    =
    \delta_j.
    \]

    Since the Osgood condition \eqref{osgoodcond} holds, we have \(\ell_j^{\delta_j}(r)\to0\) as \(r\to0_+\). On the other hand, by the definition of \(g_j\) through \eqref{definitiong}, the assumption \eqref{condPhi}
    implies that \(g_1=o(g_2)\) as \(r\to0_+\). Equivalently,
    \[
    \frac{1}{u g_1(u)} \gg \frac{1}{u g_2(u)}
    \qquad \text{as } u\to0_+.
    \]
    More precisely, given any \(\kappa>0\), there exists \(0<r_0<1\) such that
    \[
    \frac{1}{u g_1(u)}>\frac{1}{\kappa}\frac{1}{u g_2(u)}
    \]
    whenever \(0<r<r_0\) and \(r<u<\ell_2^{\delta_2}(r)\).

    Therefore, for such \(r\),
    \[
    \int_r^{\ell_1^{\delta_1}(r)} \frac{\dd u}{u g_1(u)}=\delta_1
    \qquad\text{and}\qquad
    \int_r^{\ell_2^{\delta_2}(r)} \frac{\dd u}{u g_1(u)}
    >
    \frac{\delta_2}{\kappa}.
    \]
    In particular, if \(\kappa<\delta_2/\delta_1\), then for all sufficiently small \(r\) one has
    \(
    \ell_2^{\delta_2}(r)>\ell_1^{\delta_1}(r).
    \)
    Moreover,
    \[
    \int_{\ell_1^{\delta_1}(r)}^{\ell_2^{\delta_2}(r)} \frac{\dd u}{u g_1(u)}
    \geq
    \frac{\delta_2}{\kappa}-\delta_1.
    \]

    Now let
    \(
    \gamma=\inf_{(0,1)} g_1 >0.
    \)
    Then,
    \[
    \int_{\ell_1^{\delta_1}(r)}^{\ell_2^{\delta_2}(r)} \frac{\dd u}{u g_1(u)}
    \leq
    \frac{1}{\gamma}\int_{\ell_1^{\delta_1}(r)}^{\ell_2^{\delta_2}(r)}\frac{\dd u}{u}
    =
    \frac{1}{\gamma}\log\!\left(\frac{\ell_2^{\delta_2}(r)}{\ell_1^{\delta_1}(r)}\right).
    \]
    Hence
    \[
    \log\!\left(\frac{\ell_2^{\delta_2}(r)}{\ell_1^{\delta_1}(r)}\right)
    \geq
    \gamma\left(\frac{\delta_2}{\kappa}-\delta_1\right)
    \qquad \text{for } 0<r<r_0.
    \]
    It follows that
    \[
    \liminf_{r\to0_+}\log\!\left(\frac{\ell_2^{\delta_2}(r)}{\ell_1^{\delta_1}(r)}\right)
    \geq
    \gamma\left(\frac{\delta_2}{\kappa}-\delta_1\right).
    \]
    Since \(\kappa>0\) is arbitrary, we conclude that
    \[
    \lim_{r\to0_+}\log\!\left(\frac{\ell_2^{\delta_2}(r)}{\ell_1^{\delta_1}(r)}\right)=\infty,
    \]
    and therefore
    \[
    \frac{\ell_1^{\delta_1}(r)}{\ell_2^{\delta_2}(r)}\to0
    \qquad\text{as } r\to0_+,
    \]
    as we wanted to show.
\end{proof}

\subsection{An example of sharpness of \texorpdfstring{\Cref{lusin}}{}}

We now prove the optimality of the moduli of continuity on the class of regularly varying functions described in \Cref{examples}.

\begin{proof}[Proof of \Cref{examples}]
Before turning to the details, let us briefly explain the strategy. Consider any continuous increasing function \(\Phi:[0,\infty)\to[0,\infty)\) satisfying \eqref{normalization}, \eqref{superlinearity}, and \Cref{regularityforPhi}. 
Let \(g\) and \(G\) be defined, respectively, by \eqref{definitiong} and \eqref{definitionG}. We then consider the ordinary differential equation
\begin{equation}
  \frac{\dd \boldsymbol{X}}{\dd t}(t,0,x) = \boldsymbol{b}(\boldsymbol{X}(t,0,x)),
  \label{4.ode}
\end{equation}
where the autonomous vector field \(\boldsymbol{b} : \mathbb R^d \to \mathbb R^d\) is given by
\begin{equation}
    \boldsymbol{b}(x)=
    \begin{dcases}
        g(|x|)\eta(|x|)\,x
        =\dfrac{\eta(|x|)}{|x|^{-d-1}\Psi'(|x|^{-d})}\,\dfrac{x}{|x|}
        & \text{if } x\neq0,\\[1ex]
        0 & \text{if } x=0,
    \end{dcases}
\end{equation}
and \(\eta\) is a smooth cut-off function supported near the origin. Clearly,
\[
\boldsymbol{b}\in C^1(\mathbb R^d\setminus\{0\})\cap C_c(\mathbb R^d),
\]
so that \eqref{4.ode} has a unique solution for every nonzero initial datum. Moreover, \(\boldsymbol{b}\) is radial, and therefore the corresponding flow \(\boldsymbol{X}(t,0,x)\) is radial as well. Since \(g(r)\) satisfies the Osgood condition \eqref{osgoodcond}, the solution of \eqref{4.ode} exists and is unique globally for any initial datum.

Accordingly, \(\boldsymbol{X}(t,0,x)\) has the form
\[
\boldsymbol X(t,0,x) =
\begin{dcases}
    \rho(t,|x|)\dfrac{x}{|x|} & \text{if } x\neq0,\\[1ex]
    0 & \text{if } x=0,
\end{dcases}
\]
where
\(
\rho\in C^1([0,\infty)\times(0,\infty))\cap C([0,\infty)^2)
\)
solves the one-dimensional equation
\[
\frac{\dd \rho}{\dd t}(t,r)=\boldsymbol{c}(\rho(t,r)),
\]
with
\[
\boldsymbol{c}(\rho)=
\begin{dcases}
    g(\rho)\eta(\rho)\rho & \text{if } \rho>0,\\
    0 & \text{if } \rho=0.
\end{dcases}
\]
Thus, for \(r\) and \(t\) sufficiently small,
\[
\int_r^{\rho(t,r)} \frac{\dd u}{u g(u)} = t.
\]
In other words,
\[
G(|\boldsymbol{X}(t,0,x)-\boldsymbol{X}(t,0,0)|)=G(|x - 0|)+t
\]
for \(|x|\) sufficiently small and \(t>0\) sufficiently small (both depending on $\delta$).

Therefore, in order to establish the desired optimality statement, it remains to check that \(\boldsymbol{b}\) satisfies the DiPerna--Lions assumptions \eqref{regularity}, \eqref{growth}, and \eqref{compressibility}, that \(\Phi(D\boldsymbol{b})\in L^1_{\loc}(\mathbb R^d)\), and that
\[
\operatorname{M}^g_\infty(D\boldsymbol{b})\in L^\infty(\mathbb R^d).
\]
Once this is done, the conclusion follows from the definition of the sets \(\Omega_{\epsilon,R}\) and the Chebyschev inequality used in \Cref{lusin}.

To fix notation, we shall assume that
\[
\eta(\rho)=\phi\bigg(\frac{\rho}{\delta}\bigg),
\]
where \(\delta>0\), \(\phi\in C_c^\infty(\mathbb R)\) satisfies \(\phi(s)=1\) for \(s<1/4\), \(\phi(s)=0\) for \(s>1\), and
\[
|s\phi'(s)|\leq \frac32
\qquad\text{for all } s.
\]
The choice of \(\delta>0\) is not very material, but it will be discussed throughout the proof. Furthermore, for \(x\neq0\), let us write
\[
r=|x|,
\qquad
e_r=\frac{x}{|x|},
\]
and \(\partial_r\) for the radial derivative.

\textit{Step 1: verification of \eqref{regularity}.}
Since \(\boldsymbol{b}\in C_c(B_\delta)\), it is immediate that
\(
\boldsymbol{b}\in L^1(\mathbb R^d)\cap L^\infty(\mathbb R^d).
\)
Moreover, for \(x\neq0\),
\begin{equation}
    D \boldsymbol{b}(x) = \boldsymbol{c}'(r)\, e_r \otimes e_r + \frac{\boldsymbol{c}(r)}{r}\bigl(\mathbb I_d - e_r\otimes e_r\bigr),
    \label{Dc}
\end{equation}
where \(\mathbb I_d\) denotes the identity on \(\mathbb R^d\). Hence \(\boldsymbol{b}\in W^{1,1}(\mathbb R^d)\) as soon as one shows that
\[
\int_0^\delta |\boldsymbol{c}'(r)|\,r^{d-1}\,\dd r<\infty
\qquad\text{and}\qquad
\int_0^\delta \frac{|\boldsymbol{c}(r)|}{r}\,r^{d-1}\,\dd r<\infty.
\]

The second integral is simpler:
\[
\int_0^\delta \frac{|\boldsymbol{c}(r)|}{r}\,r^{d-1}\,\dd r
=
\int_0^\delta g(r)\eta(r)\,r^{d-1}\,\dd r,
\]
and this is finite because \(g\) is slowly varying at the origin and therefore, by Potter's bounds,
\[
g(r)=o(r^{-\epsilon})
\qquad\text{as } r\to0_+
\]
for every \(\epsilon\in(0,1)\).

For the first integral, write
\[
\boldsymbol{c}(r)=g(r)\eta(r)\,r.
\]
By Leibniz' rule, the only term requiring separate attention is
\[
\int_0^\delta |g'(r)|\,r\,\eta(r)\,r^{d-1}\,\dd r.
\]
Here we use in an essential way that \(g\) is normalized slowly varying at the origin: by \Cref{normalized0} and \Cref{equivalence},
\begin{equation}
    \lim_{r\to0_+}\frac{r g'(r)}{g(r)}=0.
    \label{normalizedlimit}
\end{equation}
Therefore,
\[
\int_0^\delta |g'(r)|\,r\,\eta(r)\,r^{d-1}\,\dd r
\leq
C_{g,\delta}\int_0^\delta g(r)\eta(r)\,r^{d-1}\,\dd r,
\]
and the right-hand side is finite by the previous argument. This proves that
\(
D\boldsymbol{b}\in L^1(\mathbb R^d).
\)

A closer inspection of the preceding estimates shows that
\begin{equation}
    |D\boldsymbol{b}(x)|=O(g(|x|))
    \qquad\text{as } x\to0.
    \label{Og}
\end{equation}

\textit{Step 2: verification of \eqref{growth}.}
The growth assumption \eqref{growth} is immediate, since \(\boldsymbol{b}\in L^\infty(\mathbb R^d)\).

\textit{Step 3: verification of \eqref{compressibility}.}
Recall that \(d\geq2\). Since \(\boldsymbol{b}(x)=\boldsymbol{c}(r)e_r\), one has
\begin{align}
    \Div \boldsymbol{b}(x) &=
\frac{1}{r^{d-1}}\frac{\partial}{\partial r}\bigl(r^{d-1}\boldsymbol{c}(r)\bigr) \nonumber\\ &= g(r)\bigg(\frac{r g'(r)}{g(r)}+1\bigg)\eta(r)
+
g(r)\big((d-1)\eta(r)+r\eta'(r)\big). \label{divergence}
\end{align}
Using again \eqref{normalizedlimit}, one sees that
\[
\Div \boldsymbol{b}_-\in C_c(\mathbb R^d)\subset L^\infty(\mathbb R^d).
\]

If one moreover wishes to ensure that \(\Div \boldsymbol{b}\geq0\) in \(\mathcal D'\), this can be achieved by choosing \(\delta>0\) sufficiently small (in spite of us eventually taking $\delta>0$ large). Indeed, let
$$
\begin{dcases}
    \operatorname{(I)} = g(r)\bigg(\frac{r g'(r)}{g(r)}+1\bigg)\eta(r) \text{ and} \\
    \operatorname{(II)} = g(r)\big((d-1)\eta(r)+r\eta'(r)\big),
\end{dcases}$$
so that \eqref{divergence} asserts that $\Div \boldsymbol{b}(x) = \operatorname{(I)} + \operatorname{(II)}.$ If \(\delta\) is sufficiently small, then \(\operatorname{(I)}\geq0\) by \eqref{normalizedlimit}. On the other hand, since \(|r\eta'(r)|\leq \frac32\eta(r)\), it follows that \(\operatorname{(II)}\geq0\) as well.

\textit{Step 4: verification that \(\Phi(D\boldsymbol{b})\in L^1(\mathbb R^d)\).}
Notice that \(\Psi\) is increasing and slowly varying at \(\infty\), as follows from L'Hôpital's rule. In view of \eqref{Dc}, the assertion \(\Phi(D\boldsymbol{b})\in L^1(\mathbb R^d)\) reduces to checking that
\begin{align}
    \operatorname{(I)} &= \int_0^\delta |\boldsymbol{c}'(r)|\,\Psi(|\boldsymbol{c}'(r)|)\,r^{d-1}\,\dd r < \infty,
    \label{I}\\
    \operatorname{(II)} &= \int_0^\delta \frac{|\boldsymbol{c}(r)|}{r}\,\Psi\bigg(\frac{|\boldsymbol{c}(r)|}{r}\bigg)\,r^{d-1}\,\dd r < \infty.
    \label{II}
\end{align}

We begin with \eqref{II}. The key observation is that if \(\ell_1\in \operatorname{RV}_0(\infty)\) and \(\ell_2\in \operatorname{RV}_0(0)\) are positive and \(\ell_2(0+)\) exists, then \(\ell_1\circ \ell_2\in \operatorname{RV}_0(0)\). This is immediate from the definition when \(\ell_1(0+)<\infty\), and otherwise follows from the uniform convergence theorem; see \cite[Theorem 1.5.2]{BGT}. Consequently,
\[
\frac{\boldsymbol{c}(r)}{r}\,\Psi\bigg(\frac{\boldsymbol{c}(r)}{r}\bigg)
\]
is slowly varying at \(0\), and the finiteness of \eqref{II} follows from Potter's bound.

By the monotonicity and slow variation of \(\Psi\), the same strategy applies to \eqref{I}. Indeed, using Leibniz' rule, one obtains
\[
\operatorname{(I)}
\leq
C_{\eta,\Psi}\bigg(
1
+
\int_0^\delta r|g'(r)|\,\Psi(r|g'(r)|)\,r^{d-1}\,\dd r
+
\int_0^\delta g(r)\Psi(g(r))\,r^{d-1}\,\dd r
\bigg).
\]
The last integral is finite by the previous argument, since \(g(r)\Psi(g(r))\) is slowly varying at \(0\). On the other hand, \eqref{normalizedlimit} gives
\[
\int_0^\delta r|g'(r)|\,\Psi(r|g'(r)|)\,r^{d-1}\,\dd r
\leq
C\int_0^\delta g(r)\Psi(Cg(r))\,r^{d-1}\,\dd r,
\]
which is of the same type. This proves \eqref{I}.

\textit{Step 5: analysis of \(M^\infty_g\).}
We now show that
\[
\operatorname{M}^g_\infty(D\boldsymbol{b})\in L^\infty(\mathbb R^d).
\]
Since \(\inf_{r>0} g(r)>0\) and \(D\boldsymbol{b}\in L^1(\mathbb R^d)\), it suffices to prove that
\begin{equation}
    \operatorname M^g_{R_0}(D\boldsymbol{b})\in L^\infty(\mathbb R^d),
    \label{goalM}
\end{equation}
where \(R_0>0\) is chosen so that
\begin{itemize}
    \item \(g\) is nonincreasing on \((0,3R_0)\); and
    \item \(|D\boldsymbol{b}(x)|\leq C_g g(|x|)\) for \(|x|<3R_0\), in view of \eqref{Og}.
\end{itemize}

We first prove that
\begin{equation}
    \operatorname M^g_{3R_0}(D\boldsymbol{b})(0)<\infty.
    \label{goal2}
\end{equation}
Indeed, for \(0<r<R_0\),
\begin{align*}
    \frac{1}{\mathcal L^d(B_r) g(r)} \int_{B_r} |D \boldsymbol{b}(x)|\,\dd x
    &\leq
    \frac{1}{r^d g(r)} \int_0^r |g'(s)| s^{d}\,\dd s
    +
    \frac{d}{r^d g(r)} \int_0^r g(s) s^{d-1}\,\dd s \\
    &= \operatorname{(I)}+\operatorname{(II)}.
\end{align*}
By \eqref{normalizedlimit}, the estimate of \(\operatorname{(I)}\) reduces to that of \(\operatorname{(II)}\). For \(\operatorname{(II)}\), Karamata's theorem (see, e.g., \cite[Proposition 1.5.10]{BGT}) gives
\[
\lim_{r\to0}\dfrac{\int_0^r s^{d-1}g(s)\,\dd s}{r^d g(r)}=\frac{1}{d}.
\]
This proves \eqref{goal2}.

Next we claim that \eqref{goal2} implies
\begin{equation}
    \operatorname M^g_{R_0}(D\boldsymbol{b})(x)\ \text{is bounded for } |x|<2R_0.
    \label{goal3}
\end{equation}
Indeed, fix \(|x|<2R_0\). If \(|x|/2\leq r<R_0\), then
\begin{align}
    \frac{1}{\mathcal L^d(B_r(x)) g(r)}&\int_{B_r(x)} |D\boldsymbol{b}(y)|\,\dd y \nonumber\\
    &\leq
    \frac{\mathcal L^d(B_{r+|x|})g(r+|x|)}{\mathcal L^d(B_r(x)) g(r)}
    \frac{1}{\mathcal L^d(B_{r+|x|})|g(r+|x|)}
    \int_{B_{r+|x|}(0)} |D\boldsymbol{b}(y)|\,\dd y \nonumber\\
    &\leq
    \frac{\mathcal L^d(B_{r+|x|})g(r+|x|)}{\mathcal L^d(B_r(x)) g(r)}\, M^{3R_0}_g(D\boldsymbol{b})(0) \nonumber\\
    &\leq
    C\, M^{3R_0}_g(D\boldsymbol{b})(0).
    \label{goal31}
\end{align}
On the other hand, if \(r<|x|/2<R_0\), then
\begin{align}
    \frac{1}{\mathcal L^d(B_r(x)) g(r)}\int_{B_r(x)} |D\boldsymbol{b}(y)|\,\dd y
    &\leq
    \frac{C_g}{\mathcal L^d(B_r(x)) g(r)} \int_{B_r(x)} g(|y|)\,\dd y \nonumber\\
    &\leq
    \frac{C_g}{\mathcal L^d(B_r(x)) g(r)} \int_{B_r(x)} g(|x|-r)\,\dd y \nonumber\\
    &\leq
    C_g \frac{g(|x|/2)}{g(r)} \leq C_g.
    \label{goal32}
\end{align}
Together, \eqref{goal31} and \eqref{goal32} yield \eqref{goal3}.

Finally, \(\operatorname M^g_{R_0}(D\boldsymbol{b})(x)\) is bounded on the annulus \(\{x\in\mathbb R^d:\ |x|>2R_0\}\), since \(D\boldsymbol{b}\) is bounded on \(\{x\in\mathbb R^d:\ |x|>R_0\}\), which is precisely the region probed by \(\operatorname M^g_{R_0}(D\boldsymbol{b})(x)\) when \(|x|>2R_0\).

We conclude that \eqref{goalM} holds, and consequently that \(M_\infty^g(D\boldsymbol{b})\) is bounded everywhere.

\textit{Step 6: conclusion.}
We may now conclude. We have shown that \(\boldsymbol{b}\) is a DiPerna--Lions vector field for which there exists a superlinear function \(\Phi\) satisfying \Cref{regularityforPhi}. Moreover, for \(\epsilon=\epsilon(R,\delta)\) sufficiently small, the good sets \(\Omega_{\epsilon,R}\) coincide with \(B_R\). Since \(\boldsymbol{b}\) is autonomous, radial, and sufficiently regular, the corresponding flow \(\boldsymbol{X}(t,0,x)\) may be computed explicitly, and for suitable \(x,y\) and \(t>0\) one has
\[
G(|\boldsymbol{X}(t,0,x)-\boldsymbol{X}(t,0,y)|)=G(|x-y|)+t.
\]
By choosing \(\delta\) sufficiently large, one may moreover arrange that \(t>0\) be arbitrarily large. The proof is complete. \end{proof}

\subsection{An example of a regular Lagrangian flow that does not obey \texorpdfstring{\eqref{estimate}}{}} 

We will now show that the classical Lipschitz-type estimate \eqref{estimate}, comparing \(\boldsymbol X(t,\cdot)\) with \(\boldsymbol X(s,\cdot)\), may fail in the general case where \(\boldsymbol b\) merely satisfies \eqref{regularity}, \eqref{growth}, and \eqref{compressibility}.

\begin{proposition} \label{counter} There exists an autonomous vector field \(\boldsymbol b \in W^{1,1}_{\loc}(\mathbb R)\) such that \begin{itemize} \item \(D\boldsymbol b \in L^1(\mathbb R)\), \item \(D\boldsymbol b \geq 0\) almost everywhere, and \item \(\boldsymbol b \in L^\infty(\mathbb R)\), \end{itemize} with the following property: for every \(T>0\) and every bounded open interval \(I\subset \mathbb R\) containing \(0\), there does not exist a function \( k \in L^1((0,T)\times I) \) such that \eqref{estimate} holds for all \(x,y\in I\) and all \(0\leq s<t\leq T\), where \(\boldsymbol X\) denotes the regular Lagrangian flow associated with \eqref{flow}.\end{proposition}
\begin{proof}
The construction is quite simple. Define 
\[ f(x)\coloneqq  \begin{dcases} \dfrac{1}{x\log(x)^2} & \text{if } x\in(0,e^{-2}),\\[1ex] 0 & \text{if } x\notin(0,e^{-2}), \end{dcases} \] 
and set 
\begin{equation} 
\boldsymbol b(x)\coloneqq 1+\int_0^x f(\xi)\,\dd\xi. \label{defb} \end{equation}
Since \(f\in L^1(\mathbb R)\), \(f\geq 0\), and \( \int_{\mathbb R} f(\xi)\,\dd\xi = \frac12, \) it follows that \(\boldsymbol b\in W^{1,1}_{\loc}(\mathbb R)\), \(\boldsymbol b'=f\in L^1(\mathbb R)\), \(\boldsymbol b'\geq 0\) almost everywhere, and \begin{equation} 1\leq \boldsymbol b(x)\leq \frac32 \qquad \text{for all } x\in\mathbb R. \label{Linftybound} \end{equation}
    
In particular, the regular Lagrangian flow \(\boldsymbol X\) is globally well defined. In fact, by separation of variables, \begin{equation} \boldsymbol X(t,0,x)=H^{-1}(H(x)+t), \label{separation} \end{equation} where \[ H(x)=\int_0^x \frac{1}{\boldsymbol b(\xi)}\,\dd\xi. \] By \eqref{Linftybound}, \(H\) is well defined, strictly increasing, and satisfies \[ \frac23 \leq H'(x)\leq 1 \qquad \text{for all } x\in\mathbb R. \] 
Hence, \(H\) is invertible. Differentiating the identity \(H(\boldsymbol X(t,0,x))=H(x)+t \) with respect to \(x\), we obtain \[ \partial_x \boldsymbol X(t,0,x) = \frac{H'(x)}{H'(\boldsymbol X(t,0,x))} = \frac{\boldsymbol b(\boldsymbol X(t,0,x))}{\boldsymbol b(x)}. \] Therefore, \begin{equation} \frac23 \leq \partial_x \boldsymbol X(t,0,x) \leq \frac32 \qquad \text{for all } t\geq 0 \text{ and } x\in\mathbb R. \label{lipschitzX} \end{equation} In particular, for every \(t>0\), the map \[ x\mapsto \boldsymbol X_t(x)\coloneqq \boldsymbol X(t,0,x) \] is an increasing global diffeomorphism with its derivative bounded by below and above.
    
We claim that \eqref{estimate} cannot hold for any \(k\in L^1((0,T)\times I)\), where \(I\subset \mathbb R\) is any bounded open interval containing \(0\). Without loss of generality, we may assume \( I=(-\delta,\delta) \) for some \(\delta>0\). Suppose, by contradiction, that such a function \(k\) exists.

If \(x<y\) in \(I\), then \(\boldsymbol X_t(x)<\boldsymbol X_t(y)\) for all \(t>0\) by \eqref{lipschitzX}. Hence \eqref{estimate} implies \[ \log\bigl(\boldsymbol X_t(y)-\boldsymbol X_t(x)\bigr) - \log\bigl(\boldsymbol X_s(y)-\boldsymbol X_s(x)\bigr) \leq \int_s^t \bigl(k(\tau,x)+k(\tau,y)\bigr)\,\dd\tau \qquad (0\leq s<t\leq T). \]
By the Lebesgue differentiation theorem for vector-valued functions, \begin{equation} \frac{\int_{\boldsymbol{X}_t(x)}^{\boldsymbol{X}_t(y)} f(\xi)\,\dd \xi}{\boldsymbol{X}_t(y) -\boldsymbol{X}_t(x)} = \frac{\dd}{\dd t}\log\big(\boldsymbol{X}_t(y) -\boldsymbol{X}_t(x)\big) \leq k(t,x) + k(t,y) \quad \text{ in $\mathcal D'(I\times I)$  for }t \in W, \nonumber \end{equation} 
where $W \subset (0,T)$ has full measure.
    
Accordingly, there exist \(t\in(0,\delta/2)\) and $E_t \subset I$ such that $E_t$ has full measure, the ``trace'' \(k_t(x)\coloneqq  k(t,x)\) belongs to \(L^1(I)\), and it holds that
\begin{equation} \frac{1}{\boldsymbol{X}_t(y) -\boldsymbol{X}_t(x)}\int_{\boldsymbol{X}_t(x)}^{\boldsymbol{X}_t(y)} f(\xi)\,\dd \xi = \frac{\dd}{\dd t}\log\big(\boldsymbol{X}_t(y) -\boldsymbol{X}_t(x)\big) \leq k_t(x) + k_t(y) \quad \text{ in $\forall x<y \in E_t$.} \label{diff}
\end{equation}
Since \(\boldsymbol X_t^{-1}\) is Lipschitz by \eqref{lipschitzX}, the set \( \Omega\coloneqq \boldsymbol X_t^{-1}(E_t) \) has full measure in the interval \( J\coloneqq \boldsymbol X_t^{-1}(I), \) and the function \( \eta \mapsto q(\eta)\coloneqq k_t(\boldsymbol X_t^{-1}(\eta)) \) belongs to \(L^1(J)\). Moreover, due to the specific form of $\boldsymbol{b}$ \eqref{defb} and the fact that $0<t<\delta/2$, we have that \begin{equation} 0 \in \operatorname{int}(J). \label{0inJ} \end{equation}

Rewriting \eqref{diff} in terms of the variables \( \zeta=\boldsymbol X_t(x)\) and \( \eta=\boldsymbol X_t(y), \) we obtain 
\begin{equation} \frac{1}{\eta-\zeta}\int_\zeta^\eta f(\xi)\,\dd\xi \leq q(\zeta)+q(\eta) \label{diff2} \end{equation} 
for  every \(\zeta<\eta\) in \(\Omega\).

We now show that \eqref{diff2} is impossible. By \eqref{0inJ}, there exists \(N\geq 1\) such that the intervals \[ A_n\coloneqq \left(-e^{-n},-\frac12 e^{-n}\right), \qquad B_n\coloneqq \left(\frac12 e^{-n},e^{-n}\right) \] are contained in \(J\) for all \(n\geq N\). If \(\zeta\in A_n\cap \Omega\) and \(\eta\in B_n\cap \Omega\), then \begin{equation} \eta-\zeta\leq 2e^{-n}, \label{dist} \end{equation} while \begin{align} \int_\zeta^\eta f(\xi)\,\dd\xi &\geq \int_{0}^{\frac{1}{2}e^{-n}} \frac{1}{\xi\log(\xi)^2}\,\dd\xi = \frac{1}{n+\log 2}\label{integral} \end{align}
for all \(n\geq N\). Combining \eqref{dist}, \eqref{integral}, and \eqref{diff2}, we infer that \[ q(\zeta)+q(\eta)\geq \frac{c_1 e^n}{n} \qquad \text{for almost every } (\zeta,\eta)\in A_n\times B_n, \] for some constant \(c_1>0\). Integrating over \(A_n\times B_n\), we obtain \[ \mathcal L^1(A_n)\int_{A_n} q(\zeta)\,\dd\zeta + \mathcal L^1(B_n)\int_{B_n} q(\eta)\,\dd\eta \geq \frac{c_1 e^n}{n}\mathcal L^1(A_n)\mathcal L^1(B_n). \] Since \[ \mathcal L^1(A_n)=\mathcal L^1(B_n)=\frac12 e^{-n}, \] it follows that there exists \(c_2>0\) such that \[ \int_{A_n} q(\zeta)\,\dd\zeta+\int_{B_n} q(\eta)\,\dd\eta \geq \frac{c_2}{n} \qquad \text{for all } n\geq N. \] As the family \(\{A_n,B_n\}_{n\geq N}\) is pairwise disjoint, we conclude that \[ \int_J q(\xi)\,\dd\xi \geq \sum_{n=N}^\infty \left( \int_{A_n} q(\zeta)\,\dd\zeta + \int_{B_n} q(\eta)\,\dd\eta \right) \geq c_2\sum_{n=N}^\infty \frac{1}{n} = \infty, \] which contradicts the fact that \(q\in L^1(J)\).

Therefore, \eqref{estimate} cannot hold, and the proof is complete.\end{proof}

\subsection{An example of a regular Lagrangian flow for a \texorpdfstring{$\bv$}{} vector field that does not obey \texorpdfstring{\eqref{estimatenew}}{}} 

In our final example, we will provide a $\bv$ vector field whose regular Lagrangian flow cannot satisfy \eqref{estimatenew}. Thus, one can argue that \eqref{estimatenew} is particular and sharp to the DiPerna--Lions class $L^1((0,T); W_{\loc}^{1,1}(\mathbb R^d))$.

\begin{proposition}\label{counter2}
There exists an autonomous vector field \(\boldsymbol b\in \bv_{\loc}(\mathbb R)\) such that
\begin{itemize}
    \item \(D\boldsymbol b\in \mathcal M(\mathbb R)\),
    \item \(D\boldsymbol b\geq 0\) in \(\mathcal D'(\mathbb R)\), and
    \item \(\boldsymbol b\in L^\infty(\mathbb R)\),
\end{itemize}
with the following property. Let \(T>0\) and let \(I\subset\mathbb R\) be a bounded open interval containing \(0\). Then there do not exist
\begin{itemize}
    \item a continuous function \(g:\mathbb R_+\to\mathbb R_+\), nonincreasing near \(0\), and such that the Osgood condition
    \[
    \int_0^\epsilon \frac{\dd u}{u g(u)}=\infty
    \]
    is satisfied for some $\epsilon>0$, and
    \item a function \(k\in L^1((0,T)\times I)\),
\end{itemize}
for which the estimate \eqref{estimatenew}
holds for almost every \(x,y\in I\) and all \(0\leq s<t\leq T\), where
\[
G(r)=\int_a^r \frac{\dd u}{u g(u)}
\]
for some \(a>0\), and \(\boldsymbol X\) denotes the regular Lagrangian flow associated with \eqref{flow}.
\end{proposition}

\begin{proof}
The argument is a variation on the proof of \Cref{counter}, but the present example is even simpler. Consider
\begin{equation}
    \boldsymbol b(x)\coloneqq 1+1_{[0,\infty)}(x).
    \label{defb2}
\end{equation}
Then \(\boldsymbol b\in \bv_{\loc}(\mathbb R)\cap L^\infty(\mathbb R)\), and
\(
D\boldsymbol b=\delta_0,
\)
the Dirac mass at the origin. In particular, \(D\boldsymbol b\geq0\) in the sense of distributions.

The regular Lagrangian flow \(\boldsymbol X\) is globally well defined by the Ambrosio theory \cite{ambrosio}. In the present case, however, it can be written explicitly again. Indeed, since the velocity is equal to \(1\) on \((-\infty,0)\) and to \(2\) on \([0,\infty)\), one finds
\[
\boldsymbol X_t(x)\coloneqq \boldsymbol X(t,0,x)=
\begin{dcases}
x+t & \text{if } x<-t,\\[0.5ex]
2x+2t & \text{if } -t\leq x<0,\\[0.5ex]
x+2t & \text{if } x\geq 0.
\end{dcases}
\]
In particular, for every \(t>0\), the map \(\boldsymbol X_t\) is an increasing bi-Lipschitz homeomorphism of \(\mathbb R\).

We now show that \eqref{estimatenew} cannot hold. Let \(I\subset\mathbb R\) be a bounded open interval containing \(0\). Shrinking if necessary, we may assume
\(
I=(-\delta,\delta)
\)
for some \(\delta>0\). Suppose, by contradiction, that there exist \(g\) and \(k\) as in the statement such that \eqref{estimatenew} holds.

As before, there exist \(t\in(0,\delta/2)\) and $E_t \subset I$ such that $E_t$ has full measure, the ``trace'' \(k_t(x)\coloneqq k(t,x)\) belongs to \(L^1(I)\), and it holds that \begin{equation} \frac{\int_{\boldsymbol{X}_t(x)}^{\boldsymbol{X}_t(y)} \dd\boldsymbol{b}' (\xi)}{(\boldsymbol{X}_t(y) -\boldsymbol{X}_t(x))g(\boldsymbol{X}_t(y) -\boldsymbol{X}_t(x))} = \frac{\dd}{\dd t}G\big(\boldsymbol{X}_t(y) -\boldsymbol{X}_t(x)\big) \leq k_t(x) + k_t(y) \quad \text{ in $\forall x<y \in E_t$.} \nonumber \end{equation} We again define \( \Omega\coloneqq \boldsymbol X_t^{-1}(E_t) \), which has full measure in \( J\coloneqq \boldsymbol X_t^{-1}(I) \), and the function \( \eta \mapsto q(\eta)\coloneqq k_t(\boldsymbol X_t^{-1}(\eta)) \), which belongs to \(L^1(J)\). Moreover, due to the specific form of $\boldsymbol{b}$ \eqref{defb2} and the fact that $0<t<\delta/2$, we have that $0 \in \operatorname{int}(J).$

Thus, by a change of variables, deduce that \begin{equation} \frac{1}{(\eta-\zeta )g(\eta-\zeta)}\int_\zeta^\eta \boldsymbol{b}' (\dd\xi) \leq q(\zeta)+q(\eta) \nonumber \end{equation} for almost every \(\zeta<\eta\) in \(\Omega\). If we choose $\zeta<0<\eta$, we obtain \begin{equation} \frac{1}{(\eta-\zeta)g(\eta-\zeta)} \leq q(\zeta)+q(\eta) \text{ for $\zeta<0<\eta$ in $\Omega. $} \label{diff3}\end{equation}

We now derive a contradiction. Since \(0\in \operatorname{int}(J)\), there exists \(N\geq1\) such that
\[
A_n\coloneqq \left(-e^{-n},-\frac12 e^{-n}\right),
\qquad
B_n\coloneqq \left(\frac12 e^{-n},e^{-n}\right)
\]
are contained in \(J\) for all \(n\geq N\), and such that \(g\) is nonincreasing on \((0,\frac12 e^{-N})\). If \(\zeta\in A_n\cap\Omega\) and \(\eta\in B_n\cap\Omega\), then
\(
e^{-n}\leq \eta-\zeta\leq 2e^{-n},
\)
and therefore \eqref{diff3} gives
\[
q(\zeta)+q(\eta)\geq \frac{c_1e^n}{g(e^{-n})}
\qquad \text{for almost every } (\zeta,\eta)\in A_n\times B_n,
\]
for some constant \(c_1>0\). Integrating this over \(A_n\times B_n\), we obtain \[ \mathcal L^1(A_n)\int_{A_n} q(\zeta)\,\dd\zeta + \mathcal L^1(B_n)\int_{B_n} q(\eta)\,\dd\eta \geq \frac{c_1e^n}{g(e^{-n})}\mathcal L^1(A_n)\mathcal L^1(B_n). \] Since \( \mathcal L^1(A_n)=\mathcal L^1(B_n)=\frac12 e^{-n}, \) there exists some \(c_2>0\) such that \[ \int_{A_n} q(\zeta)\,\dd\zeta+\int_{B_n} q(\eta)\,\dd\eta \geq \frac{c_2}{g(e^{-n})} \qquad \text{for all } n\geq N. \]

Because the family \(\{A_n,B_n\}_{n\geq N}\) is pairwise disjoint, we conclude that
\[
\int_J q(\xi)\,\dd\xi
\geq
c_2\sum_{n=N}^\infty \frac{1}{g(e^{-n})}.
\]
Finally, since \(g\) is nonincreasing near \(0\), the integral test gives
\[
\sum_{n=N}^\infty \frac{1}{g(e^{-n})}
\sim
\int_N^\infty \frac{\dd t}{g(e^{-t})}
=
\int_0^{e^{-N}} \frac{\dd r}{r g(r)}
=
\infty,
\]
by the Osgood condition. This contradicts the fact that \(q\in L^1(J)\).

The contradiction proves that no such pair \((g,k)\) can exist. This completes the proof.
\end{proof}

\section{Regularity results for solutions to the transport equation}\label{SecReg} 

We now turn to the regularity of solutions \(u(t,x)\) to the transport equation \eqref{transport}. For simplicity, throughout this section we shall work under the stronger global assumption
\begin{equation} D\boldsymbol b \in L^1((0,T);L^1(\mathbb R^d)), \label{regularity2} \end{equation} 
in addition to \eqref{growth}, and the almost compressibility condition
\begin{equation} \Div \boldsymbol b \in L^1((0,T);L^\infty(\mathbb R^d)). \label{compressibility2} \end{equation} 
We also assume that the initial datum satisfies \begin{equation} 
u(0,x)=:u_0(x)\in (L^\infty\cap \bv)(\mathbb R^d). \label{cauchy} 
\end{equation} 

Under these hypotheses, the classical theory of DiPerna and Lions \cite{dipernalions} ensures the existence and uniqueness of a bounded solution \( u\in L^\infty((0,T)\times \mathbb R^d) \) to \eqref{transport} with initial datum \eqref{cauchy}. Moreover, one has the Lagrangian representation \[ u(t,\boldsymbol X(t,0,x))=u_0(x), \] where \(\boldsymbol X\) denotes the regular Lagrangian flow associated with \(\boldsymbol b\). Since \eqref{compressibility2} yields time reversibility of \eqref{flow}, one may also introduce the backward flow \(\boldsymbol Y(t,x) = \boldsymbol{X}(0,t,x)\). As a consequence, the solution admits the representation \begin{equation} u(t,x)=u_0(\boldsymbol Y(t,x)), \label{repr} \end{equation} 
whose regularity properties we shall now investigate. 

To motivate the result, let us first recall the picture under the hypothesis that 
\begin{equation} \boldsymbol b \text{ is divergence free and } \boldsymbol b \in L^1((0,T);W^{1,p}(\mathbb R^d)) \quad \text{for some } 1<p<\infty. \label{regularityLp} \end{equation}
Even under \eqref{regularityLp}, solutions to \eqref{transport} with initial datum \eqref{cauchy} may become highly irregular and may lose all fractional Sobolev regularity \(W^{s,q}\), even when \(0<q<1\); see \cite{jabin,mazzucato1,mazzucato2,bruenguyen}.  On the other hand, one may still preserve a logarithmic Sobolev regularity. More precisely, one is led to consider the seminorm 
\begin{equation} \|f\|_{\dot H^{\log}(\mathbb R^d)}^2 = \int_{\mathbb R^d}\int_{B_\delta} \frac{|f(x+h)-f(x)|^2}{|h|^d}\,\dd x\,\dd h, \label{hlog} \end{equation} for some \(\delta>0\). (The notation \(\dot H^{\log}\) can be explained by \Cref{mixingthm2} below.) An important result of Léger shows that, under \eqref{regularityLp}, logarithmic regularity of this type may still be propagated. Indeed, a corollary of the Léger theory is that
\[ \|u(t,\cdot)\|_{\dot H^{\log}(\mathbb R^d)}^2 \leq C_{p,d} \left( \|u_0\|_{(\dot H^{\log}\cap L^2)(\mathbb R^d)}^2 + \|u_0\|_{(L^1\cap L^\infty)(\mathbb R^d)}^2 \int_0^t \|D\boldsymbol b(\tau,\cdot)\|_{L^p(\mathbb R^d)}\,\dd\tau \right); \]
see also \cite{MS, huysmans0, bcpj, dpf} for recent developments under other approaches.

This phenomenon was later sharpened by Bruè and Nguyen \cite{bruenguyen}, who proved that under \eqref{regularityLp}, 
\begin{align} 
\int_{\mathbb R^d}\int_{B_{1/3}} &\frac{|u(t,x+h)-u(t,x)|^2}{|h|^d\log(1/|h|)^{1-p}}\,\dd x\,\dd h \nonumber\\ &\leq C_{p,d} \left( \left( \int_0^t \|D\boldsymbol b(\tau,\cdot)\|_{L^p(\mathbb R^d)}\,\dd\tau \right)^p + \|u_0\|_{\bv(\mathbb R^d)}^p + \|u_0\|_{L^1(\mathbb R^d)} \right). \label{BN} 
\end{align} 
They also used the mixing flow of Alberti, Crippa, and Mazzucato \cite{mazzucato0} to show the sharpness of this estimate: for every \(q\geq 1\), they constructed a compactly supported divergence-free vector field \( \boldsymbol b\in L^\infty((0,\infty);W^{1,q}(\mathbb R^d)) \) and a compactly supported datum \( u_0\in (L^\infty\cap W^{1,d})(\mathbb R^d) \) such that \[ \int_{\mathbb R^d}\int_{B_{1/3}} \frac{|u(t,x+h)-u(t,x)|^2}{|h|^d\log(1/|h|)^\gamma}\,\dd x\,\dd h = \infty \] for every \(t>0\) and every \(\gamma>1-q\). 

The weight appearing in \eqref{BN}, \( \log(1/r)^{1-p}, \) is slowly varying at the origin. Naturally, this weight degenerates when \(p=1\), and the endpoint case is not covered by \cite{bruenguyen}. However, this suggests replacing \(\log(1/r)^{1-p}\) by the \(g\)-weights introduced earlier, and asking whether one can still maintain a suitable ``logarithmic'' Sobolev regularity in the endpoint setting \eqref{regularity2} by means of the Osgood regularity. This is precisely the content of the next theorem. 

\begin{theorem}\label{bressan} 

Let \(\boldsymbol b\) satisfy \eqref{growth}, \eqref{regularity2}, and \eqref{compressibility2}. Let \( u\in L^\infty((0,T)\times\mathbb R^d) \) be the unique solution to the transport equation \eqref{transport} with initial datum \(u_0\) satisfying \eqref{cauchy}. Finally, consider functions $\Phi$ and $\Phi_0$ as in \Cref{regularityforPhi} with $\lambda = \infty$. 

Then, for every \(a>0\) and \(R>0\), there exists \(\delta_a>0\) sufficiently small such that \begin{align} \sup_{t\in(0,T)} \int_{B_{\delta_a}}&\int_{B_R} \frac{\min\{a^2,|u(t,x+h)-u(t,x)|^2\}}{|h|^dg(|h|)}\,\dd x\,\dd h \leq C_{d,R,\Phi,a} \exp\!\left( \|\Div \boldsymbol b\|_{L^1((0,T);L^\infty(\mathbb R^d))} \right) \nonumber \\ &\left( \int_0^T\|D\boldsymbol b(t,\cdot)\|_{L^1(\mathbb R^d)}\,\dd t + \int_0^T\|D\boldsymbol b(t,\cdot)\|_{L^{\Phi_0}(\mathbb R^d)}\,\dd t + \|u_0\|_{\bv(\mathbb R^d)} \right),  \label{gagliardo} \end{align} 
where \(g\) is defined by \eqref{definitiong} and \(G\) by \eqref{definitionG}. In fact, \(\delta_a\) may be chosen as any number in $(0,R)$ such that \(G(\delta_a)<-a. \) \end{theorem} 

\begin{remark} 
\textnormal{Some remarks are in order.}

\begin{itemize} 

    \item \textnormal{By \eqref{compressibility2} and \eqref{repr}, one has \[ \|u(t,\cdot)\|_{L^\infty(\mathbb R^d)}=\|u_0\|_{L^\infty(\mathbb R^d)}. \] Consequently, if one chooses \(a\geq \|u_0\|_{L^\infty(\mathbb R^d)}\) in \Cref{bressan}, then the truncation in \eqref{gagliardo} becomes irrelevant, and one obtains a regularity estimate for the full solution \(u(t,\cdot)\).} 

    \item \textnormal{We note that one could replace the assumptions \eqref{growth} and \eqref{regularity2} by \( \boldsymbol b\in L^1((0,T);L^\infty(\mathbb R^d)) \)  and \eqref{regularity}, respectively. In that case, the norms on the right-hand side of \eqref{gagliardo} should be taken over the region \(B_{2R+2\lambda}\), where \( \lambda=\int_0^T \|\boldsymbol b(t,\cdot)\|_{L^\infty(\mathbb R^d)}\,\dd t, \) rather than over the whole space \(\mathbb R^d\).}

    \item \textnormal{We stress that as explained in \Cref{qualitativesub}, \Cref{regularityforPhi} is always valid for some $\Phi$ and $\Phi_0$. Nevertheless, there is an interesting perspective on \Cref{bressan}, which is, instead of working with $\Phi$, to construct $g$-weights that lead to $g$-maximal operators with the Lusin mean value property, Wiener bound, and the Osgood condition, and inspecting which integrability properties of $D \boldsymbol{b}$ are necessary.}
\end{itemize}

\end{remark}

Our argument is deeply inspired by that of Bruè and Nguyen in \cite{bruenguyen}. Nonetheless, while they had the pointwise estimate \eqref{estimate} at their disposal, we must rely on the following weaker estimate based on \eqref{estimatenew}, which is significantly more complex to handle. We begin with the analogue of \cite[Corollary 2.11]{bruenguyen}. Henceforth, $L\geq 1$ will stand for the compressibility constant 
\begin{equation}
    L\coloneqq\exp\left(\int_0^T\|\Div \boldsymbol{b}(t,\cdot)\|_{L^\infty(\mathbb{R}^d)}\,\dd t\right). \label{defL}
\end{equation}


\begin{lemma}[The pointwise bound]\label{pointwisefortransport}
    Let $\boldsymbol{b}$ and $u_0$ be vector field and function as in \Cref{bressan}. 
    
    Then, there exist measurable functions $v(t,x),w(t,x):[0,T]\times \mathbb{R}^d\mapsto \mathbb{R}\cup \{\infty\}$ such that for any $u_0$ initial data of \eqref{transport}, the unique bounded solution $u$ of \eqref{transport} satisfies for almost every $x,\, h\in B_R$ and all $t\in (0,T)$
    \[|u(t,x+h)-u(t,x)|\leq \big(v(t,x)+v(t,x+h)\big)\,G^{-1}\bigg(G(|h|)+w(t,x)+w(t,x+h)\bigg)\]
    where for $L$ given by \eqref{defL}, it holds that
    \[\begin{split} \pseudo {v(t,\cdot)}_{\log L(B_{2R})}\,\dd t&\leq  C_d \left(\mathcal L^d(B_{2R})+L\|u_0\|_{\bv(\mathbb{R}^d)}\right)\quad\text{and}\\ \|w(t,\cdot)\|_{L^1(B_{2R})}\,\dd x&\leq C_{d,R,\Phi} L \left(\int_0^T\|D\boldsymbol{b}(\tau,\cdot)\|_{L^1(\mathbb R^d)}\,\dd \tau+ \int_0^T \Vert D\boldsymbol{b}(\tau,\cdot)\|_{L^{\Phi_0}(\mathbb R^d)}\,\dd \tau\right).\end{split}\]
\end{lemma}
\begin{proof}
    We will work with the Lagrangian representation \eqref{repr}. The main observation to have in mind here is that \Cref{quantitative} can be applied $\boldsymbol{Y}$, yielding the bound
    \begin{equation}
       |\boldsymbol{Y}(t,x+h)-\boldsymbol{Y}(t,x)|\leq G^{-1}\left(G(|h|)+ \int_0^t\left(k(\tau,x+h)+k(\tau,x)\right)\,\dd \tau\right) \label{estimateY}
    \end{equation}
    where the function $k$ can be taken explicitly
    \[k(x,\tau)\coloneqq C_d\operatorname{M}_{\infty}^g D\boldsymbol{b}(\tau,\boldsymbol{Y}(\tau,x)).\]
    
    Therefore, by \eqref{repr}, the classical Lusin--Lipschitz mean value inequality yields 
    \[\begin{split}
    &|u(t,x+h)-u(t,x)|\\&=|u_0(\boldsymbol{Y}(t,x+h))-u_0(\boldsymbol{Y}(t,x))|\\&\leq \big(\operatorname{M} D u_0(\boldsymbol{Y}(t,x+h))+\operatorname{M} D u_0(\boldsymbol{Y}(t,x))\big)|\boldsymbol{Y}(t,x+h)-\boldsymbol{Y}(t,x)|\\ &\leq 
    \big(\operatorname{M} D u_0(\boldsymbol{Y}(t,x+h))+\operatorname{M} D u_0(\boldsymbol{Y}(t,x))\big) G^{-1}\bigg(G(|h|)+ \int_0^t\big(k(\tau,x+h)+k(\tau,x)\big)\,\dd \tau\bigg).
    \end{split}\]
    Let us thus define
    \[v(t,x)\coloneqq \operatorname{M} (D u_0)(\boldsymbol{Y}(t,x))\quad \text{and}\quad w(t,x)\coloneqq \int_0^tk(\tau,x)\,\dd \tau.\]

    The $L^\infty((0,T); L^1(2B_R))$-bounds of $w\in L^{1}((0,T)\times \mathbb{R}^d)$ follow directly from the compressibility of $\boldsymbol{Y}$, Lemmas \ref{constructiong} and \ref{quantitative}, and Corollary \ref{constructiong2}. As  for the estimate for $v(t,x)$, notice that again by the compressibility of $\boldsymbol{Y}$ and the Hardy--Littlewood inequality $\operatorname{M} (D u_0)\in L^1_w(\mathbb{R}^d)$, we have that
    \[\begin{split}
    \pseudo{v(t,\cdot)}_{\log (B_R)}  =\int_{B_R}\log(1+|v(t,x)|)\,\dd x=&\int_{B_R}\int_0^{|v(t,x)|}\frac{1}{1+\tau}\,\dd \tau\,\dd x\\=&\int_0^\infty\int_{B_R\cap\{|v(t,\cdot)|>\tau\}}\frac{1}{1+\tau}\,\dd x\,\dd \tau\\ =&\int_0^\infty\frac{\mathcal L^d\{x\in B_R: |v(t,x)|>\tau\}}{1+\tau}\,\dd\tau\\\leq& \log(2)\mathcal L^d(B_R)+\pseudo{v(t,\cdot)}_{L^1_w(\mathbb{R}^d)}\\\leq & C_d\left(\mathcal L^d(B_R)+L|Du_0|(\mathbb{R}^d)\right),\end{split}\]
    by the Cavalieri principle and the definition of the weak $L^1$ pseudo-norm \[\pseudo{f}_{L^1_w} = \sup_{\lambda >0} \lambda \operatorname{meas} \{|f|>\lambda\}.\] The proof is complete.
\end{proof}

We next prove the analogue of \cite[Lemma 2.18]{bruenguyen}. Since the argument is of a purely functional-analytic nature, we state it in a slightly more general form.

\begin{lemma}\label{lemma217}
Assume that \(u\in L^1(\mathbb R^d)\) satisfies the pointwise estimate
\begin{equation}\label{pointwiseforu}
    |u(x+h)-u(x)|
    \leq
    \bigl(v(x+h)+v(x)\bigr)\,
    G^{-1}\bigl(G(|h|)+w(x+h)+w(x)\bigr),
\end{equation}
where \(v,w : \mathbb R^d \to [0,\infty]\) are measurable and such that \(\log_+ v,w\in L^1_{\loc}(\mathbb R^d)\), and
\(G:(0,a]\to(-\infty,0]\)
is an increasing \(C^1\) function with $G(0+) = -\infty$.

Then, for every \(R>0\) and every \(h\in B_R\) such that \(G(|h|)\leq -a\), one has
\[
\begin{aligned}
\int_{B_R}\min\big\{a^2,|u(x+h)-u(x)|^2\big\}\,\dd x
\leq C_d\Bigg[
&G^{-1}\bigl(G(|h|)+a\bigr)\|u\|_{L^1(\mathbb R^d)}
\\
&+
\int_a^{-G(|h|)}
\partial_\lambda\bigg(\Big(G^{-1}\bigl(G(|h|)+\lambda\bigr)\Big)^2\bigg)
\,\mu(h,\lambda)\,\dd\lambda
\Bigg],
\end{aligned}
\]
where
\begin{align}
    \mu(h,\lambda) 
&\coloneqq
\mathcal L^d
\bigg\{
x\in B_R: \nonumber \\ & \quad
\bigl(v(x+h)+v(x)\bigr)\,
G^{-1}\bigl(G(|h|)+w(x+h)+w(x)\bigr)
\geq
G^{-1}\bigl(G(|h|)+\lambda\bigr)
\bigg\}. \label{defmu}
\end{align}
\end{lemma}

Before proving this result, let us note that in the special case \(G\equiv \log\) and $a=1$, \eqref{pointwiseforu} may be rewritten as
\[
|u(x+h)-u(x)|
\leq
|h|\,\exp\bigl\{\widetilde w(x+h)+\widetilde w(x)\bigr\},
\]
with
\(
\widetilde w(x):=2w(x)+2\log_+ v(x).
\)
Thus, \Cref{lemma217} yields
\begin{equation}\label{lemma217log}
\int_{B_R}\min\{1,|u(x+h)-u(x)|^2\}\,\dd x
\leq
C_d\left(
|h|\,\|u\|_{L^1(\mathbb R^d)}
+
|h|^2
\int_1^{-\log|h|}
e^{2\lambda}\nu(\lambda)\,\dd\lambda
\right),
\end{equation}
where
\(
\nu(\lambda)
\coloneqq
\mathcal L^d\bigl(\{x\in B_R:2\widetilde w(x)\geq \lambda\}\bigr).
\)
This is precisely \cite[Lemma 2.18]{bruenguyen} in the case of \(L^1((0,T);W^{1,p}(\mathbb R^d))\) vector fields. In particular, \Cref{lemma217} shows that \eqref{lemma217log} remains valid even in the limiting case of \(L^1((0,T);W^{1,L\log L}(\mathbb R^d))\) vector fields.

\begin{proof}
By Cavalieri's principle,
\[
\begin{aligned}
\int_{B_R}\min\big\{a^2,|u(x+h)-u(x)|^2\big\}\,\dd x
&=
2\int_0^{G^{-1}(G(|h|)+a)}
s\,\mathcal L^d\bigl(\{x\in B_R:|u(x+h)-u(x)|>s\}\bigr)\,\dd s
\\
&\quad
+
2\int_{G^{-1}(G(|h|)+a)}^a
s\,\mathcal L^d\bigl(\{x\in B_R:|u(x+h)-u(x)|>s\}\bigr)\,\dd s.
\end{aligned}
\]
For the first integral, we use the elementary bound
\[
\int_{B_R}|u(x+h)-u(x)|\,\dd x
\leq
2\|u\|_{L^1(\mathbb R^d)},
\]
and obtain
\[
2\int_0^{G^{-1}(G(|h|)+a)}
s\,\mathcal L^d\bigl(\{x\in B_R:|u(x+h)-u(x)|>s\}\bigr)\,\dd s
\leq
C_d\,G^{-1}\bigl(G(|h|)+a\bigr)\|u\|_{L^1(\mathbb R^d)}.
\]
It therefore remains to estimate the second integral.

To this end, perform the change of variables
\[
s=G^{-1}\bigl(G(|h|)+\lambda\bigr).
\]
Then
\[
s\,\dd s
=
\frac{G^{-1}\bigl(G(|h|)+\lambda\bigr)}
{G'\bigl(G^{-1}(G(|h|)+\lambda)\bigr)}
\,\dd\lambda
=
\frac12
\partial_\lambda
\bigg(
\Big(G^{-1}\bigl(G(|h|)+\lambda\bigr)\Big)^2
\bigg)\,\dd\lambda.
\]
Since \(G\) is increasing, the map
\(
\lambda\mapsto G^{-1}\bigl(G(|h|)+\lambda\bigr)
\)
is increasing as well, and therefore, for $G^{-1}$ is nonnegative,
\[
\partial_\lambda
\bigg(
\Big(G^{-1}\bigl(G(|h|)+\lambda\bigr)\Big)^2
\bigg)\geq 0.
\]
Using \eqref{pointwiseforu}, we conclude that
\[
\begin{aligned}
2\int_{G^{-1}(G(|h|)+a)}^a
s\,\mathcal L^d&\bigl(\{x\in B_R:|u(x+h)-u(x)|>s\}\bigr)\,\dd s \\
&\leq
\int_a^{-G(|h|)}
\partial_\lambda
\bigg(
\Big(G^{-1}\bigl(G(|h|)+\lambda\bigr)\Big)^2
\bigg)
\mu(h,\lambda)\,\dd\lambda.
\end{aligned}
\]
Combining the two estimates proves the lemma.
\end{proof}

Finally, we prove the analogue of the so-called ``key lemma'' \cite[Proposition 2.16]{bruenguyen} in the more general setting where one only has the weaker pointwise estimate \eqref{pointwiseforu}. As in \Cref{lemma217}, we state the result in a form that is purely functional-analytic. 

\begin{lemma}[Key lemma]\label{keylemma} 
Let \(u\in L^1(\mathbb R^d)\) satisfy \eqref{pointwiseforu}, where \( v,w:\mathbb R^d\to [0,\infty) \) are measurable and such that \(\log_+ v,w\in L^1_{\loc}(\mathbb R^d)\), and \[ G(r)=\int_a^r \frac{\dd s}{s\,g(s)}, \] where \(g:\mathbb R_+\to\mathbb R_+\) is bounded away from zero. Assume moreover that \( G(0+)=-\infty. \) 

Then, for every \(R>0\) and every \(\delta_a\in(0,R)\) such that \(G(\delta_a)<-a\), one has \[ \int_{B_{\delta_a}}\int_{B_R} \frac{\min\{a^2,|u(x+h)-u(x)|^2\}}{|h|^d\,g(|h|)}\,\dd x\,\dd h \leq C_{d,g,a} \left( \|u\|_{L^1(\mathbb R^d)} + \pseudo{v}_{\log L(B_{2R})} + \|w\|_{L^1(B_{2R})} \right). \] 
\end{lemma}

\begin{proof} 

Combining \Cref{lemma217} with polar coordinates, we obtain \begin{align}\int_{B_{\delta_a}}&\int_{B_R} \frac{\min\{a^2,|u(x+h)-u(x)|^2\}}{|h|^d\,g(|h|)}\,\dd x\,\dd h \nonumber \\&\leq C_d\|u\|_{L^1(\mathbb R^d)} \int_0^{\delta_a} G^{-1}\bigl(G(r)+a\bigr)\frac{\dd r}{r\,g(r)} \nonumber\\ &\quad + C_d \int_0^{\delta_a}\int_a^{-G(r)} \partial_\lambda\bigg( \Big(G^{-1}\bigl(G(r)+\lambda\bigr)\Big)^2 \bigg) \nu(r,\lambda)\frac{\dd\lambda\,\dd r}{r\,g(r)} \nonumber\\ &=:\operatorname{(I)}+\operatorname{(II)}, \label{keylemmainequality} 
\end{align} 
where, for \(\mu(h,\lambda)\) given by \eqref{defmu}, \[ \nu(r,\lambda):= \int_{\mathbb S^{d-1}}\mu(r\omega,\lambda)\,\mathcal H^{d-1}(\dd\omega). \] 

We now estimate \(\operatorname{(I)}\) and \(\operatorname{(II)}\) separately. 

\textit{Step 1: estimate of \(\operatorname{(I)}\).} Since \(G'(r)=\frac{1}{r\,g(r)}\), we can rewrite \(\operatorname{(I)}\) as \[ \operatorname{(I)} = C_d\|u\|_{L^1(\mathbb R^d)} \int_0^{\delta_a} G^{-1}\bigl(G(r)+a\bigr)\,G'(r)\,\dd r. \] 
Performing the change of variables \( s=G^{-1}\bigl(G(r)+a\bigr), \) for which \(G'(s)\,\dd s=G'(r)\,\dd r\) yields 
\[ \int_0^{\delta_a} G^{-1}\bigl(G(r)+a\bigr)\,G'(r)\,\dd r = \int_0^{s_a} s\,G'(s)\,\dd s = \int_0^{s_a}\frac{\dd s}{g(s)}, \] 
where \( s_a:=G^{-1}\bigl(G(\delta_a)+a\bigr). \) Since \(g\) is bounded from below, the latter integral is finite. Consequently, 
\begin{equation} 
\operatorname{(I)}\leq C_{d,g,a}\|u\|_{L^1(\mathbb R^d)}. \label{estimateI} 
\end{equation} 

\textit{Step 2: first decomposition of \(\operatorname{(II)}\).} Again using \(G'(r)=\frac{1}{r\,g(r)}\), we may write 
\[ \begin{aligned} \operatorname{(II)} &= \int_a^\infty \int_0^{\min\{G^{-1}(-\lambda),\delta_a\}} \partial_\lambda\big([G^{-1}(G(r)+\lambda)]^2\big)\,G'(r)\,\nu(r,\lambda)\,\dd r\,\dd\lambda \\ &= \int_a^\infty \int_0^{\min\{G^{-1}(-\lambda),\delta_a\}} \partial_r\big([G^{-1}(G(r)+\lambda)]^2\big)\,\nu(r,\lambda)\,\dd r\,\dd\lambda. \end{aligned} \] 
In order to estimate \(\nu(r,\lambda)\), let us introduce \[ y(x,h):= G^{-1}\bigl(G(|h|)+w(x+h)+w(x)\bigr). \] 
Then, \[ \begin{aligned} G&\Big( \big(v(x+h)+v(x)\big) G^{-1}\big(G(|h|)+w(x+h)+w(x)\big) \Big)-G(|h|) \nonumber \\ &= \Big[ G\big((v(x+h)+v(x))\,y(x,h)\big)-G(y(x,h)) \Big] + \big(w(x+h)+w(x)\big). \end{aligned} \] 
Therefore, by \eqref{defmu} and the monotonicity of \(G\), 
\begin{align} 
\mu(h,\lambda) &= \mathcal L^d\bigg\{ x\in B_R: G\Big( \big(v(x+h)+v(x)\big) G^{-1}\big(G(|h|)+w(x+h)+w(x)\big) \Big)-G(|h|) \geq \lambda \bigg\} \nonumber\\ &\leq \mathcal L^d\bigg\{ x\in B_R: G\big((v(x+h)+v(x))\,y(x,h)\big)-G(y(x,h)) \geq \frac{\lambda}{2} \bigg\} \nonumber\\ &\quad + \mathcal L^d\bigg\{ x\in B_R: w(x+h)+w(x)\geq \frac{\lambda}{2} \bigg\} \nonumber\\ &=:\operatorname{(A)}+\operatorname{(B)}. \label{defAB} \end{align} 

We now estimate \(\operatorname{(A)}\) and \(\operatorname{(B)}\). 

\textit{Step 3: estimate of \(\operatorname{(A)}\).} Define the auxiliary function 
\[ H(z):= \sup_{y>0}\bigl(G(zy)-G(y)\bigr), \qquad z\geq 1. \] 
Since \(g\) is bounded below on \((0,\infty)\), we have 
\[ H(z) = \sup_{y>0}\int_y^{zy}\frac{\dd s}{s\,g(s)} \leq C_{d,g}\log z. \] 
Consequently, for $G$ is increasing,
\begin{align*}
    G\big((v(x+h)+v(x))\,y(x,h)\big)-G(y(x,h)) &\leq H\big(\max\{v(x+h)+v(x),1\}\big) \\ &\leq C_{d,g}\log_+\big(v(x+h)+v(x)\big). 
\end{align*}
Therefore, \[ \operatorname{(A)} \leq \mathcal L^d\bigg\{ x\in B_R: \log_+\big(v(x+h)+v(x)\big)\geq \frac{\lambda}{2C_{d,g}} \bigg\}. \] 
Using the elementary inequality 
\[ \log_+(a+b)\leq \log(1+a)+\log(1+b), \qquad a,b\geq 0, \] 
we deduce that 
\begin{equation} \operatorname{(A)} \leq 2\, \mathcal L^d\bigg\{ x\in B_{2R}: \log(1+v(x))\geq \frac{\lambda}{4C_{d,g}} \bigg\}. \label{estA} \end{equation} 

\textit{Step 4: estimate of \(\operatorname{(B)}\).} Exactly in the same way, one obtains 
\begin{equation} \operatorname{(B)} \leq 2\, \mathcal L^d\bigg\{ x\in B_{2R}: w(x)\geq \frac{\lambda}{4} \bigg\}. \label{estB} \end{equation} 

\textit{Step 5: estimate of \(\operatorname{(II)}\).} Combining \eqref{defAB}, \eqref{estA}, and \eqref{estB}, we infer that 
\begin{align} \sup_{r\in(0,\delta_a)}&\nu(r,\lambda)\nonumber\\ &\leq C_{d,g}\bigg( \mathcal L^d\bigg\{ x\in B_{2R}: \log(1+v(x))\geq \frac{\lambda}{4C_{d,g}} \bigg\}+ \mathcal L^d\bigg\{ x\in B_{2R}: w(x)\geq \frac{\lambda}{4} \bigg\} \bigg). \label{boundmu} \end{align} 
Since \(G(0+)=-\infty\), \(G^{-1}(0)=a\), and \(\partial_r\big([G^{-1}(G(r)+\lambda)]^2\big)\geq 0, \) we obtain by Fubini--Tonelli
\begin{align} \operatorname{(II)} &= \int_a^\infty \int_0^{\min\{G^{-1}(-\lambda),\delta_a\}} \partial_r\big([G^{-1}(G(r)+\lambda)]^2\big)\, \nu(r,\lambda)\,\dd r\,\dd\lambda \nonumber\\ &\leq \int_a^\infty \sup_{r\in(0,\delta_a)}\nu(r,\lambda) \int_0^{\min\{G^{-1}(-\lambda),\delta_a\}} \partial_r\big([G^{-1}(G(r)+\lambda)]^2\big)\,\dd r\,\dd\lambda \nonumber\\ &\leq C_d a^2 \int_a^\infty \sup_{r\in(0,\delta_a)}\nu(r,\lambda)\,\dd\lambda. \label{estimateII0} \end{align} Finally, inserting \eqref{boundmu} into \eqref{estimateII0} and using Cavalieri's principle, we arrive at \begin{equation} \operatorname{(II)} \leq C_{d,g,a} \left( \pseudo{v}_{\log L(B_{2R})} + \|w\|_{L^1(B_{2R})} \right). \label{estimateII} \end{equation} 

\textit{Step 6: conclusion.} Combining \eqref{estimateI} and \eqref{estimateII} in \eqref{keylemmainequality}, we obtain the desired estimate. \end{proof} We have thus proved \Cref{bressan}. Indeed, by \Cref{pointwisefortransport}, the function \(u(t,\cdot)\) satisfies the hypotheses of \Cref{lemma217}; mingling that lemma with \Cref{keylemma} yields exactly the estimate asserted in the theorem. 

Before closing this section, we also derive a decay estimate for the first-order difference \[ \int_{B_R}|u(t,x+h)-u(t,x)|\,\dd x, \] thereby extending the result of Bru\`e--Nguyen \cite[Remark 2.19]{bruenguyen}. We note in particular that the following statement partially answers Open Question 3.13 in that paper. Its implications will be discussed in the next section.

\begin{corollary} \label{weakbound1}
Assume the same hypotheses and notation as in \Cref{bressan}.

Then, there exists some $r_0 \in (0,R)$ such that \(0<r<r_0\), and \(0<t<T\),
\begin{align}
    \sup_{0<|h|<r}\bigg\{&|G(|h|)|
    \int_{B_R}|u(t,x+h)-u(t,x)|\,\dd x\bigg\}
    \nonumber\\
    &\leq L\,K_{d,g}(r)\,|Du_0|(\mathbb R^d)
    \nonumber \\&\quad+C_{d,R,\Phi}\,L\,\|u_0\|_{L^\infty(\mathbb R^d)} \bigg(\int_0^T \Vert D\boldsymbol{b}(\tau,\cdot) \Vert_{L^1(\mathbb R^d)}\,\dd \tau + \int_0^T \Vert D\boldsymbol{b}(\tau,\cdot) \Vert_{L^{\Phi_0}(\mathbb R^d)}\,\dd \tau\bigg),
    \label{estpontual}
\end{align}
where $L$ is given in \eqref{defL}, and \(K_{d,g}(r)=o(1)\) as \(r\to0_+\).
\end{corollary}

\begin{proof}
The argument is a variation on the proof of \Cref{pointwisefortransport}, and we keep the same notation.

Choose \(r_0 \in (0,a)\), where $a$ is as in \eqref{definitionG}. In particular, for \(0<s<r_0\), one has
\[
|G(s)|=-G(s).
\]
Fix \(0<r<r_0\), \(0<t<T\), \(0<|h|<r\), and \(R>0\).

Let us rewrite \eqref{estimateY} in the form
\[
G(|\boldsymbol{Y}(t,x+h)-\boldsymbol{Y}(t,x)|)
\leq
G(|h|)+w(t,x+h)+w(t,x).
\]
Fix \(0<\theta<1\), and define the corresponding good set
\[
A=A_{t,h,\theta}
\coloneqq 
\bigl\{x\in B_R:\ w(t,x+h)+w(t,x)\leq \theta |G(|h|)|\bigr\}.
\]
Since \(G(|h|)<0\), for \(x\in A\) we obtain
\[
G(|\boldsymbol{Y}(t,x+h)-\boldsymbol{Y}(t,x)|)
\leq
(1-\theta)G(|h|).
\]
Hence
\[
|\boldsymbol{Y}(t,x+h)-\boldsymbol{Y}(t,x)|
\leq
G^{-1}\bigl((1-\theta)G(|h|)\bigr)
=:\rho.
\]

Accordingly, we decompose
\begin{equation}\begin{split}
    &\int_{B_R}|u(t,x+h)-u(t,x)|\,\dd x
    \\&=
    \int_A |u(t,x+h)-u(t,x)|\,\dd x
    +
    \int_{B_R\setminus A}|u(t,x+h)-u(t,x)|\,\dd x\\
    &\eqqcolon\textnormal{(I)}+\textnormal{(II)}.\end{split}
    \label{estIeII}
\end{equation}

We begin with \(\textnormal{(II)}\). By Chebyschev's inequality and the \(L^\infty\) bound for solutions,
\begin{align}
    \textnormal{(II)}
    &\leq
    2\|u_0\|_{L^\infty}\,\mathcal L^d(B_R\setminus A)
    \nonumber\\
    &\leq
    \frac{2\|u_0\|_{L^\infty}}{\theta |G(|h|)|}
    \int_{B_R}\bigl(w(t,x+h)+w(t,x)\bigr)\,\dd x
    \nonumber\\
    &\leq
    \frac{4\|u_0\|_{L^\infty}}{\theta |G(|h|)|}
    \int_{B_{R+r}} w(t,x)\,\dd x.
    \label{IIest}
\end{align}

We now turn to \(\textnormal{(I)}\). We claim that
\begin{equation}
    \textnormal{(I)}\leq C_dL\rho\,|Du_0|(\mathbb R^d).
    \label{estI}
\end{equation}
Assuming this for the moment, \eqref{estIeII} and \eqref{IIest} yield
\begin{align*}
    |G(|h|)|
    \int_{B_R}|u(t,x+h)-u(t,x)|\,\dd x
    &\leq
    C_dL\,|Du_0|(\mathbb R^d)\,
    |G(|h|)|\,G^{-1}\bigl((1-\theta)G(|h|)\bigr)
    \\
    &\qquad
    +\frac{4\|u_0\|_{L^\infty}}{\theta}
    \int_{B_{R+r}} w(t,x)\,\dd x.
\end{align*}
Now define
\[
K_{d,g}(r)\coloneqq 
\sup_{0<s<r}
\bigl\{
|G(s)|\,G^{-1}\bigl((1-\theta)G(s)\bigr)
\bigr\}.
\]
Then \(K_{d,g}(r)\to0\) as \(r\to0_+\). Indeed, by setting \(\tau=(1-\theta)G(s)\to-\infty\) and \(\sigma=G^{-1}(\tau)\to0_+\), one finds
\[
\lim_{s\to0_+}|G(s)|\,G^{-1}\bigl((1-\theta)G(s)\bigr)
=
\frac{1}{1-\theta}\lim_{\sigma\to0_+}\sigma |G(\sigma)|
=
0,
\]
the last identity following from L'Hôpital's rule and the fact that \(\sigma/g(\sigma)\to0\) as \(\sigma\to0_+\). This gives \eqref{estpontual}.

It therefore remains to prove \eqref{estI}. We begin with the special case
\begin{equation}
    u_0=1_E
    \qquad \text{for some set } E\subset\mathbb R^d \text{ of finite perimeter.}
    \label{perimeter}
\end{equation}

Let \(\{Q\}\) be a disjoint family of open cubes of side length \(\rho\) covering \(\mathcal L^d\)-almost all of \(\mathbb R^d\). Since \Cref{transport} is linear,
\[
u(t,x)=1_E(\boldsymbol{Y}(t,x))
=\sum_Q 1_{E\cap Q}(\boldsymbol{Y}(t,x)).
\]
Let \(Q^\star\) denote the cube concentric with \(Q\) and with side length \(3\rho\). Because \(x\in A\) implies
\[
|\boldsymbol{Y}(t,x+h)-\boldsymbol{Y}(t,x)|\leq \rho,
\]
it follows that if \(\boldsymbol{Y}(t,x)\in Q\), then \(\boldsymbol{Y}(t,x+h)\in Q^\star\).

Using this observation, we obtain for almost every \(x\in A\),
\begin{align*}
    |1_E(\boldsymbol{Y}(t,x+h))-1_E(\boldsymbol{Y}(t,x))|
    &\leq
    \sum_Q
    \Big(
    1_{E\cap Q^\star}(\boldsymbol{Y}(t,x))
    1_{Q^\star\setminus E}(\boldsymbol{Y}(t,x+h))
    \\
    &\hspace{3.8em}
    +
    1_{E\cap Q^\star}(\boldsymbol{Y}(t,x+h))
    1_{Q^\star\setminus E}(\boldsymbol{Y}(t,x))
    \Big).
\end{align*}
By the monotone convergence theorem,
\begin{align*}
    \int_A |1_E(\boldsymbol{Y}(t,x+h))-1_E(\boldsymbol{Y}(t,x))|\,\dd x
    &\leq
    \sum_Q \int_A
    \Big(
    1_{E\cap Q^\star}(\boldsymbol{Y}(t,x))
    1_{Q^\star\setminus E}(\boldsymbol{Y}(t,x+h))
    \\
    &\hspace{6em}
    +
    1_{E\cap Q^\star}(\boldsymbol{Y}(t,x+h))
    1_{Q^\star\setminus E}(\boldsymbol{Y}(t,x))
    \Big)\,\dd x.
\end{align*}
Using the compressibility constant \eqref{defL}, together with the elementary bound
\[
ab\leq \min\{a,b\}
\qquad \text{for } a,b\in\{0,1\},
\]
we infer that
\begin{equation}
    \int_A |1_E(\boldsymbol{Y}(t,x+h))-1_E(\boldsymbol{Y}(t,x))|\,\dd x
    \leq
    2L\sum_Q \min\{\mathcal L^d(E\cap Q^\star),\mathcal L^d(Q^\star\setminus E)\}.
    \label{Iboundpre}
\end{equation}

We now invoke the relative isoperimetric inequality on \(Q^\star\); see \cite[Section 3.4]{ambrosio2000functions}. It yields
\[
\min\{\mathcal L^d(E\cap Q^\star),\mathcal L^d(Q^\star\setminus E)\}^{(d-1)/d}
\leq
C_d P(E,Q^\star),
\]
where \(P(E,Q^\star)\) denotes the perimeter of \(E\) in \(Q^\star\). Since also
\[
\min\{\mathcal L^d(E\cap Q^\star),\mathcal L^d(Q^\star\setminus E)\}^{1/d}
\leq
\mathcal L^d(Q^\star)^{1/d}
\leq
C_d\rho,
\]
we conclude that
\[
\min\{\mathcal L^d(E\cap Q^\star),\mathcal L^d(Q^\star\setminus E)\}
\leq
C_d\, \rho\, P(E,Q^\star).
\]
Returning to \eqref{Iboundpre}, and using that the family \(\{Q^\star\}\) has bounded overlap, together with the measure-theoretic properties of \(Q^\star\mapsto P(E,Q^\star)\) (see \cite[Proposition 3.38]{ambrosio2000functions}), we obtain
\begin{equation}
    \int_A |1_E(\boldsymbol{Y}(t,x+h))-1_E(\boldsymbol{Y}(t,x))|\,\dd x
    \leq
    C_dL\rho\,P(E, \mathbb R^d).
    \label{Iboundpre2}
\end{equation}
Because \(P(E, \mathbb R^d)=|D1_E|(\mathbb R^d)\), this proves \eqref{estI} under the assumption \eqref{perimeter}.

Finally, we derive \eqref{estI} for a general \(u_0\in \bv(\mathbb R^d)\) by means of the coarea formula. Using the kinetic representation
\[
|u_0(\boldsymbol{Y}(t,x+h))-u_0(\boldsymbol{Y}(t,x))|
=
\int_{-\infty}^{\infty}
\big|
1_{\{u_0>\lambda\}}(\boldsymbol{Y}(t,x+h))
-
1_{\{u_0>\lambda\}}(\boldsymbol{Y}(t,x))
\big|
\,\dd\lambda,
\]
we obtain, by Fubini--Tonelli, \eqref{Iboundpre2}, and the coarea formula,
\begin{align*}
    \int_A |u_0(\boldsymbol{Y}(t,x+h))-u_0(\boldsymbol{Y}(t,x))|\,\dd x
    &=
    \int_{-\infty}^{\infty}
    \int_A
    \big|
    1_{\{u_0>\lambda\}}(\boldsymbol{Y}(t,x+h))
    -
    1_{\{u_0>\lambda\}}(\boldsymbol{Y}(t,x))
    \big|
    \,\dd x\,\dd\lambda
    \\
    &\leq
    C_dL\rho
    \int_{-\infty}^{\infty}
    P(\{x\in\mathbb R^d:\ u_0(x)>\lambda\}, \mathbb R^d)\,\dd\lambda
    \\
    &=
    C_dL\rho\,|Du_0|(\mathbb R^d).
\end{align*}
This is exactly \eqref{estI}, and the proof is complete.
\end{proof}

\section{An application to mixing bounds} \label{SecMix}

The Bressan mixing conjecture \cite{Bressan1,Bressan2} is a long-standing open problem that may be phrased informally as follows: how fast, or how efficiently, can an incompressible velocity field mix two regions?

Let us state a version of the conjecture that can be naturally analyzed using the techniques developed in this paper (Bressan's original conjecture is formulated for periodic velocity fields. Adapting our results to that setting would be straightforward). Let \(U\subset\mathbb R^d\) be a bounded open set with sufficiently smooth boundary, and let
\(
U=\Omega\cup (U\setminus\Omega)
\)
be a partition such that \(\Omega\) has finite perimeter and
\(
\mathcal L^d(\Omega)=\frac12\mathcal L^d(U)>0.
\)
Let also \(\boldsymbol b:[0,T]\times\mathbb R^d\to\mathbb R^d\) be a sufficiently smooth vector field with
\begin{equation}
\supp \boldsymbol b \subset [0,T]\times \overline U,
\label{support}
\end{equation}
and denote by \(\boldsymbol X(t,0,x)\) the flow associated with \eqref{flow}. One says that the flow \(\boldsymbol X(T,0,\cdot)\) mixes \(\Omega\) and \(U\setminus\Omega\) up to scale \(\epsilon>0\) if, for every \(x\in U\), at least one third of the points in \(B_\epsilon(x)\cap U\) come from \(\Omega\), and at least one third come from \(U\setminus\Omega\). Quantitatively, this means that
\begin{equation}
\frac13
\leq
\frac{\mathcal L^d\bigl(B_\epsilon(x)\cap \boldsymbol X(T,0,\Omega)\bigr)}
{\mathcal L^d(B_\epsilon(x)\cap U)}
\leq
\frac23
\qquad \text{for every } x\in U.
\label{mixing}
\end{equation}

In this formulation, Bressan's conjecture asserts that there exists a constant \(\beta=\beta(\Omega,U)>0\) such that, for every \(T>0\), every sufficiently small \(\epsilon>0\), and every sufficiently smooth divergence-free vector field \(\boldsymbol b:[0,T]\times\mathbb R^d\to\mathbb R^d\) satisfying \eqref{support}, if \(\boldsymbol X(T,0,\cdot)\) mixes \(\Omega\) and \(U\setminus\Omega\) up to scale \(\epsilon\), then
\[
\int_0^T\int_{\mathbb R^d}|D\boldsymbol b(t,x)|\,\dd x\,\dd t
\geq
\beta |\log \epsilon|.
\]

This problem has been studied intensively over the last two decades. Crippa and De Lellis \cite{crippadelellis} proved the corresponding lower bound in the supercritical Sobolev regime, namely
\[
\int_0^T \left(\int_{\mathbb R^d}|D\boldsymbol b(t,x)|^p\,\dd x\right)^{1/p}\dd t
\geq
\beta_p |\log \epsilon|,
\]
for every \(p>1\). Later, Had\v zi\'c, Seeger, Smart, and Street \cite{street} developed a Hardy-space approach to the problem. Several other approaches have also been proposed; see \cite{seis,kiselev,leger,bruenguyen,cooperman,huysmans0,huysmans,bruecolombo} and the references therein. On the constructive side, remarkable examples of efficient incompressible mixing flows have been obtained in \cite{mazzucato0,mazzucato1,elgindi1,elgindi2}.

Since Bressan's conjecture is formulated in terms of the \(L^1((0,T);W^{1,1}(\mathbb R^d))\)-norm of the velocity field, it is natural to revisit the problem from a genuinely endpoint \(L^1\) perspective. This point of view was recently pursued by Cooperman \cite{cooperman} for a special class of two-dimensional velocity fields, and by Huysmans and Said \cite{huysmans} for autonomous divergence-free \(\bv\) vector fields. All the remaining works mentioned above assume regularity slightly stronger than \eqref{regularity}, even if qualitatively, as in the recent proof by Bru\`e, Colombo, and Johansson \cite{bruecolombo} of an asymptotic version of Bressan's conjecture.

We now show that, assuming only the DiPerna--Lions regularity hypothesis \eqref{regularity2},
one can easily derive lower bounds on the admissible mixing scale \(\epsilon\) by means of the regularity results established in the previous section.

\begin{theorem}\label{mixingthm}
Let \(U\subset\mathbb R^d\) be a bounded open set of finite perimeter, and let
\(
U=\Omega\cup (U\setminus\Omega)
\)
be a partition such that \(\Omega\) has finite perimeter and
\(
\mathcal L^d(\Omega)=\frac12\mathcal L^d(U)>0.
\)
Let \(\boldsymbol b\) be divergence-free and satisfy \eqref{regularity2} and \eqref{support}. Let \(\Phi_0\) and $\Phi$ be as in \Cref{regularityforPhi}. As usual, set
\(
\Psi(s)=s^{-1}\Phi(s).
\)

Then there exist \(\epsilon_0=\epsilon_0(\Phi,U,\Omega,d)>0\) and \(C=C(\Phi,\Omega,d)>0\) with the following property: if the regular Lagrangian flow \(\boldsymbol X(T,0,\cdot)\) of \(\boldsymbol b\) mixes \(\Omega\) and \(U\setminus\Omega\) up to a scale \(\epsilon\in(0,\epsilon_0)\), then
\begin{equation}
    \Psi(\epsilon^{-d})
\leq
C\left(
\int_0^T\Vert D\boldsymbol{b} (t,\cdot)\Vert_{L^1(\mathbb R^d)}\,\dd\tau + \int_0^T \Vert D\boldsymbol{b} (t,\cdot)\Vert_{L^{\Phi_0}(\mathbb R^d)}\,\dd\tau
\right). \label{mixingineq}
\end{equation}
\end{theorem}

\begin{remark}
\textnormal{Some remarks are in order.}
\begin{itemize}
\item \textnormal{Although the statement of \Cref{mixingthm} is technical, it may be more easily understood in conjunction with \Cref{onPhiandPhi0}. Indeed, assume that $D \boldsymbol{b} \in L^1((0,T); L^{\Phi_0}(\mathbb R^d))$ where $\Psi_0(z) \coloneqq  z^{-1}\Phi_0(z)$ is unbounded and eventually $C^1$, with $z\Psi_0'(z)$ being a eventually positive, eventually nonincreasing and normalized slowly varying function at $\infty$. In this case, we may take $\Psi(z) = \Psi_0(z)$, and, diminishing $\epsilon_0$ if necessary, \Cref{mixingineq} can be put as
$$    \Psi_0(\epsilon^{-d})
\leq
C\left(
\int_0^T\Vert D\boldsymbol{b} (t,\cdot)\Vert_{L^1(\mathbb R^d)}\,\dd\tau + \int_0^T \Vert D\boldsymbol{b} (t,\cdot)\Vert_{L^{z\Psi_0(z)}(\mathbb R^d)}\,\dd\tau
\right).$$}

\textnormal{From this, we may employ the examples of \Cref{exampleprop} and derive several new mixing bounds, as for instance,
\begin{align*}
 \log (\epsilon^{-d})^b &\leq C\left(
\int_0^T\Vert D\boldsymbol{b} (t,\cdot)\Vert_{L^1(\mathbb R^d)}\,\dd\tau + \int_0^T \Vert D\boldsymbol{b} (t,\cdot)\Vert_{L \log^b L(\mathbb R^d)}\,\dd\tau \right) \text{ for $b \in (0,1]$}, \\
 \frac{\log (\epsilon^{-d})}{\log^{(n)}(\epsilon^{-d})} &\leq C\left(
\int_0^T\Vert D\boldsymbol{b} (t,\cdot)\Vert_{L^1(\mathbb R^d)}\,\dd\tau + \int_0^T \Vert D\boldsymbol{b} (t,\cdot)\Vert_{L \frac{\log}{\log^{(n)}} L(\mathbb R^d)}\,\dd\tau \right) \text{ for $n\geq2$}, \\
    \log^{(n)}(\epsilon^{-d})
&\leq
C\left(
\int_0^T\Vert D\boldsymbol{b} (t,\cdot)\Vert_{L^1(\mathbb R^d)}\,\dd\tau + \int_0^T \Vert D\boldsymbol{b} (t,\cdot)\Vert_{L \log^{(n)} L(\mathbb R^d)}\,\dd\tau \right) \text{ for $n\geq2$},
\end{align*}
among others. Notice that taking $\Phi_0(z) = z \log_+ z$ (which is the case $b=1$ above), we recover the classical logarithmic lower bound of Crippa and De Lellis \cite{crippadelellis}}.

\item \textnormal{The Hardy-space approach of Had\v zi\'c, Seeger, Smart, and Street \cite{street} suggests that it may be fruitful to investigate weighted grand maximal functions and corresponding weighted Hardy spaces in the spirit of the present paper; see also \cite{bouchutcrippa}.}

\item \textnormal{At present, we do not know whether the bound obtained here is sharp; see the next section.}
\end{itemize}
\end{remark}

\begin{proof}
The argument is based on the following classical observation. Define
\[
u_0(x)=1_\Omega(x)-1_{U\setminus\Omega}(x)\in \bv(\mathbb R^d),
\]
and let \(u(t,x)\) be the corresponding solution to the transport equation \eqref{transport} with initial datum \(u(0,x)=u_0(x)\). Since \(u_0\) vanishes outside \(U\), and \(\boldsymbol b(t,\cdot)\) is supported in \(\overline U\), the transported scalar \(u(t,\cdot)\) also vanishes almost everywhere outside \(U\). Moreover, \eqref{mixing} immediately implies that
\begin{equation}
\left|
\frac{1}{\mathcal L^d(B_\epsilon(x))}
\int_{B_\epsilon(x)} u(T,y)\,\dd y
\right|
\leq \frac13
\qquad \text{for every } x\in U.
\label{boundmixing}
\end{equation}
We shall combine \eqref{boundmixing} with the regularity estimates of the previous section in order to derive lower bounds on \(\epsilon\).

Let \(r_0>0\) be the same as in \Cref{weakbound1}, and recall that
\[
G(r)=\frac{1}{d}\bigl(\Psi(r^{-d})-\Psi(a^{-d})\bigr)
\]
for some \(a>0\). We may assume that $r_0 < a$. Assume also that \(0<\epsilon<r_0\). By \eqref{boundmixing}, for every \(x\in\{z:u(T,z)=1\}\) one has
\[
\frac23
\leq
\frac{1}{\mathcal L^d(B_\epsilon)}
\int_{B_\epsilon}
|u(T,x+h)-u(T,x)|\,\dd h.
\]
Integrating this inequality over \(x\in\{z:u(T,z)=1\}\), and recalling that
\(
\mathcal L^d(\{z:u(T,z)=1\})=\mathcal L^d(\Omega),
\)
we obtain
\begin{align*}
\frac23\mathcal L^d(\Omega)
&\leq
\frac{1}{\mathcal L^d(B_\epsilon)}
\int_{\{u(T,\cdot)=1\}}\int_{B_\epsilon}
|u(T,x+h)-u(T,x)|\,\dd h\,\dd x\\
&\leq
\frac{1}{\mathcal L^d(B_\epsilon)}
\int_{\mathbb R^d}\int_{B_\epsilon}
|u(T,x+h)-u(T,x)|\,\dd h\,\dd x\\
&\leq
\left(
\sup_{0<|h|<\epsilon}
\left\{
|G(|h|)|
\int_{\mathbb R^d}|u(T,x+h)-u(T,x)|\,\dd x
\right\}
\right)
\frac{1}{\mathcal L^d(B_\epsilon)}
\int_{B_\epsilon}\frac{\dd h}{|G(|h|)|}.
\end{align*}
Since \(|G(r)|\) is nondecreasing for \(r\in (0,r_0)\), it follows that
\[
\frac{1}{\mathcal L^d(B_\epsilon)}
\int_{B_\epsilon}\frac{\dd h}{|G(|h|)|}
\leq
\frac{1}{|G(\epsilon)|}.
\]
Therefore,
\[
\frac23\mathcal L^d(\Omega)\,|G(\epsilon)|
\leq
\sup_{0<|h|<\epsilon}
\left\{
|G(|h|)|
\int_{\mathbb R^d}|u(T,x+h)-u(T,x)|\,\dd x
\right\}.
\]

Applying \Cref{weakbound1}, we deduce that
\begin{align*}
|G(\epsilon)|
\leq
C(\Omega,d)\bigg(
o(1)\big(P(\Omega)+P(U\setminus\Omega)\big)
&+
\int_0^T \Vert D\boldsymbol{b}(t,\cdot) \Vert_{L^1(\mathbb R^d)}\,\dd\tau + \int_0^T \Vert D\boldsymbol{b} (t,\cdot)\Vert_{L^{\Phi_0}(\mathbb R^d)}\,\dd\tau \bigg)
\end{align*}
as \(\epsilon\to0_+\). Since
\[
|G(\epsilon)|=\frac{1}{d}\bigl(\Psi(\epsilon^{-d})-\Psi(a^{-d})\bigr),
\]
the conclusion follows by choosing \(0<\epsilon_0<r_0\) so small that the \(o(1)\)-term, and the quantity \(\Psi(a^{-d})\) can be absorbed into the left-hand side.
\end{proof}

We close this section with a closely related proposition concerning functional mixing, more specifically the decay of \(\|u(t,\cdot)\|_{\dot H^{-1}(\mathbb R^d)}\), where \(u(t,x)\) is a solution of \eqref{transport}. For this, we shall need the following lemma, which explains why the Gagliardo seminorm introduced in \Cref{bressan} gives rise to a logarithmic Sobolev-type space.

\begin{lemma}\label{logsobolev}
Assume that \(d\geq 2\). Let \(g\) be given by \eqref{definitiong} for some $\Phi$ satisfying conditions (1)--(3) of \Cref{regularityforPhi}. Let also \(f\in L^2(\mathbb R^d)\) be such that, for some \(\delta>0\),
\begin{equation}
\int_{B_\delta}\int_{\mathbb R^d}\frac{|f(x+h)-f(x)|^2}{|h|^d g(|h|)}\,\dd x\,\dd h<\infty,
\label{energy}
\end{equation}
where $\Psi(s) = s^{-1} \Phi(s)$ and $g(r)$ is as in \eqref{definitiong}.

As usual, we denote by
\[
\widehat F(\xi)=(2\pi)^{-d/2}\int_{\mathbb R^d}F(x)e^{-ix\cdot\xi}\,\dd x
\]
the Fourier transform on \(\mathbb R^d\). Then, there exists a constant \(C_{d,g}>0\) such that
\begin{equation}
\begin{dcases}
\int_{\mathbb R^d}\Psi(|\xi|^d)\,|\widehat f(\xi)|^2\,\dd \xi
\leq
C_{d,g}\left(
\int_{B_\delta}\int_{\mathbb R^d}\frac{|f(x+h)-f(x)|^2}{|h|^d g(|h|)}\,\dd x\,\dd h
+
\|f\|_{L^2(\mathbb R^d)}^2
\right),
\\
\int_{B_\delta}\int_{\mathbb R^d}\frac{|f(x+h)-f(x)|^2}{|h|^d g(|h|)}\,\dd x\,\dd h
\leq
C_{d,g}\left(
\int_{\mathbb R^d}\Psi(|\xi|^d)\,|\widehat f(\xi)|^2\,\dd \xi
+
\|f\|_{L^2(\mathbb R^d)}^2
\right).
\end{dcases}
\label{preprebound}
\end{equation}
\end{lemma}

\begin{proof}
The proof is based on Plancherel's theorem. Since
\[
\int_{\mathbb R^d}|f(x+h)-f(x)|^2\,\dd x
=
\int_{\mathbb R^d}|e^{i\xi\cdot h}-1|^2\,|\widehat f(\xi)|^2\,\dd \xi,
\]
we may rewrite the integral in \eqref{energy} as
\begin{equation}
\int_{\mathbb R^d} m_g(\xi)\,|\widehat f(\xi)|^2\,\dd \xi,
\label{energy2}
\end{equation}
where
\[
m_g(\xi)=2\int_{B_\delta}\frac{1-\cos(\xi\cdot h)}{g(|h|)\,|h|^d}\,\dd h.
\]
Passing to polar coordinates \(h=re\), with \(r>0\) and \(e\in \mathbb S^{d-1}\), we obtain
\[
m_g(\xi)=2\int_0^\delta \frac{1}{r g(r)} \left(\int_{\mathbb S^{d-1}}(1-\cos(r\xi\cdot e))\,\dd e\right)\dd r.
\]
Writing \(\xi=|\xi|\omega\), we may therefore express \(m_g\) as
\[
m_g(\xi)=2\int_0^\delta \frac{A(r|\xi|)}{r g(r)}\,\dd r,
\]
where
\[
A(s)=\int_{\mathbb S^{d-1}}(1-\cos(s\,\omega\cdot e))\,\dd e,
\]
and \(A\) does not depend on the particular choice of \(e\in \mathbb S^{d-1}\). Since \(d\geq2\), the function \(A\) satisfies
\[
c_d^{-1}\min\{1,s^2\}\leq A(s)\leq c_d\min\{1,s^2\}
\qquad \text{for all } s\geq0
\]
for some constant \(c_d>0\).

Hence, for \(|\xi|>\delta^{-1}\), we obtain
\[
c_d^{-1}\left(
|\xi|^2\int_0^{1/|\xi|}\frac{r}{g(r)}\,\dd r
+
\int_{1/|\xi|}^{\delta}\frac{\dd r}{r g(r)}
\right)
\leq
m_g(\xi)
\leq
c_d\left(
|\xi|^2\int_0^{1/|\xi|}\frac{r}{g(r)}\,\dd r
+
\int_{1/|\xi|}^{\delta}\frac{\dd r}{r g(r)}
\right).
\]
We now estimate the two terms on the right-hand side. By Karamata's theorem,
\[
|\xi|^2\int_0^{1/|\xi|}\frac{r}{g(r)}\,\dd r
\sim
\frac{1}{2\,g(1/|\xi|)}
\qquad \text{as } |\xi|\to\infty.
\]
On the other hand, by \eqref{definitionG},
\[
\int_{1/|\xi|}^{\delta}\frac{\dd r}{r g(r)}
=
G(\delta)-G(1/|\xi|)
=
\frac{1}{d}\bigl(\Psi(|\xi|^d)-\Psi(\delta^{-d})\bigr).
\]
Since \(\Psi\) is increasing and unbounded, we may assume that
\[
\frac{1}{2d}\Psi(|\xi|^d)
\leq
G(\delta)-G(1/|\xi|)
\leq
\frac{1}{d}\Psi(|\xi|^d)
\qquad \text{for all sufficiently large } |\xi|.
\]
Consequently, there exists \(C_{d,g}>0\) such that
\[
C_{d,g}^{-1}\Psi(|\xi|^d)
\leq
m_g(\xi)
\leq
C_{d,g}\Psi(|\xi|^d)
\qquad \text{for all sufficiently large } |\xi|.
\]
Since \(\Psi\) is bounded on bounded intervals, we may increase \(C_{d,g}\) if necessary and obtain
\[
C_{d,g}^{-1}\Psi(|\xi|^d)\leq m_g(\xi)+1
\qquad\text{and}\qquad
m_g(\xi)\leq C_{d,g}\bigl(\Psi(|\xi|^d)+1\bigr)
\]
for all \(\xi\in\mathbb R^d\). Returning to \eqref{energy} and \eqref{energy2}, we conclude \eqref{preprebound}.
\end{proof}

As a consequence, we obtain the following interpolation inequality.

\begin{lemma}\label{interpolationbound}
Keep the same assumptions as in \Cref{logsobolev}, and assume moreover that \(f\neq0\).

Then there exist constants \(c>0\) and \(C_{d,g}>0\) such that
\begin{equation}
\|f\|_{L^2(\mathbb R^d)}^2\, \Psi\!\left( c\,\frac{\|f\|_{L^2(\mathbb R^d)}^d}{\|f\|_{\dot H^{-1}(\mathbb R^d)}^d} \right)
\leq
C_{d,g}\left(
\int_{B_\delta}\int_{\mathbb R^d}\frac{|f(x+h)-f(x)|^2}{|h|^d g(|h|)}\,\dd x\,\dd h
+
\|f\|_{L^2(\mathbb R^d)}^2
\right).
\label{boundbound}
\end{equation}
\end{lemma}

\begin{proof}
This is a simple consequence of \Cref{logsobolev}. Since \(\Psi\) is increasing,
\begin{equation}
\int_{\mathbb R^d}\Psi(|\xi|^d)\,|\widehat f(\xi)|^2\,\dd \xi
\geq
\Psi(\eta^d)\int_{|\xi|>\eta}|\widehat f(\xi)|^2\,\dd \xi
\label{prebound}
\end{equation}
for any $\eta>0$. Now fix \(0<\theta<1\) and choose
\[
\eta=\theta\frac{\|f\|_{L^2(\mathbb R^d)}}{\|f\|_{\dot H^{-1}(\mathbb R^d)}}.
\]
Then
\[
\int_{|\xi|<\eta}|\widehat f(\xi)|^2\,\dd \xi
\leq
\eta^2\|f\|_{\dot H^{-1}(\mathbb R^d)}^2
\leq
\theta^2\|f\|_{L^2(\mathbb R^d)}^2,
\]
and therefore, by Plancherel,
\begin{equation}
\int_{|\xi|>\eta}|\widehat f(\xi)|^2\,\dd \xi
\geq
(1-\theta^2)\|f\|_{L^2(\mathbb R^d)}^2.
\label{prebound1}
\end{equation}
Combining \eqref{preprebound}, \eqref{prebound}, and \eqref{prebound1}, we obtain
\[
(1-\theta^2)\|f\|_{L^2(\mathbb R^d)}^2
\Psi\!\left(
\theta^d\frac{\|f\|_{L^2(\mathbb R^d)}^d}{\|f\|_{\dot H^{-1}(\mathbb R^d)}^d}
\right)
\leq
C_{d,g}\left(
\int_{B_\delta}\int_{\mathbb R^d}\frac{|f(x+h)-f(x)|^2}{|h|^d g(|h|)}\,\dd x\,\dd h
+
\|f\|_{L^2(\mathbb R^d)}^2
\right).
\]
This is exactly \eqref{boundbound}, upon setting \(c=\theta^d\) and absorbing the factor \(1-\theta^2\) into the constant.
\end{proof}

As a consequence of \Cref{bressan} and \Cref{interpolationbound}, we immediately obtain the following functional mixing estimate.

\begin{theorem} \label{mixingthm2}
Assume that $d \geq 2$. Let \(\boldsymbol b\) be divergence-free and satisfy \eqref{regularity2} and \eqref{support} for some open set $U \subset\subset \mathbb R^d$. Let \(u(t,x)\) be a solution to \eqref{transport} with initial datum satisfying \eqref{cauchy}. Finally, let \(\Phi_0\) and $\Phi$ be as in \Cref{regularityforPhi}. As usual, set
\(
\Psi(s)=s^{-1}\Phi(s).
\)

Then, there exist constants \(c>0\) and \(C_{d,\Phi, \Vert u_0 \Vert_{L^\infty(\mathbb R^d)},U}>0\) such that for every \(0<t<T\),
\begin{align*}
\|u_0\|_{L^2(\mathbb R^d)}^2\,
\Psi\!\Bigg(
c\,\frac{\|u_0\|_{L^2(\mathbb R^d)}^d}{\|u(t,\cdot)\|_{\dot H^{-1}(\mathbb R^d)}^d}
\Bigg)
&\leq
C_{d,\Phi, \Vert u_0 \Vert_{L^\infty(\mathbb R^d)}, U}\bigg(
\int_0^T\Vert D\boldsymbol{b}(t,\cdot) \Vert_{L^1(\mathbb R^d)}\,\dd\tau \nonumber \\&\quad+ \int_0^T \Vert D\boldsymbol{b}(t,\cdot) \Vert_{L^{\Phi_0}(\mathbb R^d)}\,\dd\tau
+
\|u_0\|_{\bv(\mathbb R^d)} + \|u_0\|_{L^2(\mathbb R^d)}^2
\bigg).
\end{align*}
\end{theorem}

\section{Concluding remarks and open questions} \label{SecFinal}

We conclude by emphasizing the main contribution of this work: the Osgood bound \eqref{estimatenew}, which is based on a new class of maximal operators and is sharp in many respects. Beyond the regularity theory of regular Lagrangian flows, this estimate also yields new information on the behavior of solutions to the transport equation \eqref{transport} and on quantitative mixing problems.

At the same time, the results obtained here point toward a natural refinement of our theory, one that would further the properties of the transport equation \eqref{transport}. This question is also closely connected with the recent paper of Bru\`e, Colombo, and Johansson \cite{bruecolombo}.

\begin{question}\label{conjecture1}
    Let \(\boldsymbol b \in L^1((0,T);W^{1,1}(\mathbb R^d))\) satisfy \eqref{growth} and \eqref{compressibility}, and let \(\boldsymbol X(t,s,x)\) denote the associated regular Lagrangian flow of \eqref{flow}. Given a ball \(B_R\subset\mathbb R^d\), does there exist a function \(K\in L^1(B_R)\) such that
    \begin{equation}
           |\boldsymbol X(t,0,x)-\boldsymbol X(t,0,y)|
    \leq
    \exp\bigl\{K(x)+K(y)\bigr\}|x-y| \label{estimatefinal}
    \end{equation}
    for all \(t\in[0,T]\) and almost every \(x,y\in B_R\)? If so, how should \(\|K\|_{L^1(B_R)}\) depend on an appropriate norm of \(D\boldsymbol b\)?
\end{question}

The bound in \eqref{estimatefinal} differs substantially from \eqref{estimate}. Indeed, it does not ask for a comparison between
\[
|\boldsymbol X(t,0,x)-\boldsymbol X(t,0,y)|
\quad \text{and} \quad
|\boldsymbol X(s,0,x)-\boldsymbol X(s,0,y)|,
\]
nor does it require the control function to be represented as an integral in time. 
Therefore, there is no contradiction with \Cref{counter}.

If the solution to \Cref{conjecture1} is affirmative, then the regularity theory developed in \Cref{SecReg} would hold with \(g=1\). In particular, one would obtain a local \(H^{\log}\)-regularity and possibly derive Bressan's conjecture; see the discussion in \cite{bruecolombo}.

On the other hand, if the answer to \Cref{conjecture1} is negative, then one should expect a genuinely Osgood-type behavior for a general regular Lagrangian flow \(\boldsymbol X(t,0,x)\). Since Osgood flows are considerably ``wilder'' than Cauchy--Lipschitz flows, informal heuristics suggest that such flows might mix sets more efficiently. In that case, it becomes conceivable that the bounds obtained in \Cref{mixingthm} are closer to the correct endpoint behavior, and possibly even sharp.

We also record a second formulation of this problem, which in fact served as one of the original motivations for the present paper.

\begin{question}\label{conjecture2}
    Let \(\boldsymbol b \in L^1((0,T);W^{1,1}(\mathbb R^d))\) satisfy \eqref{growth} and \eqref{compressibility}, and let \(\boldsymbol X(t,s,x)\) denote the associated regular Lagrangian flow of \eqref{flow}. Given a ball \(B_R\subset\mathbb R^d\) and \(\epsilon>0\), does there exist a measurable set \(\Omega_{\epsilon,R}\subset B_R\) such that
    \(
    \mathcal L^d(B_R\setminus \Omega_{\epsilon,R})<\epsilon
    \)
    and such that the restriction
    \(
    \boldsymbol X(t,0,\cdot)\big|_{\Omega_{\epsilon,R}}
    \)
    is Lipschitz for every \(t\in[0,T]\)? If so, how may
\(\operatorname{Lip}\bigl(\boldsymbol X(t,0,\cdot)\big|_{\Omega_{\epsilon,R}}\bigr)
    \)
    grow as \(\epsilon\to0\)?
\end{question}

\begin{remark}
    \textnormal{The problem concerning the growth of
\(\operatorname{Lip}\bigl(\boldsymbol X(t,0,\cdot)\big|_{\Omega_{\epsilon,R}}\bigr)
    \)
    is motivated by the paper of Bonicatto and Marconi \cite{Bonicatto02122021}.}
\end{remark}

Finally, we close the paper by solving \Cref{conjecture1} and \Cref{conjecture2} in the one-dimensional setting \(d=1\).

\begin{proposition} \label{proof1d}
    Assume that \(d=1\), and let \(\boldsymbol b\) satisfy \eqref{growth} and \eqref{compressibility}. Assume moreover that, for almost every \(t\in(0,T)\), the map
    \(
    x\mapsto \boldsymbol b_t(x)\coloneqq \boldsymbol b(t,x)
    \)
    belongs to \(\bv_{\loc}(\mathbb R)\), with
    \[
    \int_0^T |D\boldsymbol b_t|(I)\,\dd t<\infty
    \qquad \text{for every bounded interval } I\subset \mathbb R.
    \]
    Let \(\boldsymbol X(t,s,x)\) be the regular Lagrangian flow associated with \eqref{flow} (see \cite{ambrosio}), and fix a compact interval \(J\subset \mathbb R\).

    Then there exists a nonnegative function \(K\in L^1(J)\) such that
    \begin{equation}
        |\boldsymbol X(t,0,x)-\boldsymbol X(t,0,y)|
        \leq
        \exp\bigl\{K(x)+K(y)\bigr\}|x-y|
        \label{superlipschitz}
    \end{equation}
    for almost every \(x,y\in J\) and every \(t\in[0,T]\). Furthermore, there exists a compact interval \(J_0\supset J\) such that
    \begin{equation}
        \int_J K(x)\,\dd x
        \leq
        C\left(
            \mathcal L^1(J)+\int_0^T |D\boldsymbol b_s|(J_0)\,\dd s
        \right),
        \label{boundK}
    \end{equation}
    where \(C>0\) is a universal constant.
\end{proposition}

\begin{proof}
    We first recall a convention that will be used throughout the proof. For almost every \(s\in(0,T)\), we choose the right-continuous representative of the one-dimensional \(\bv_{\loc}\) function
    \(
        x\mapsto \boldsymbol b_s(x)\coloneqq \boldsymbol b(s,x).
    \)
    With this convention, if \(u<v\), then
    \begin{equation}
              \boldsymbol b_s(v)-\boldsymbol b_s(u)
        =
        D\boldsymbol b_s((u,v])
        \leq
        (D\boldsymbol b_s)_+((u,v]).  \label{fundthm}
    \end{equation}
   Furthermore, let us denote by $\Omega \subset \mathbb R$ a set of full measure such that \[\boldsymbol{X}(t,0,x) = x + \int_0^t \boldsymbol{b}_s(\boldsymbol{X}(s,0,x))\,\dd s\] for all $x \in \Omega$. In order to simplify the notation, we also write $\boldsymbol{X}_s(x) = \boldsymbol{X}(s,0,x)$.

 A key fact in one space dimension is that the regular Lagrangian flow is monotone in the spatial variable. More precisely, we have that
 \begin{equation}
             x\leq y
        \quad\Longrightarrow\quad
        \boldsymbol X_s(x) \leq \boldsymbol X_s(y)\label{monotonicity}
 \end{equation}
    for every \(s\in[0,T]\) and $x , y \in \Omega$ (diminishing $\Omega$ if necessary). This follows, for instance, by approximation with smooth vector fields and passage to the limit. 

    Finally, let us choose an open interval \((a,b)\supset J\), with \(a\) and \(b\) belonging to $\Omega$. If we set
    \[
        J_0\coloneqq 
        \left[
            \inf_{0\leq s\leq T}\boldsymbol X_s(a),
            \sup_{0\leq s\leq T}\boldsymbol X_s(b)
        \right],
    \]
    we would then have the invariance
    \(
        \boldsymbol X_s(J)\subset J_0 \text{ for every }s\in[0,T] \text{ and $x \in \Omega \cap J$}.
    \)

    As a result, if \(x<y\) are in $\Omega \cap J$, we get
    \[
    \begin{aligned}
        \boldsymbol X_t(y)-\boldsymbol X_t(x)
        &=
        y-x
        +
        \int_0^t
        \Big(
            \boldsymbol b_s(\boldsymbol X_s(y))
            -
            \boldsymbol b_s(\boldsymbol X_s(x))
        \Big)\,\dd s,
    \end{aligned}
    \]
    and, for
    \(
        \boldsymbol X_s(x)\leq \boldsymbol X_s(y),
    \)
    \eqref{fundthm} gives
    \begin{equation}
               \boldsymbol X_t(y)-\boldsymbol X_t(x)
        \leq
        y-x
        +
        \int_0^t
        (D\boldsymbol b_s)_+
        \bigl(
            (\boldsymbol X_s(x),\boldsymbol X_s(y)]
        \bigr)
        \,\dd s. \label{fundthm2} 
    \end{equation}
 
    We now define a positive measure to bound the right-hand side of \eqref{fundthm2} for $x, y \in J$. First, due to \eqref{monotonicity}, we may define for $0\leq s \leq T$ and $x \in \R$
    $$\boldsymbol{X}_s^+(\xi) \coloneqq \lim_{\Omega \ni \eta \to \xi_+} \boldsymbol{X}_s(\eta).$$
    As a result, we see that, for $r \in (a,b)$,
    \[
        H(r)\coloneqq 
        \int_0^T
        (D\boldsymbol b_s)_+
        \bigl(
            (\boldsymbol X_s^+(a),\boldsymbol X_s^+(r)]
        \bigr)
        \,\dd s
    \]
    is nondecreasing and right-continuous. Hence, it defines a positive Borel-Stieltjes measure \(\mu_T\) on \((a,b)\). Furthermore, it is not difficult to see that \eqref{fundthm2} implies
    \begin{equation}
                |\boldsymbol X_t(x)-\boldsymbol X_t(y)|
\leq |x-y|+\mu_T([x,y])\label{fundthm3}
    \end{equation}
    for $x<y$ in $\Omega \cap J$ and $t \in [0,T]$.

    Next, let us consider the local uncentered maximal function
    \[
        M_J\mu_T(x)\coloneqq 
        \sup_{\substack{x\in I, \operatorname{Int}(I) \neq \varnothing\\ I\subset J\text{ interval}}}
        \frac{\mu_T(I)}{\mathcal L^1(I)}.
    \]
    Since \([x,y]\subset J\) contains both \(x\) and \(y\), we have
    \[
        \mu_T([x,y])
        \leq
        \frac12
        \bigl(
            M_J\mu_T(x)+M_J\mu_T(y)
        \bigr)|x-y|.
    \]
    Combining this with \eqref{fundthm3}, we obtain
    \[
        |\boldsymbol X(t,0,y)-\boldsymbol X(t,0,x)|
        \leq
        \left(
            1+
            \frac12 M_J\mu_T(x)
            +
            \frac12 M_J\mu_T(y)
        \right)|x-y|.
    \]
    Letting
    \[
    K(x)\coloneqq \log\bigl(1+M_J\mu_T(x)\bigr),
    \]
    and using
    \[
        1+\frac12 A+\frac12 B
        \leq
        (1+A)(1+B),
        \qquad A,B\geq0,
    \]
    we derive the desired bound \eqref{superlipschitz}.

    It remains to establish \eqref{boundK}. First, it is clear that
    \begin{equation}
                \mu_T(J)
        \leq
        \int_0^T
        (D\boldsymbol b_s)_+(J_0)\,\dd s
        \leq
        \int_0^T
        |D\boldsymbol b_s|(J_0)\,\dd s. \label{fundthm4} 
    \end{equation}
    What is more, by the Hardy--Littlewood theorem,
    \[
        \mathcal L^1
        \bigl(
            \{x\in J:\ M_J\mu_T(x)>\lambda\}
        \bigr)
        \leq
        \min\left\{
            \mathcal L^1(J),
            \frac{C\mu_T(J)}{\lambda}
        \right\}
    \]
    for every \(\lambda>0\), with \(C>0\) universal. Consequently, the Cavalieri principle provides
    \[
    \begin{aligned}
        \int_J K(x)\,\dd x
        &=
        \int_0^\infty
        \mathcal L^1
        \bigl(
            \{x\in J:\ \log(1+M_J\mu_T(x))>r\}
        \bigr)
        \,\dd r
        \\
        &=
        \int_0^\infty
        \frac{
        \mathcal L^1
        \bigl(
            \{x\in J:\ M_J\mu_T(x)>\lambda\}
        \bigr)
        }{1+\lambda}\,\dd\lambda\\
        &\leq
        \mathcal L^1(J)
        \int_0^1\frac{\dd\lambda}{1+\lambda}
        +
        C\mu_T(J)
        \int_1^\infty
        \frac{\dd\lambda}{\lambda(1+\lambda)}
        \\
        &\leq
        C\bigl(\mathcal L^1(J)+\mu_T(J)\bigr).
    \end{aligned}
    \]
    Combining this with \eqref{fundthm4}, we conclude  \eqref{boundK}. The proof is complete.
\end{proof}

It is worth stressing that, although the proof of \Cref{proof1d} still uses maximal functions, they appear only after the time variable has been integrated out. In particular, one does not need to integrate differential inequalities involving maximal functions along the flow. Thus, while the argument bears a superficial resemblance to the methods used throughout the paper, its underlying mechanism is somewhat different. At the same time, it is worth mentioning that \Cref{proof1d} strongly suggests that any counterexample to \Cref{conjecture1}, should one exist, must be genuinely multidimensional.

        \section*{Conflict of interest}

The authors state that there are no conflicts of interest regarding this work.

        \section*{Availability of data and material}

Data sharing is not applicable to this article since no datasets were generated or examined in the course of this study.

\bibliographystyle{plain}
\bibliography{bib}
\end{document}